\documentclass[a4paper,11pt]{amsart}
\usepackage{amssymb,amscd,amsxtra,graphicx,psfrag,xypic,array}

\usepackage[pdftex]{hyperref}
\hypersetup{%
bookmarksnumbered=true,%
colorlinks=true,%
setpagesize=false,%
pdftitle={},%
pdfauthor={Takuzo Okada}}

\theoremstyle{plain}
\newtheorem{Thm}{Theorem}[section]
\newtheorem{Lem}[Thm]{Lemma}
\newtheorem{Cor}[Thm]{Corollary}
\newtheorem{Prop}[Thm]{Proposition}

\theoremstyle{definition}
\newtheorem{Def}[Thm]{Definition}
\newtheorem{Def-Lem}[Thm]{Definition-Lemma}
\newtheorem{Cond}[Thm]{Condition}
\newtheorem{Rem}[Thm]{Remark}

\newtheorem*{Ack}{Acknowledgments}
\newtheorem{Ex}[Thm]{Example}
\newtheorem{Assumption}[Thm]{Assumption}

\newcommand{\prt}{\partial}

\newcommand{\rank}{\operatorname{rank}}

\newcommand{\Pic}{\operatorname{Pic}}

\newcommand{\Bs}{\operatorname{Bs}}

\newcommand{\Exc}{\operatorname{Exc}}
\newcommand{\mult}{\operatorname{mult}}

\newcommand{\lcm}{\operatorname{lcm}}

\newcommand{\wt}{\operatorname{wt}}

\newcommand{\bsum}{\sum\nolimits}

\newcommand{\mbA}{\mathbb{A}}

\newcommand{\mbC}{\mathbb{C}}

\newcommand{\mbN}{\mathbb{N}}
\newcommand{\mbP}{\mathbb{P}}
\newcommand{\mbQ}{\mathbb{Q}}

\newcommand{\mbZ}{\mathbb{Z}}

\newcommand{\mcG}{\mathcal{G}}
\newcommand{\mcH}{\mathcal{H}}
\newcommand{\mcI}{\mathcal{I}}

\newcommand{\mcL}{\mathcal{L}}
\newcommand{\mcM}{\mathcal{M}}

\newcommand{\mcO}{\mathcal{O}}

\newcommand{\msp}{\mathsf{p}}
\newcommand{\msq}{\mathsf{q}}

\newcommand{\inj}{\hookrightarrow}
\newcommand{\ratmap}{\dashrightarrow}

\title[$\mathbb{Q}$-Fano $3$-fold weighted complete intersections, II]{Birational Mori fiber structures of $\mathbb{Q}$-Fano $3$-fold weighted complete intersections, II}
\author[Takuzo Okada]{Takuzo Okada}
\address{Department of Mathematics, Faculty of Science and Engineering\endgraf
Saga University, Saga 845-8502 Japan}
\email{okada@cc.saga-u.ac.jp}
\subjclass[2000]{14J45 \and 14E08 \and 14J30}
\date{}

\begin{document}

\begin{abstract}
In \cite{Okada}, we proved that, among $85$ families of $\mathbb{Q}$-Fano threefold weighted complete intersections of codimension $2$, $19$ families consist of birationally rigid varieties and the remaining families consists of birationally non-rigid varieties.
The aim of this paper is to study systematically the remaining families and prove that every quasismooth member of $14$ families is birational to another $\mathbb{Q}$-Fano threefold but not birational to any other Mori fiber space.
\end{abstract}

\maketitle

\section{Introduction} \label{sec:intro}

In the paper \cite{Okada}, we studied birational geometry of anticanonically embedded $\mbQ$-Fano $3$-fold weighted complete intersections of codimension $2$ ($\mbQ$-Fano WCIs of codimension $2$, for short).
As a consequence, among $85$ families of such $3$-folds, general members of $19$ families are proved to be birationally rigid and the remaining $66$ families consist of birationally non-rigid varieties.
Recently, Ahmadinezhad-Zucconi \cite{AZ14} strengthened the birational rigidity result of $18$ families among the above mentioned $19$ families for every quasismooth member.
The aim of this paper is to continue the work and determine the birational Mori fiber structures of members of suitable families. 

Let $X$ be a $\mbQ$-Fano variety with only terminal singularities and with Picard number one.
A Mori fiber space which is birational to $X$ is called a {\it birational Mori fiber structure} of $X$.
We say that $X$ is {\it birationally rigid} (resp.\ {\it birationally birigid}) if the birational Mori fiber structures of $X$ consists of exactly one Mori fiber space which is $X$ itself (resp.\ two Mori fiber spaces including $X$ itself).

We explain known results for $\mbQ$-Fano $3$-folds embedded in weighted projective spaces in this direction.
There are $95$ families of anticanonically embedded $\mbQ$-Fano $3$-fold weighted hypersurfaces, and Corti-Pukhlikov-Reid \cite{CPR} proved birational rigidity of a general member of each one of $95$ families.
Recently, Cheltsov-Park \cite{CP} strengthened the result and proved birational rigidity of every quasismooth member. 
Corti-Mella \cite{CM} studied special quartic $3$-folds and exhibited an example of birationally birigid $\mbQ$-Fano $3$-folds. 
Brown-Zucconi \cite{BZ10} and Ahmadinezhad-Zucconi \cite{AZ13} studied $\mbQ$-Fano $3$-folds embedded in weighted projective $6$-spaces and in particular proved that many of them are not birationally rigid.
In this codimension 3 case, the situation becomes quite complicated compared to lower codimensional cases and the complete determination of birational Mori fiber structures seems to be quite difficult.

Before stating the main theorem of this paper, we recall the  previous result in \cite{Okada} and explain the main objects of this paper.
The $85$ families of $\mbQ$-Fano WCIs of codimension $2$ are given in \cite{IF00} and we denote by $\mcG_i$ the family No.~$i$ for $i \in I:= \{1,2,\dots,85\}$ consisting of the quasismooth members.

\begin{Thm}[{\cite[Theorems 1.2 and 1.3]{Okada}}] \label{thm:prev}
Let $X$ be a general member of $\mcG_i$, $i \in I^* := I \setminus \{1,2,3\}$.
Then maximal singularities on $X$ are classified and there is a disjoint decomposition $I^* = I^*_{br} \cup I^*_F \cup I^*_{dP}$ with the following properties.
\begin{enumerate}
\item $X$ is birationally rigid for $i \in I^*_{br}$.
\item Assume that $i \in I^*_F$.
There is a Sarkisov link to a $\mbQ$-Fano $3$-fold $X'$ and each maximal singularity on $X$ is untwisted by a Sarkisov link which is either a birational involution or a link to $X'$.
The target $X'$ is not isomorphic to $X$ and it is an anticanonically embedded $\mbQ$-Fano weighted hypersurface with a unique non-quotient teminal singular point.
\item Assume that $i \in I^*_{dP}$.
There is a Sarkisov link to a del Pezzo fiber space $X'/\mbP^1$ over $\mbP^1$ and each maximal singularity on $X$ is untwisted by a Sarkisov link which is either a birational involution or a link to $X'$.
\end{enumerate}
\end{Thm} 

Theorem \ref{thm:prev} is not enough to determine the Mori fiber structures of $X \in \mcG_i$ for $i \in I^*_F \cup I^*_{dP}$ because it says nothing about maximal singularities on the Mori fiber space $X'$ in (2) or (3) and $X'$ may admit a Sarkisov link to a Mori fiber space other than $X$ and $X'$.
The aim of this paper is to prove a result similar to Theorem \ref{thm:prev} for $X'$ and conclude birational birigidity of $X \in \mcG_i$ for $i$ belonging to a suitable subset $I_{F,cAx/2} \cup I_{F,cAx/4}$ of $I^*_F$. 

We explain main objects of this paper.
It is well known that $3$-dimensional terminal singularities which are not quotient singularities are compound du Val singularities and they are divided into several types such as $cA$, $cD$, $cE$, $cA/n$, etc.
Among them we especially consider $cAx/2$ and $cAx/4$ singular points.
We fix $i \in I^*_F$.
For $X \in \mcG_i$, the $\mbQ$-Fano $3$-fold $X'$ in Theorem \ref{thm:prev} is an anticanonically embedded $\mbQ$-Fano weighted hypersurface admitting a unique non-quotient terminal singular points together with some terminal quotient singular points.
The non-quotient singular points of general members of the family $\mcG'_i := \{X' \mid X \in \mcG_i\}$ share the same type. 
We consider subsets $I_{F,cAx/2}$ and $I_{F,cAx/4}$ of $I$ which are defined by the following rule: we have $i \in I_{F,cAx/2}$ (resp.\ $i \in I_{F,cAx/4}$) if and only if a general member of $\mcG'_i$ has a singular point of type $cAx/2$ (resp.\ $cAx/4$).
Specifically, those subsets are given as
\[
\begin{split}
I_{F,cAx/2} &= \{17,19,29,30,41,42,55,69,77\}, \\
I_{F,cAx/4} &= \{23,49,50,74,82\}.
\end{split}
\]
We state the main theorem of this paper and its direct consequence.
We emphasize that $\mcG_i$ consists of the quasismooth members of the family No.~$i$ and we do not impose any additional condition beyond quasismoothness.

\begin{Thm} \label{thm:mainbi} 
Every member $X$ of the family $\mcG_i$ is birationally birigid for $i \in I_{F,cAx/2} \cup I_{F,cAx/4}$.
More precisely, $X$ is birational to a member $X'$ of the family $\mcG'_i$ and is not birational to any other Mori fiber space.
\end{Thm}

\begin{Cor}
Every member of the family $\mcG_i$ is not rational for $i \in I_{F,cAx/2} \cup I_{F,cAx/4}$.
\end{Cor}

\begin{Ack}
The author would like to thank the referees for careful reading of the manuscript, pointing out errors and useful suggestions.
The author is partially supported by JSPS KAKENHI Grant Number 24840034 and 26800019.
\end{Ack}

\section{Preliminaries and the structure of proof} \label{sec:prelim}

\subsection{Notation and convention}

Throughout the paper we work over the field $\mbC$ of complex numbers.
A normal projective variety $X$ is said to be a $\mbQ$-{\it Fano variety} if $-K_X$ is ample, it is $\mbQ$-factorial, has only terminal singularities and its Picard number is one.
We say that an algebraic fiber space $X \to S$ is a {\it Mori fiber space} if $X$ is a normal projective $\mbQ$-factorial variety with at most terminal singularities, $\dim S < \dim X$, the anticanonical divisor of $X$ is relatively ample over $S$ and the relative Picard number is one.

Let $X$ be a normal projective $\mbQ$-factorial variety, $\mcH$ a linear system on $X$, $D \subset X$ a Weil divisor and $C \subset X$ a curve.
We say that $\mcH$ is $\mbQ$-{\it linearly equivalent} to $D$, denoted by $\mcH \sim_{\mbQ} D$, if a member of $\mcH$ is $\mbQ$-linearly equivalent to $D$.
We define $(\mcH \cdot C) := (H \cdot C)$ for $H \in \mcH$.
Assume that $X$ has only terminal singularities.
In this paper, by an {\it extremal divisorial extraction} $\varphi \colon Y \to X$ centered along $\Gamma \subset X$, we mean a contraction of a $K_Y$-negative extremal ray which contracts a divisor onto $\Gamma$ from a normal projective $\mbQ$-factorial variety $Y$ with only terminal singularities.

A closed subscheme $Z$ in a weighted projective space $\mbP (a_0,\dots,a_n)$ is {\it quasismooth} (resp.\ {\it quasismooth outside} $\msp$, where $\msp$ is a vertex) if the affine cone $C_Z \subset \mbA^{n+1}$ is smooth outside the origin (resp. outside the closure of the inverse image of $\msp$ via the morphism $\mbA^{n+1} \setminus \{o\} \to \mbP (a_0,\dots,a_n)$).
For $i = 0,1,\dots,n$, we denote by $\msp_i$ the vertex $(0 \!:\! \cdots \!:\! 1 \!:\! \cdots \!:\! 0)$ of $\mbP (a_0,\dots,a_n)$, where the $1$ is in the $(i+1)$-th position.  

Let $X'$ be a weighted hypersurface in $\mbP := \mbP (a_0,\dots,a_4)$ which is a member of $\mcG'_i$ for $i \in I_{F,cAx/2} \cup I_{F,cAx/4}$.
We write $x_0,\dots,x_3,w$ (resp.\ $x,y,z\dots,w$) for the homogeneous coordinates when we treat several families at a time (resp.\ a specific family).
For example, we write $x_0,x_1,y,z,w$ (resp.\ $x,y,z,t,w$) for the homogeneous coordinates of $\mbP (1,1,2,3,2)$ (resp.\ $\mbP (1,2,3,5,4)$).
The coordinate $w$ is distinguished so that the vertex at which only the coordinate $w$ is non-zero is the unique $cAx/2$ or $cAx/4$ point of $X'$.
For homogeneous polynomials $f_1,\dots,f_m$ in the variables $x_0,\dots,x_3,w$ or $x,y,z,\dots,w$, we denote by $(f_1 = \cdots = f_m = 0)$ the closed subscheme of $\mbP$ defined by the homogeneous ideal $(f_1,\dots,f_m)$ and denote by $(f_1 = \cdots = f_m = 0)_{X'}$ the scheme-theoretic intersection of $(f_1 = \cdots = f_m = 0)$ and $X'$.
For a polynomial $f = f (x_0,\dots,x_3,w)$ and a monomial $x_0^{x_0} \cdots x_3^{c_3} w^{d}$, we write $x_0^{c_0} \cdots x_3^{c_3} w^c \in f$ (resp.\ $x_0^{c_0} \cdots x_3^{c_3} w^c \notin f$) if the coefficient of the monomial in $f$ is non-zero (resp. zero).

A {\it weighted complete intersection curve} ({\it WCI curve}, for short) of type $(c_1,c_2,c_3)$ (resp.\ of type $(c_1,c_2,c_3,c_4)$) in $\mbP (a_0,\dots,a_4)$ (resp.\ $\mbP (a_0,\dots,a_5)$) is an irreducible and reduced curve defined by three (resp.\ four) homogeneous polynomials of degree $c_1,c_2$ and $c_3$ (resp.\ $c_1,c_2,c_3$ and $c_4$).

\subsection{Divisorial extraction centered at a singular point}

In this subsection we recall the classification of divisorial extractions centered at a terminal quotient singular point, a $cAx/2$ point and a $cAx/4$ point.

\begin{Thm}[{\cite{Kaw96}}] \label{thm:cyclicquot}
Let $(X, \msp)$ be a germ of a terminal quotient singular point of type $\frac{1}{r} (1,a,r-a)$, where $r$ is coprime to $a$ and $0 < a < r$.
If $\varphi \colon Y \to X$ is a divisorial contraction centered along $\Gamma$ such that $\msp \in \Gamma$, then $\Gamma = \msp$ and $\varphi$ is the weighted blowup with weights $\frac{1}{r} (1,a,r-a)$.
\end{Thm}  

The weighted blowup in the above theorem is called the {\it Kawamata blowup} at $\msp$.
It is the unique extremal divisorial extraction centered at a terminal quotient singular point $\msp$. 
This theorem in particular implies that there is no divisorial extraction centered along a curve through a terminal quotient singular point.
In the following, $\mbZ_m$ denotes the cyclic group of order $m$, $(x,y,z,u)/\mbZ_m (a,b,c,d)$ is the quotient of the affine $4$-space with affine coordinates $x,y,z,u$ under the $\mbZ_m$-action given by $(x,y,z,u) \mapsto (\zeta_m^a x, \zeta_m^b y, \zeta_m^c z,\zeta_m^d u)$, where $\zeta_m$ is a primitive $m$-th root of unity,  and $(g (x,y,z,u) = 0)/\mbZ_m (a,b,c,d)$ is the quotient space of the hypersurface $g = 0$ in $\mbA^4$ for a $\mbZ_m$-semi-invariant polynomial $g$.

\begin{Def} \label{def:cAx}
Let $X$ be a germ of a $3$-dimensional terminal singularity.
We say that the singularity is {\it of type} $cAx/2$ (resp.\ $cAx/4$) if there is an embedding $X \inj (x,y,z,u)/\mbZ_2 (0,1,1,1)$ (resp.\ $X \inj (x,y,z,u)/\mbZ_4 (1,3,1,2)$) such that
\[
\begin{split}
& X \cong (x^2 + y^2 + f(z,u) = 0) / \mbZ_2 (0,1,1,1), \\
(\text{resp.\ } & X \cong (x^2 + y^2 + f(z,u) =0) / \mbZ_4 (1,3,1,2)),
\end{split}
\]
where $f (z,u) \in (z,u)^4 \mbC \{z,u\}$ is a $\mbZ_2$-invariant (resp.\ $f (z,u) \in \mbC \{z,u\}$ is a $\mbZ_4$-semi-invariant and $u \notin f (z,u)$).
\end{Def}

By the main results of \cite{Ka05}, every divisorial extraction centered at $cAx/2$ (resp.\ $cAx/4$) point has discrepancy $1/2$ (resp.\ $1/4$), hence the following result gives the complete classification of divisorial extractions centered at $cAx/2$ or $cAx/4$ point.

\begin{Thm}[{\cite[Theorems 8.4 and 8.9]{Ha99}}] \label{thm:divext/2}
Let $X$ be a germ of a $cAx/2$ point with the embedding
\[
X \cong (x^2 + y^2 + f(z,u) = 0) /\mbZ_2 (0,1,1,1) \subset (x,y,z,u)/\mbZ_2 (0,1,1,1).
\]
Let us consider the weight given by $\wt (z) = \wt (u) = 1/2$ and let $k = \wt (f (z,u))$ be the weight of $f(z,u)$.
We denote by $f_{\wt = k} (z,u)$ the weight $= k$ part of $f (z,u)$. 
\begin{enumerate}
\item If $f_{\wt = k} (z,u)$ is not a square, then there is a unique divisorial extraction $\varphi \colon Y \to X$ of $X$ with discrepancy $1/2$, which is the weighted blowup with 
\[
\wt (x,y,z,u) =
\begin{cases}
\frac{1}{2} (k,k+1,1,1), & \text{if } k \text{ is even}, \\
\frac{1}{2} (k+1,k,1,1), & \text{if } k \text{ is odd}.
\end{cases}
\]
\item If $f_{\wt = k} (z,u)$ is a square, then there are exactly two divisorial extractions $\varphi_{\pm} \colon Y_{\pm} \to X$ of $X$ with discrepancy $1/2$ $($see Remark \ref{rem:divext} below for the description of $\varphi_{\pm}$$)$.
\end{enumerate}
\end{Thm}

\begin{Thm}[{\cite[Theorems 7.4 and 7.9]{Ha99}}] \label{thm:divext/4}
Let $X$ be a germ of a $cAx/4$ point with the embedding
\[
X \cong (x^2 + y^2 + f(z,u) = 0) / \mbZ_4 (1,3,1,2) \subset (x,y,z,u) / \mbZ_4 (1,3,1,2).
\]
Let us consider the weight given by $\wt (z) = 1/4$ and $\wt (u) = 1/2$, and let $k$ be the nonnegative integer such that $\wt (f (z,u)) = (2 k +1)/2$.
We denote by $f_{\wt = (2k+1)/2}$ the weight $= (2 k +1)/2$ part of $f (z,u)$.
\begin{enumerate}
\item If $f_{\wt = (2k+1)/2} (z,u)$ is not a square, then there is a unique divisorial extraction $\varphi \colon Y \to X$ of $X$ with discrepancy $1/4$, which is the weighted blowup with
\[
\wt (x,y,z,u) =
\begin{cases}
\frac{1}{4} (2k+1,2k+3,1,2), & \text{if } k \text{ is even}, \\
\frac{1}{4} (2k+3,2k+1,1,2), & \text{if } k \text{ is odd}.
\end{cases}
\]
\item If $f_{\wt = (2k+1)/2}$ is a square, then there are exactly two divisorial extractions $\varphi_{\pm} \colon Y_{\pm} \to X$ of $X$ with discrepancy $1/4$ $($see Remark \ref{rem:divext} below for the description$)$.
\end{enumerate}
\end{Thm}

\begin{Rem} \label{rem:divext}
We explain the description of $\varphi_{\pm}$ in Theorems \ref{thm:divext/2} and \ref{thm:divext/4}.
We refer the reader to \cite[Sections 7 and 8]{Ha99} for details.
Let $X$ be the germ of a $cAx/2$ (resp.\ $cAx/4$) point with the embedding as in Theorem \ref{thm:divext/2} (resp.\ \ref{thm:divext/4}).
Assume that $f_{\wt = k} (z,u) = - g (z,u)^2$ (resp.\ $f_{\wt = (2k+1)/2} = -g (z,u)^2$) for some $g (z,u)$.

We first consider the case where $k$ is even. 
Let 
\[
\chi_{\pm} \colon (x,y,z,u)/\mbZ_2 \cong (x_1,y_1,z_1,u_1)/\mbZ_2 
\]
\[
(\text{resp.\ } \chi_{\pm} \colon (x,y,z,u)/\mbZ_4 \cong (x_1,y_1,z_1,u_1)/\mbZ_4)
\] be the isomorphism defined by
\[
\chi_{\pm}^* (x_1) = x \pm g (z,u),\ \chi_{\pm}^* (y_1) = y,\ \chi_{\pm}^* (z_1) = z,\ \chi_{\pm}^* (u_1) = u.
\]
Composing the above isomorphism, $X$ is embedded into $(x_1,y_1,z_1,u_1)/\mbZ_2$ (resp.\ $(x_1,y_1,z_1,u_1)/\mbZ_4$) with the defining equation
\[
x_1^2 \mp 2 x_1 g (z_1,u_1) + y_1^2 + h(z_1,u_1) = 0,
\]
where $h = f - f_{\wt = k}$ (resp.\ $h = f - f_{\wt = (2k+1)/2}$).
Then $\varphi_{\pm}$ is the weighted blowup of $X$ with $\wt (x_1,y_1,z_1,u_1) = \frac{1}{2} (k+2,k+1,1,1)$ (resp.\ $\frac{1}{4} (2k+5,2k+3,1,2)$).

We consider the case where $k$ is odd.
Replacing $x$ and $y$ in the above argument, we have the embedding $X \inj (x_1,y_1,z_1,u_1)/\mbZ_2$ (resp.\ $(x_1,y_1,z_1,u_1)/\mbZ_4$) with the equation
\[
x_1^2 + y_1^2 \mp 2 y_1 g (z_1,u_1) + h (z_1,u_1) = 0,
\]
where $h = f - f_{\wt = k}$ (resp., $h = f - f_{\wt = (2k+1)/2}$).
Then $\varphi_{\pm}$ is the weighted blowup with $\wt (x_1,y_1,z_1,u_1) = \frac{1}{2} (k+1,k+2,1,1)$ (resp.\ $\frac{1}{4} (2k+3,2k+5,1,2)$).
Note that $\varphi_+$ and $\varphi_-$ are distinct divisorial extractions of $X$.
\end{Rem}

\begin{Def}
Let $X$ be a germ of a $cAx/2$ (resp.\ $cAx/4$) point with the embedding as in Theorem \ref{thm:divext/2} (resp.\ \ref{thm:divext/4}) and we keep notation there.
We say that it is {\it of square type} if the lowest weight term $f_{\wt = k} (z,u)$ (resp.\ $f_{\wt = (2k+1)/2}$) of $f(z,u)$ is a square.
We say that it is {\it of non-square type} if it is not a square type. 
\end{Def}

\subsection{Definition of maximal singularities}

Let $X$ be a $\mbQ$-Fano variety and $\mcH$ a movable linear system on $X$, that is, $\mcH$ is a linear system without base divisor.
Let $n$ be a positive rational number such that $\mcH \sim_{\mbQ} - n K_X$.

\begin{Def}
We define
\[
c (X, \mcH) := \max \{ \lambda \mid K_X + \lambda \mcH \text{ is canonical}\}
\]
and call it the {\it canonical threshold} of the pair $(X, \mcH)$.
\end{Def}

\begin{Def}
A {\it maximal singularity of} $\mcH$ is an extremal divisorial extraction $Y \to X$ having exceptional divisor $E$ with 
\[
\frac{1}{n} > c (X,\mcH) = \frac{a_E (K_X)}{m_E (\mcH)},
\]
where $n$ is the positive rational number such that $\mcH \sim_{\mbQ} - n K_X$, $m_E (\mcH)$ is the multiplicity of $\mcH$ along $E$ and $a_E (K_X)$ is the discrepancy of $K_X$ along $E$.
We say that an extremal divisorial extraction is a {\it maximal singularity} if there is a movable linear system $\mcH$ on $X$ such that the extraction is a maximal singularity of $\mcH$. 
A subvariety $\Gamma \subset X$ is called a {\it maximal center} if there is an maximal singularity $Y \to X$ whose center is $\Gamma$. 
\end{Def}

A maximal singularity in this paper is called a strong maximal singularity in \cite{CPR}.
Note that this definition of maximal singularity is different from the original one given in \cite{IM71}.
A maximal singularity in the original form is a divisor $E$ over $X$ such that the inequality
\[
m_E(\mcH) > n a_E (\mcH)
\]
holds, so the corresponding maximal center is the center of non-canonical singularities of the pair $(X, \frac{1}{n} \mcH)$.
The paper \cite{CP} employs this definition.
The point is that we only consider divisors over $X$ which appear as exceptional divisors of extremal divisorial extractions.
We have the following implications: if $\varphi \colon Y \to X$ is a maximal singularity (resp.\ $\Gamma \subset X$ is a maximal center), then the exceptional divisor $E$ of $\varphi$ is a maximal singularity (resp.\ $\Gamma$ is a maximal center) in the original sense.

\subsection{Excluding methods}

In this subsection, we explain several methods to exclude maximal singularities in a general setting.
Although most of the methods in this subsection have already appeared in the literature at least for particular classes of $\mbQ$-Fano $3$-folds, we re-state them in a general setting and give a proof for some of them.
We believe that this also serves as a possible future reference.

In this subsection, let $X$ be a $\mbQ$-Fano $3$-fold and we set $A := -K_X$.
If we are given a birational morphism $\varphi \colon Y \to X$, then we denote by $B$ the anticanonical divisor $-K_Y$ of $Y$.
Recall that a $\mbQ$-Fano variety in this paper is $\mbQ$-factorial with Picard number one and has only terminal singularities.

We explain methods which are mainly used to exclude curves.

\begin{Lem} \label{exclcurvelowdeg}
Let $\Gamma \subset X$ be an irreducible and reduced curve.
If $(A \cdot \Gamma) \ge (A^3)$, then $\Gamma$ is not a maximal center.
\end{Lem}

\begin{proof}
The proof is taken from STEP 1 of the proof of \cite[Theorem 5.1.1]{CPR} (see Remark \ref{rem:maxcurve} below).
Assume that $\Gamma$ is a maximal center of a movable linear system $\mcH \sim_{\mbQ} n A$.
Then there is an divisorial extraction $\varphi \colon Y \to X$ centered along $\Gamma$ with exceptional divisor $E$, which is the blowup of $X$ along $\Gamma$ around the generic point of $\Gamma$.
We have $m := \mult_E (\mcH) > n$.
Let $s > 0$ be a sufficiently divisible integer such that $s A$ is very ample.
Let $H_1, H_2 \in \mcH$ and $S \in |s A|$ be general members.
Then we have
\[
s n^2 (A^3) = (S \cdot H_1 \cdot H_2) \ge s m^2 (A \cdot \Gamma) \ge s m^2 (A^3).
\]
This shows that $n \ge m$.
This is a contradiction and $\Gamma$ is not a maximal center.
\end{proof}

\begin{Rem} \label{rem:maxcurve}
In STEP 1 of the proof of \cite[Theorem 5.1.1]{CPR}, $X$ is a quasismooth $\mbQ$-Fano weighted hypersurface, hence it has only terminal quotient singularities, and two claims are proved: for an irreducible curve $\Gamma \subset X$ which is a maximal center, (i) the inequality $(A^3) < (A \cdot \Gamma)$ holds and (ii) $\Gamma$ is contained in the nonsingular locus of $X$ (so that $(A \cdot \Gamma) \in \mbN$).
In our general setting, the proof of the claim (i), which is the one given above, does hold without any change but we cannot prove (ii).
\end{Rem}

\begin{Lem}
Let $\Gamma \subset X$ be an irreducible and reduced curve.
Assume that there is an effective divisor $S$ on $X$ containing $\Gamma$ and a movable linear system $\mcM$ on $X$ whose base locus contains $\Gamma$ with the following properties.
\begin{enumerate} 
\item $S \sim_{\mbQ} m A$ for some rational number $m \ge 1$.
\item For a general member $T \in \mcM$, $T$ is a normal surface, the intersection $S \cap T$ is contained in the base locus of $\mcM$ set-theoretically and $S \cap T$ is reduced along $\Gamma$.
\item Let $T \in \mcM$ be a general member and let $\Gamma$, $\Gamma_1, \dots, \Gamma_l$ the irreducible and reduced curves contained in the base locus of $\mcM$.
For each $i = 1,\dots,l$, there is an effective $1$-cycle $\Delta_i$ on $T$ such that $(\Gamma \cdot \Delta_i) \ge (A \cdot \Delta_i) > 0$ and $(\Gamma_j \cdot \Delta_i) \ge 0$ for $j \ne i$. 
\end{enumerate}
Then $\Gamma$ is not a maximal center.
\end{Lem}

\begin{proof}
Assume that $\Gamma$ is a maximal center of a movable linear system $\mcH \sim_{\mbQ} n A$.
Then there is an exceptional divisor $E$ over $X$ with center $\Gamma$ such that $\mult_E \mcH > n$.
Let $T \in \mcM$ be a general member.
Since the base curves of $\mcM$ are $\Gamma$, $\Gamma_1,\dots,\Gamma_l$ and $S \cap T \subset \Bs \mcM$ is reduced along $\Gamma$, we can write
\[
\begin{split}
A|_T &\sim_{\mbQ} \frac{1}{n} \mcH |_T = \frac{1}{n} \mcL + \gamma \Gamma + \bsum_{i = 1}^l \gamma_i \Gamma_i, \\
m A|_T &\sim_{\mbQ} S|_T = \Gamma + \bsum_{i=1}^l c_i \Gamma_i,
\end{split}
\]
where $\gamma, \gamma_i \ge 0$, $c_i \ge 0$ and $\mcL$ is a movable linear system on $T$.
We have
\[
\gamma \ge \frac{1}{n} \mult_E \mcH > 1.
\]
We set
\[
D := (A- \gamma S)|_T \sim_{\mbQ} \frac{1}{n} \mcL + \bsum_{i=1}^l (\gamma_i - c_i \gamma) \Gamma_i.
\]

If $c_1 = \cdots = c_l = 0$ (this includes the case where $l = 0$), then
\[
(1- m \gamma) (A \cdot \Gamma) = (D \cdot \Gamma) = \frac{1}{n} (\mcL \cdot \Gamma) + \bsum_{i=1}^l \gamma_i (\Gamma_i \cdot \Gamma) \ge 0
\]
since $\mcL$ is nef and $\Gamma \ne \Gamma_i$.
This implies $\gamma \le 1/m \le 1$, a contradiction.
Thus, possibly re-ordering $\Gamma_i$'s, we may assume that $c_1,\dots,c_k$ are non-zero, $c_{k+1} = \cdots = c_l = 0$ and $\gamma_k / c_k \ge \cdots \ge \gamma_1/c_1$ for some $1 \le k \le l$.

Assume that $\gamma \le \gamma_1/c_1$.
Then we have
\[
(1-m \gamma) (A \cdot \Gamma) = (D \cdot \Gamma) = \frac{1}{n} (\mcL \cdot \Gamma) + \bsum_{i=1}^l (\gamma_i - c_i \gamma) (\Gamma_i \cdot \Gamma) \ge 0.
\]
It follows that $\gamma \le 1/m \le 1$.
This is a contradiction and we have $\gamma > \gamma_1/c_1$.
We set
\[
D_1 := (c_1 A - \gamma_1 S)|_T \sim_{\mbQ} \frac{c_1}{n} \mcL + (c_1 \gamma - \gamma_1) \Gamma + \bsum_{i = 2}^l (c_1 \gamma_i - c_i \gamma_1) \Gamma_i.
\]
Let $\Delta_1$ be the effective $1$-cycle on $T$ as in (3).
We have
\[
(c_1 - m \gamma_1) (A \cdot \Delta_1) = (D_1 \cdot \Delta_1) \ge (c_1 \gamma - \gamma_1) (\Gamma \cdot \Delta_1) \ge (c_1 \gamma - \gamma_1) (A \cdot \Delta_1)
\]
since $\mcL$ is nef, $(\Gamma_i \cdot \Delta_1) \ge 0$ for $i \ge 2$ and $(\Gamma \cdot \Delta_1) \ge (A \cdot \Delta_1)$.
This implies $c_1 - m \gamma_1 \ge c_1 \gamma - \gamma_1$ and thus $\gamma \le ((1-m)\gamma_1 + c_1)/c_1 \le 1$.
This is a contradiction and $\Gamma$ is not a maximal center.
\end{proof}

We explain methods which are mainly used to exclude nonsingular points.

\begin{Thm}[{\cite[Theorem 5.3.2]{CPR}, \cite[Corollary 3.4]{Co00}}]  \label{multineq1}
Let $\msp \in X$ be a germ of a nonsingular $3$-fold and $\mcH$ a movable linear system on $X$.
Assume that $K_X + \frac{1}{n} \mcH$ is not canonical at $\msp$.
\begin{enumerate}
\item If $\msp \in S \subset X$ is a surface and $\mcL = \mcH |_S$, then $K_S + \frac{1}{n} \mcL$ is not log canonical.
\item If $H_1 \cap H_2$ is the intersection of two general members $H_1, H_2 \in \mcH$, then $\mult_{\msp} (H_1 \cap H_2) > 4 n^2$.
\end{enumerate}
\end{Thm}

\begin{Thm}[{\cite[Theorem 3.1]{Co00}}] \label{multineq2}
Let $\msp \in \Delta_1 + \Delta_2 \subset S$ be an analytic germ of a normal crossing curve on a nonsingular surface.
Let $\mcL$ be a movable linear system on $S$, and write $(\mcL^2)_{\msp}$ for the local intersection multiplicity $(L_1 \cdot L_2)_{\msp}$ of two general members $L_1, L_2 \in \mcL$.
Fix rational numbers $a_1,a_2 \ge 0$, and assume that
\[
K_S + (1-a_1) \Delta_1 + (1-a_2) \Delta_2 + \frac{1}{n} \mcL
\]
is not log canonical at $\msp$.
Then the following assertions hold.
\begin{enumerate}
\item If either $a_1 \le 1$ or $a_2 \le 1$, then $(\mcL^2)_{\msp} > 4 a_1 a_2 n^2$.
\item If both $a_1 > 1$ and $a_ 2 > 1$, then $(\mcL^2)_{\msp} > 4 (a_1 + a_2 - 1) n^2$.
\end{enumerate}
\end{Thm}

\begin{Def}[{\cite[Definition 5.2.4]{CPR}}] \label{def:isol}
Let $L$ be a Weil divisor class on $X$ and $\Gamma \subset X$ an irreducible subvariety of codimension $\ge 2$.
For an integer $s > 0$, consider the linear system
\[
\mcL^s_{\Gamma} := |\mcI^s_{\Gamma} (s L)|,
\]
where $\mcI_{\Gamma}$ is the ideal sheaf of $\Gamma$.
We say that the class $L$ {\it isolates} $\Gamma$ if $\Gamma \subset \Bs \mcL^s_{\Gamma}$ is an isolated component for some positive integer $s$ and, in the case $\Gamma$ is a curve, the generic point of $\Gamma$ appears in $\Bs \mcL^s_{\Gamma}$ with multiplicity one.
\end{Def} 

\begin{Lem}[{\cite{CPR}}] \label{exclmostnspt}
Let $\msp \in X$ be a nonsingular point.
If $l A$ isolates $\msp$ for some $0 < l \le 4/(A^3)$, then $\msp$ is not a maximal center.
\end{Lem}

\begin{proof}
See \cite[Proof of (A)]{CPR} for a proof.
\end{proof}

We explain methods which are mainly used to exclude singular points.

\begin{Lem} \label{lem:intNE}
Let $\Gamma \subset X$ be a maximal center and $\varphi \colon (E \subset Y) \to (\Gamma \subset X)$ the maximal extraction.
Then the $1$-cycle $(-K_Y)^2$ is in the interior of the cone $\overline{\operatorname{NE}} (Y)$ of effective curves on $Y$.
\end{Lem}

\begin{proof}
The proof of \cite[Lemma 5.2.1]{CPR} applies without any change.
\end{proof}

The following is a direct consequence of Lemma \ref{lem:intNE}.

\begin{Cor} \label{tsmethod}
Let $\varphi \colon Y \to X$ be an extremal divisorial extraction with exceptional divisor $E$ centered along an irreducible and reduced subvariety $\Gamma \subset X$.
If $(M \cdot (-K_Y)^2) \le 0$ for some nef divisor $M$ on $Y$, then $\varphi$ is not a maximal singularity.
\end{Cor}

\begin{Lem} \label{excltestsurf}
Let $\varphi \colon Y \to X$ be an extremal divisorial extraction centered at a point $\msp \in X$ with exceptional divisor $E$.
Assume that there are surfaces $S$ and $T$ on $Y$ with the following properties.
\begin{enumerate}
\item $S \sim_{\mbQ} a B + d E$ and $T \sim_{\mbQ} b B + e E$ for some integers $a, b, d, e$ such that $a, b > 0$, $0 \le e < a_E (K_X) b$ and $a e - b d \ge 0$.
\item The intersection $\Gamma := S \cap T$ is a $1$-cycle whose support consists of irreducible and reducible curves which are numerically proportional to each other.
\item $(T \cdot \Gamma) \le 0$.
\end{enumerate}
Then, $\varphi$ is not a maximal singularity.
\end{Lem}

\begin{proof}
Let $R \subset \overline{\operatorname{NE}} (Y)$ be the extremal ray generated by a curve contracted by $\varphi$ and $Q$ the other extremal ray so that $\overline{\operatorname{NE}} (Y) = R + Q$.
We write $\Gamma = \sum \gamma_i \Gamma_i$, where $\gamma_i > 0$ and $\Gamma_i$'s are irreducible and reduced curves.
By (2), each component of $\Gamma$ is in the ray spanned by $\Gamma$.

We claim that $Q$ is generated by $\Gamma$.
We have
\[
(T \cdot R) = b (\varphi^*A \cdot R) + (e - a_E (K_X) b) (E \cdot R) > 0
\]
since $(E \cdot R) < 0$ and $e - a_E (K_X) b < 0$.
By (3), $\Gamma$ does not belong to $R$ and thus $(\varphi^*A \cdot \Gamma) > 0$ and 
\[
(E \cdot \Gamma) = \frac{1}{e-a_E (K_X) b} \left( (T \cdot \Gamma) - b (\varphi^*A \cdot \Gamma) \right) > 0.
\]
Let $\alpha \ge 0$ be the rational number such that $(T + \alpha E) \cdot \Gamma = 0$.
It is enough to show that $M := T + \alpha E$ is nef.
We set $\beta = (a_E (K_X) b - e) - \alpha$.
Note that $\beta > 0$ because otherwise $\Gamma$ intersects $M \sim_{\mbQ} b \varphi^* A - \beta E$ positively.
Assume that $M$ is not nef.
Then there is an irreducible and reduced curve $C \subset Y$ such that $(M \cdot C) < 0$.
This implies that 
\[
(E \cdot C) = \frac{1}{\beta} \left( b (\varphi^*A \cdot C) - (M \cdot C) \right) > 0.
\]
Thus we have 
\[
(T \cdot C) = (M \cdot C) - \alpha (E \cdot C) < 0,
\] 
and 
\[
(S \cdot C) = (T \cdot C) - \frac{a e - b d}{b} (E \cdot C) < 0.
\]
These show that $C$ is contained in both $S$ and $T$.
It follows that $C = \Gamma_i$ for some $i$.
But this is a contradiction since $(M \cdot \Gamma_i) = 0$.
Therefore, $M$ is nef and thus $Q$ is generated by $\Gamma$.

Assume that $\varphi$ is a maximal singularity.
Then there is a movable linear system $\mcH \sim_{\mbQ} n A$ on $X$ such that the rational number $c$ defined by
\[
K_Y + \frac{1}{n} \mcH_Y = \varphi^* (K_X + \frac{1}{n} \mcH) - c E
\]
is positive.
Thus $\mcH_Y \sim_{\mbQ} n B - n c E$.
It follows that  
\[
S \cdot \mcH_Y = S \cdot (n B - n c E) = S \cdot \left(\frac{n}{b} T - \left(n c + \frac{n e}{b} \right) E\right)
\] 
is an effective $1$-cycle on $Y$ and thus it is contained in $\overline{\operatorname{NE}} (Y) = R + Q$.
This is impossible since $S \cdot E$ generates $R$ and $S \cdot T$ generates $Q$.
Therefore, $\varphi$ is not a maximal singularity.
\end{proof}

The following results are due to \cite{CP}.

\begin{Lem} \label{crispnegdef}
Let $\varphi \colon Y \to X$ be an extremal divisorial extraction with exceptional divisor $E$.
Suppose that there is an effective divisor $S \sim_{\mbQ} b B + e E$ with $b > 0$ and $e \ge 0$ on $Y$ and a normal surface $T \ne E$ on $Y$ such that the support of the one-cycle $S|_T$ consists of curves on $T$ whose intersection form is negative-definite.
Then $\varphi$ is not a maximal singularity.
\end{Lem}

\begin{proof}
This is a version of \cite[Lemma 3.2.7]{CP} and we follow the proof there.
Set $S|_T = \sum c_i C_i$, where $c_i > 0$ and $C_i$ are distinct irreducible and reduced curves on $T$.
Assume that $\varphi$ is a maximal singularity.
Then there is a movable linear system $\mcH \sim_{\mbQ} - n K_X$ for some rational number $n > 0$ such that the rational number $c$ defined by 
\[
K_Y + \frac{1}{n} \mcH_Y = \varphi^* (K_X + \frac{1}{n} \mcH) - c E \sim_{\mbQ} - c E,
\]
where $\mcH_Y$ is the birational transform of $\mcH$  by $\varphi$, is positive.
From this, we have $\mcH_Y \sim_{\mbQ} n B - n c E$ and thus
\[
(b \mcH_Y + (b c + n e) E)|_T \sim_{\mbQ}  n S|_T = \sum n c_i C_i.
\]
This contradicts the assumption that the intersection form of $C_i$'s is negative-definite, whose proof is completely parallel to that of \cite[Lemma 3.2.7]{CP} and we omit the proof.
Therefore $\varphi$ is not a maximal singularity.
\end{proof}

\begin{Lem} \label{crispinfc}
Let $\varphi \colon Y \to X$ be an extremal divisorial extraction with exceptional divisor $E$.
Suppose that there are infinitely many irreducible and reduced curve $C_{\lambda}$ on $Y$ such that $(-K_Y \cdot C_{\lambda}) \le 0$ and $(E \cdot C_{\lambda}) > 0$.
Then $\varphi$ is not a maximal singularity.
\end{Lem}

\begin{proof}
This is a re-statement of \cite[Lemma 3.2.8]{CP} in a general setting and we follow the proof there.
Assume that $\varphi$ is a maximal singularity.
Then there is a movable linear system on $\mcH \sim_{\mbQ} - n K_X$ for some rational number $n > 0$ such that the rational number $c$ defined by
\[
K_Y + \frac{1}{n} \mcH_Y = \varphi^* (K_X + \frac{1}{n} \mcH) - c E,
\]
where $\mcH_Y$ is the birational transform of  $\mcH$ by $\varphi$, is positive.
We have
\[
K_Y + \frac{1}{n} \mcH_Y \sim_{\mbQ} - c E.
\]
It follows that 
\[
(\mcH_Y \cdot C_{\lambda}) = n (-K_Y \cdot C_{\lambda}) - n c (E \cdot C_{\lambda}) < 0
\]
by the assumption on the intersection numbers and the positivity of $c$.
This shows that $C_{\lambda}$ is contained in the base locus $\mcH_Y$.
This is a contradiction since there are infinitely many such $C_{\lambda}$'s and $\mcH_Y$ is movable.
Therefore $\varphi$ is not a maximal singularity.
\end{proof}

\subsection{Untwisting birational maps}

Let $X$ be a $\mbQ$-Fano $3$-fold.
The following definition is due to Corti \cite{Co95}.

\begin{Def}
Let $\varphi \colon Y \to X$ be a maximal singularity.
We say that a birational map $\iota \colon X \ratmap X'$ to a $\mbQ$-Fano $3$-fold $X'$ {\it untwists} the maximal singularity $\varphi$, or is an {\it untwisting birational map} for $\varphi$, if for any movable linear system $\mcH \sim_{\mbQ} n (-K_X)$ such that $\varphi$ is a maximal singularity of $\mcH$, the rational number $n'$ defined by $\mcH' \sim_{\mbQ} n' (-K_{X'})$, where $\mcH'$ is the birational transform of $\mcH$ by $\iota$, satisfies $n' < n$.
\end{Def}

We say that a birational map between normal projective varieties is {\it small} if it is an isomorphism in codimension one.

\begin{Lem} \label{untwistbirmap}
Let $\varphi \colon Y \to X$ be a maximal singularity.
Let $\iota \colon X \ratmap X'$ be a birational map to a $\mbQ$-Fano $3$-fold $X'$ which is not biregular.
If there is an extremal divisorial extraction $\varphi' \colon Y' \to X'$ such that the induced birational map $\tau := \varphi' \circ \iota \circ \varphi^{-1} \colon Y \ratmap Y'$ is small, then $\iota$ untwists the maximal singularity $\varphi$.
\end{Lem}

\begin{proof}
Let $\mcH \sim_{\mbQ} n (-K_X)$ be a movable linear system such that $\varphi$ is a maximal singularity of $\mcH$, and $\mcH' \sim_{\mbQ} n' (-K_{X'})$ the birational transform of $\mcH$ on $X'$.
Let $E$ and $E'$ be exceptional divisors of $\varphi$ and $\varphi'$, respectively.
Set $B = -K_Y$ and $B' = - K_{Y'}$.
We have $\tau_* B = B'$ since $\tau$ is small.
We write $\tau_* E \sim_{\mbQ} \alpha B' + \beta E'$ for some rational numbers $\alpha$ and $\beta$.

We claim that $\alpha > 0$.
Since $\tau_* E$ is an effective divisor, we have $\alpha \ge 0$.
Suppose that $\alpha = 0$.
Then, since $\tau_*E$ is a prime divisor which is $\mbQ$-linearly equivalent to $\beta E'$, we have $\beta =1$ and $\tau_*E = E'$.
This implies that $\iota$ is small, hence an isomorphism since $\iota _*(-K_X) = -K_{X'}$.
This is a contradiction and we have $\alpha > 0$.

We write $\varphi^*\mcH = \mcH_Y + m E$ so that 
\[
\mcH_Y \sim_{\mbQ} n \varphi^* (-K_X) - m E
\sim_{\mbQ} n B - (m - n a) E,
\]
where $a = a_E (K_X)$.
We have $m > n a$ since $\varphi$ is a maximal singularity of $\mcH$. 
We have
\[
\tau_* (\mcH_Y) \sim_{\mbQ} n \tau_* (B) - (m-n a) \tau_* E
= (\alpha n a + n - \alpha m) B' - \beta (m - n a) E'.
\]
It follows that 
\[
n' = \alpha n a + n - \alpha m  < n
\]
since $\alpha > 0$ and $m > n a$.
This completes the proof.
\end{proof}

We recall the definition of Sarkisov links between $\mbQ$-Fano $3$-folds.

\begin{Def}
A {\it Sarkisov link} between two $\mbQ$-Fano $3$-folds is a birational map $\sigma \colon X \ratmap X'$ which factorizes as
\[
\xymatrix{
Y \ar@{-->}[r] \ar[d] & Y' \ar[d] \\
X \ar@{-->}[r]^{\sigma} & X'}
\]
where $Y \to X$ and $Y' \to X'$ are extremal divisorial extractions, and $Y \ratmap Y'$ is a composite of inverse flips, flops and flips (in that order).
We sometimes call $\sigma$ the {\it Sarkisov link starting with} the divisorial extraction $Y \to X$.
The center on $X$ of the contraction $Y \to X$ is called the {\it center} of a Sarkisov link.
\end{Def}

The following is a direct consequence of Lemma \ref{untwistbirmap}.

\begin{Cor} \label{cor:slinkuntwist}
If $\varphi \colon Y \to X$ is a maximal singularity, then the Sarkisov link $\iota \colon X \ratmap X'$ starting with $\varphi$ $($if it exists$)$ untwists the maximal singularity $\varphi$.
\end{Cor}

The following lemma will be frequently used and we include its proof for the reader's convenience.

\begin{Lem} \label{lem:isomrk2}
Let $V$ be a normal projective $\mbQ$-factorial variety with Picard number $2$ and $\tau \colon V \ratmap V$ a small birational map.
If there is a divisor $D$ on $V$ such that $\tau_*D \sim_{\mbQ} D$, and  $D$ and $K_V$ are linearly independent in $\Pic (V) \otimes \mbQ$, then $\tau$ is an isomorphism.
\end{Lem} 

\begin{proof}
We see that $\tau_*K_V = K_V$ since $\tau$ is small.
By the assumption, we have $\tau_*D \sim_{\mbQ} D$.
Moreover, $D$ and $K_V$ generate $\Pic (Y) \otimes \mbQ$.
If we pick an ample divisor $A = a D + b K_V$ for some $a,b \in \mbQ$, then $\tau_* A \sim_{\mbQ} A$.
Therefore, $\tau$ is an isomorphism.
\end{proof}

\begin{Lem} \label{lem:untwistslink}
Assume that $\varphi \colon Y \to X$ is a maximal singularity.
Let $\iota \colon X \ratmap X$ be a birational automorphism which is not biregular.
If the birational automorphism $\tau \colon Y \ratmap Y$ induced by $\iota$ is small, then $\iota$ is a Sarkisov link.
\end{Lem}

\begin{proof}
Since $\varphi$ is a maximal singularity, the $2$-ray game starting with $\varphi$ ends with a Mori fiber space $X'$ and gives a Sarkisov link $\eta \colon X \ratmap X'$ (see \cite{Co95}).
More precisely, we have a sequence
\[
Y = Y_0 \ratmap Y_1 \ratmap \cdots \ratmap Y_k
\]
of inverse flips, flops and flips (in that order), where each $Y_i$ has only terminal singularities, and a $K_{Y_k}$-negative extremal contraction $Y_k \to Z$.
The contraction $Y_k \to Z$ is either divisorial or of fiber type.
In the former case the output $X' := Z$ is a $\mbQ$-Fano $3$-fold and in the latter case the output $X' := Y_k$ is a Mori fiber space.
We shall show that $X' = X$ and $\eta = \iota$.

For each $0 \le i \le k$, let $\theta_i \colon Y = Y_0 \ratmap Y_i$ be the composite of $Y_0 \ratmap Y_1 \ratmap \cdots \ratmap Y_{i-1} \ratmap Y_i$ and $\tau_i$ the birational automorphism of $Y_i$ induced by $\tau$.
We set $B_i := -K_{Y_i}$ which is fixed by $(\tau_i)_*$ since $\tau_i$ small.
Note that, by Lemma \ref{lem:isomrk2}, for any divisor $D$ on $Y_i$ which is not $\mbQ$-linearly equivalent to a multiple of $B_i$, $(\tau_i)_* D$ cannot be $\mbQ$-linearly equivalent to $D$.

We claim that there is $0 \le i \le k-1$ such that $Y_i \ratmap Y_{i+1}$ is either a flop or a flip.
Indeed, if this is not the case, then $Y_0 \ratmap \cdots \ratmap Y_k$ is a composite of inverse flips, hence $B_k$ is ample since $Y_k \ratmap Y_{k-1}$ is a flip and $Y_k \ratmap Z$ is $B_k$-positive.
It follows that $\tau_k$ is biregular since $\tau_k$ preserves the ample divisor $B_k$.
Let $D$ be the exceptional divisor of $Y_k \to Z$ (resp.\ the pullback of a divisor on $Z$) if $Y_k \to Z$ is divisorial (resp.\ of fiber type).
Clearly $D \not\sim_{\mbQ} m B_k$ for any $m$.
By the definition of $D$, we have $(\tau_k)_* D \sim D$ and this implies that $\tau$ is biregular.
This is a contradiction and the claim is proved.
Now we can define
\[
j = \min \{ i \mid Y_i \ratmap Y_{i+1} \text{ is either a flop or a flip}\}.
\]
Note that $0 \le j \le k-1$ and we see that $Y_i \ratmap Y_{i+1}$ is an inverse flip (resp.\ flip) for $i < j$ (resp.\ $i > j$) and $Y_j \ratmap Y_{j+1}$ is either a flop or a flip.
In particular, $B_j = -K_{Y_j}$ is nef and big, and more precisely, it is ample if and only if $Y_j \ratmap Y_{j+1}$ is a flip.

Assume $-K_{Y_j}$ is ample.
Then $\tau_j$ is an isomorphism.
In this case, the composite 
\[
\tau = \theta_j^{-1} \circ \tau_j \circ \theta_j \colon Y \ratmap Y_j \cong Y_j \ratmap Y
\]
is a composite of inverse flips and flips.
This shows that $\iota$ is a Sarkisov link.
By the uniqueness of $2$-ray game, we have $X' = X$ and $\eta = \tau$.

Assume that $-K_{Y_j}$ is not ample but is nef and big.  
Then $Y_j \ratmap Y_{j-1}$ is a flip and $Y_j \ratmap Y_{j+1}$ is a flop.
Let $Y_j \to Z_{j-1}$ be the corresponding flipping contraction and $D$ the pullback of a divisor on $Z_{j-1}$.
Note that $B_j$ and $D$ generate the Picard group of $Y_j$.
We claim that $\tau_j$ is not biregular.
Indeed if it is biregular, then $(\tau_j)_* D \sim D$ since $(\tau_j)_* (D)$ must be the divisor class defining the unique flipping contraction $Y_j \to Z_{j-1}$.
But then $\tau$ is biregular since $\tau_* D_Y \sim D_Y$ for $D_Y := (\theta_j)_*^{-1} D$.
This is a contradiction and $\tau_j$ is not biregular.
Let $Y_j \to Z_{j+1}$ be the flopping contraction.
Then the diagram
\[
\xymatrix{
Y_j \ar[rd] \ar@{-->}[rr]^{\tau_j} & & Y_j \ar[ld] \\
& Z_{j+1} &}
\]
commutes because $Y_j \to Z_{j+1}$ is the anticanonical morphism and $(\tau_j)_* (-K_{Y_j}) = -K_{Y_j}$.
By the uniqueness of flop, we see that $Y_{j+1} \cong Y_j$ and $Y_j \ratmap Y_{j+1}$ coincides with $\tau_j \colon Y_j \ratmap Y_j$.
Thus the composite
\[
\tau = \theta_j^{-1} \circ \tau_j \circ \theta_j \colon Y \ratmap Y_j \ratmap Y_j \ratmap Y
\]
is a composite of inverse flips, flops and flips, and $\iota$ is a Sarkisov link.
By the uniqueness of $2$-ray game, we have $X' = X$ and $\eta = \iota$.
This completes the proof.
\end{proof}

\subsection{Definition of families and standard defining equations} \label{sec:stdefeq}

In this subsection we define several families of $\mbQ$-Fano weighted complete intersections.

There are $95$ (resp.\ $85$) families of anticanonically embedded $\mbQ$-Fano weighted hypersurfaces (resp.\ complete intersections of codimension $2$) of dimension $3$.  

\begin{Def}
We define subsets of $I = \{1,2,\dots,85\}$ as follows:
\begin{align*}
I'_2 &:= \{19,30,42\} & I''_2 &:= \{17,29,41,55,69,77\} \\
I'_4 &:= \{23,50\} & I''_4 &:= \{49,74,82\}
\end{align*}
and we set $I_{F,cAx/2} = I'_2 \cup I''_2$, $I_{F,cAx/4} = I'_4 \cup I''_4$, $I' = I'_2 \cup I'_4$ and $I'' = I''_2 \cup I''_4$.
\end{Def}

\begin{Def}
For each $i \in I_{F,cAx/2} \cup I_{F,cAx/4}$, we define $\mcG_i$ to be the family No.~$i$ of anticanonically embedded quasismooth weighted complete intersections of codimension $2$ (cf.\ Section \ref{sec:table1}).
\end{Def}

We note that the numbering of the family $\mcG_i$ of $\mbQ$-Fano $3$-fold weighted complete intersections of codimension $2$ coincides with the one given in the Fletcher's list.

Let $X$ be a member of $\mcG_i$ with $i \in I_{F,cAx/2} \cup I_{F,cAx/4} = I' \cup I''$.
Then there is a Sarkisov link $X \ratmap X'$ to a $\mbQ$-Fano $3$-fold weighted hypersurface $X'$.
Although there may be several such links, the target $X'$ is uniquely determined by $X$ (see \cite[Section 4.2]{Okada} and also Section \ref{sec:SLWHWCI} for details).

\begin{Def}
We call $X'$ the {\it birational counterpart} of $X$.
We define $\mcG'_i$ to be the family of birational counterparts of members of $\mcG_i$ (cf.\ Section \ref{sec:table2}).
\end{Def}

In order to make explicit the family $\mcG'_i$, we introduce standard defining equations. 
Let $X$ be a member of $\mcG_i$, which is a complete intersection $X = X_{d_1,d_2} \subset \mbP (a_0,\dots,a_5)$ of weighted hypersurfaces of degree $d_1$ and $d_2$.
In this subsection we assume that $d_1 < d_2$, $a_5 \ge a_i$ for any $i$ and $a_2 \le a_3$.
Note that we do not assume $a_0 \le \cdots \le a_5$.
Let $x_0,x_1,x_2,x_3,u$ and $v$ be the homogeneous coordinates of $\mbP (a_0,\dots,a_5)$ with $\deg x_i = a_i$ for $0 \le i \le 3$, $\deg u = a_4$ and $\deg v = a_5$.

Assume that $i \in I'$.
In this case, after re-ordering $a_i$'s, we may assume that $a_5 = a_4$, $d_1 = a_0 + a_5 = 2 a_1$ and $d_2 = a_4 + a_5 = 2 a_4$.
By a suitable choice of coordinates, defining equations of $X$ can be written as
\begin{equation} \label{eq:normformI}
\begin{cases}
v x_0 + u (x_0 + f) + g = 0, \\
v u - h = 0,
\end{cases}
\end{equation}
where $f = f (x_2,x_3)$ and $g,h \in \mbC [x_0,x_1,x_2,x_3]$.
Quasismoothness of $X$ in particular implies that $x_1^2 \in g$ and, after re-scaling $x_1$, we assume that the coefficient of $x_1^2$ in $g$ is $1$.

Assume that $i \in I''$.
In this case, after re-ordering $a_i$'s, we may assume that $a_5 > a_i$ for $i \ne 5$, $d_1 = a_0 + a_5 = 2 a_4$ and $d_2 = a_4 + a_5 = 2 a_1$.
By a suitable choice of coordinates, defining equations of $X$ can be written as
\begin{equation} \label{eq:normformII}
\begin{cases}
v x_0 + u^2 + u f + g = 0,  \\
v u - h = 0,
\end{cases}
\end{equation}
where $f, g, h \in \mbC [x_0,x_1,x_2,x_3]$.
Quasismoothness of $X$ implies $x_1^2 \in h$ and, after re-scaling $x_1$, we assume that the coefficient of $x_1^2$ in $h$ is $1$.
We call \eqref{eq:normformI} or \eqref{eq:normformII} {\it standard} defining equations of a member $X$ of $\mcG_i$.

Let $X$ be a member of $\mcG_i$, $i \in I' \cup I''$, defined by standard defining equations as above.
Then the birational counterpart of $X$ is the weighted hypersurface 
\begin{equation} \label{eq:stdX'}
X' =
\begin{cases}
(w^2 x_0 (x_0 + f) + w g + h = 0) \subset \mbP (a_0,\dots,a_3,b), & \text{if $i \in I'$}, \\
(w^3 x_0^2 + w^2 x_0 f + w g + h = 0) \subset \mbP (a_0,\dots,a_3,b), & \text{if $i \in I''$},
\end{cases}
\end{equation}
where $b := a_4 - a_0$.
We call the above equation \eqref{eq:stdX'} a {\it standard equation} for a member $X'$ of $\mcG'_i$.

\begin{Lem} \label{lem:singularity}
Let $X'$ be a member of $\mcG'_i$, where $i \in I' \cup I''$, with a standard defining equation and $\msp' = (0 \!:\! 0 \!:\! 0 \!:\! 0 \!:\! 1)$ the point of $X'$.
\begin{enumerate}
\item $X'$ is quasismooth outside $\msp'$ and it has only terminal quotient singular points except at $\msp'$.
\item If $i \in I_2$ $($resp.\ $I_4$$)$, then $X'$ has a $cAx/2$ $($resp.\ $cAx/4$$)$ point at $\msp'$.
\item If $i \in I'$ $($resp.\ $i \in I''$$)$, then the singularity of $X'$ at $\msp'$ is of non-square type if and only if $f = 0$ $($resp.\ $(\prt g/\prt x_1) (0,0,x_2,x_3) = 0$$)$ as a polynomial.
\end{enumerate}
\end{Lem}

\begin{proof}
(1) is proved in \cite[Section 4.2]{Okada}.
After a suitable analytic coordinate change, the germ $(X',\msp')$ is analytically isomorphic to
\[
(x_0^2 + x_1^2 + \varphi (x_2,x_3) = 0)/\mbZ_b \subset (x_0,x_1,x_2,x_3)/\mbZ_b,
\]
where $b = 2$ or $4$ depending on whether $i \in I_2$ or $i \in I_4$.
This proves (2).
Suppose that $i \in I'$ (resp.\ $i \in I''$).
Then the lowest weight part of $\varphi$ coincides with $- f^2/4$ (resp.\ $- {g'}^2/4$),
where $g' = (\prt g/\prt x_1)(0,0,x_2,x_3)$ and the weight is defined as in Definition \ref{def:cAx}.
This proves (3).
We explain the above arguments when $i \in I'$.
By setting $w = 1$, the germ $(X',\msp')$ is the quotient of the hypersurface $x_0 (x_0 + f) + g + h = 0$ in $\mbA^4$.
We apply Weierstrass preparation theorem with respect to $x_0^2$ and we have that $(X',\msp')$ is analytically equivalent to the quotient of $x_0^2 + x_0 p + q = 0$, where $p,q$ are holomorphic functions in variables $x_1,x_2,x_3$.
We can keep track of the lowest weight part of $p (0,x_2,x_3)$ and it is $f = f (x_2,x_3)$ and $x_1^2 \in q$.
Then, replacing $x_0 \mapsto x_0 - p/2$, the equation is transformed into $x_0^2 + q$, where $q$ is a holomorphic functions in variables $x_1,x_2,x_3$, $x_1^2 \in q$ and the lowest weight part of $q (0,x_2,x_3)$ is $- f^2/4$.
By applying Weirestrass preparation theorem to $q$ with respect to $x_1^2$ and then replacing $x_1$, we can eliminate the variable $x_1$ in $q$ except for $x_1^2$.
Therefore, $(X',\msp')$ is the quotient of $x_0^2 + x_1^2 + q$. where $q$ is a holomorphic function in variables $x_2, x_3$ whose lowest weight part is $-f^2/4$.
This completes the proof.
\end{proof}

\begin{Rem}
Let $X'$ be a member of $\mcG'_i$ with a standard defining equation.
If $i \in I''$, then the assertion $x_1^2 \in h$ follows from (1) of Lemma \ref{lem:singularity}.
But if $i \in I'$, then the assertion $x_1^2 \in g$ does not necessarily follow from (1) of Lemma \ref{lem:singularity}.
\end{Rem}

\subsection{Structure of the proof}

We explain the structure of the proof of the main theorem.
In \cite{Okada}, we classified maximal singularities of $X \in\mcG_i$ and showed that to each maximal singularity there is associated a Sarkisov link which is either a birational involution or a link to $X' \in \mcG'_i$.

In \cite{Okada}, the family $\mcG_i$ consists of quasismooth members of the family No.~$i$ satisfying additional condition \cite[Condition 3.1]{Okada}.
The additional condition is vacuous for $i \in I_{F,cAx/2} \cup I_{F,cAx/4} \setminus \{19\}$ but not for $i = 19 \in I_{F,cAx/2}$.
In this paper we strengthen the previous result \cite[Theorem 1.4]{Okada} for the family No.~$19$.
Thus we re-state the main result of the previous paper reflecting the new ingredient for the family No.~$19$.
Recall that, in this paper, $\mcG_i$ is the quasismooth member of the family No.~$i$ (without any other additional assumption).

\begin{Thm}
Let $X$ be a member of the family $\mcG_i$ with $i \in I_{F,cAx/2} \cup I_{F,cAx/4}$.
Then there is a Sarkisov link to the birational counterpart $X' \in \mcG'_i$.
Moreover, no curve and no nonsingular point on $X$ is a maximal center, and for each singular point $\msp$ of $X$, one of the following holds.
\begin{enumerate}
\item $\msp$ is not a maximal center.
\item There is a birational involution $\iota_{\msp} \colon X \ratmap X$ which is a Sarkisov link centered at $\msp$.
\item There is a Sarkisov link $\sigma_{\msp} \colon X \ratmap X'$ centered at $\msp$.
\end{enumerate}
\end{Thm}

\begin{proof}
This follows from \cite[Theorem 1.4]{Okada} except for members of the family $\mcG_{19}$ that do not satisfy \cite[Condition 3.1]{Okada}.
The proof for the exceptional members of $\mcG_{19}$ is given in Section \ref{sec:spfam19}.
\end{proof}

In this paper, we do the same thing for $X' \in \mcG'_i$ and prove the following.

\begin{Thm}
Let $X'$ be a member of $\mcG'_i$ for $i \in I_{F,cAx/2} \cup I_{F,cAx/4}$.
Then the maximal singularities on $X'$ are classified and the following assertions hold.
\begin{enumerate}
\item No curve and no nonsingular point is a maximal center.
\item If a terminal quotient singular point is a maximal center, then there is birational involution of $X'$ which is a Sarkisov link centered at $\msp$.
\item If $i \in I'$ $($resp.\ $i \in I''$$)$, then there exist a unique Sarkisov link $($resp.\ two Sarkisov links$)$ centered at the $cAx/2$ or $cAx/4$ point, and they are all Sarkisov links to the birational counterpart $X \in \mcG_i$ of $X'$.
\end{enumerate}
\end{Thm}

Finally, we explain why these results imply Theorem \ref{thm:mainbi} in a general setting.
The following follows from the Corti's proof \cite{Co95} of the more general statement, the Sarkisov program.

\begin{Lem}
Let $X = X_1,X_2, \dots,X_m$ be $\mbQ$-Fano $3$-folds which are birational to each other.
Assume that, for $i = 1,\dots,m$ and for each maximal singularity on $X_i$, there exists an untwisting birational map to $X_j$ for some $1 \le j \le m$ starting with the maximal singularity.
Then the birational Mori fiber structures of $X$ are precisely $X_1,\dots,X_m$.
\end{Lem}

\begin{proof}
Suppose that we are given a birational map $\psi \colon X=X_1 \ratmap V$ to a Mori fiber space $V/T$.
The aim is to show that $V$ is isomorphic to $X_j$ for some $j \in \{1,2,\dots,m\}$.
We fix a complete linear system $\mcH_V$ of a sufficiently ample divisor on $V$ such that $\mcH_V \sim_{\mbQ} - n_V K_V + D$, where $n_V > 0$ and $D$ is the pullback of an ample divisor on $T$.
Let $\mcH \sim_{\mbQ} n (-K_X)$ be the birational transform of $\mcH_V$ on $X$.
By the Noether-Fano-Iskovskikh inequality (cf.\ \cite[(4.2) Theorem]{Co95}) , we have $n \ge n_V$ and equality holding if and only if $\psi \colon X \ratmap V$ is an isomorphism.
We may assume that $\psi$ is not an isomorphism.
Then the pair $(X,\frac{1}{n} \mcH)$ is not canonical and there is a maximal singularity of $\mcH$ (cf.\ \cite[(2.10) Proposition-definition and (4.2) Theorem]{Co95}).
By the assumption, there is an untwisting birational map $\sigma_1 \colon X \ratmap X_{i_1}$ for some $i_1 \in \{1,2,\dots,m\}$ starting with the maximal singularity.
Set $\psi_1 = \psi \circ \sigma_1^{-1} \colon X_{i_1} \ratmap V$.
If $\psi_1$ is an isomorphism, then $V = X_{i_1}$.
Otherwise, the birational transform $\mcH_1 \sim_{\mbQ} n_1 (-K_{X_1})$ of $\mcH_V$ on $X_{i_1}$ has a maximal singularity and there is an untwisting birational map $\sigma_2 \colon X_{i_1} \ratmap X_{i_2}$ for some $i_2 \in \{1,2,\dots,m\}$.
Note that $n_1 < n$.
Repeating this procedure, we have a sequence $X \ratmap X_{i_1} \ratmap X_{i_2} \ratmap \cdots$ of untwisting birational maps.
Note that we have a sequence $n > n_1 > n_2 > \cdots$ of strictly decreasing rational numbers, where $n_j$ is such that $\mcH_j \sim_{\mbQ} n_j (-K_{X_j})$.
Since there are only finitely many $\mbQ$-Fano $3$-folds involved in this argument, the $n_j$'s have bounded denominators.
Thus, the above procedure ends in finitely many steps and this shows that $V \cong X_{i_k}$ for some $i_k \in \{1,2,\dots,m\}$.
This completes the proof.
\end{proof}

We explain the content of the paper.
In Section \ref{sec:Slinks}, we construct Sarkisov links centered at the $cAx/2$ or $cAx/4$ point $\msp'$ of $X'$ to the birational counterpart $X \in \mcG_i$ of $X'$ and show that those links exhaust Sarkisov links centered at $\msp'$.
Moreover, we construct birational involutions which are Sarkisov links centered at particular terminal quotient singular points of $X'$.
The remaining sections will be devoted to exclude the remaining centers as maximal singularity.
In Section \ref{sec:curve} and \ref{sec:nspts}, we exclude curves and nonsingular points on $X'$ as maximal center, respectively.
In Section \ref{sec:singpts}, we exclude terminal quotient singular points as maximal center, except for the points centered at which there is a birational involution, which completes the proof of the main theorem for most of the families.
We treat the remaining special centers on the families $\mcG'_{19}$ and $\mcG_{19}$ in Section \ref{sec:spinv}.
The main reason why we treat those special cases separately is that there is a birational involution centered at such a special center which is called an ``invisible birational involution" in \cite{CP}, whose construction is first given there and it is extremely complicated.

\section{Sarkisov links} \label{sec:Slinks}

We construct various Sarkisov links from members of $\mcG'_i$.

\subsection{Sarkisov links to $\mcG_i$} 
\label{sec:SLWHWCI}

Let $X' = X'_+$ be a member of the family $\mcG'_i$ with a standard defining equation, that is,
\[
X' = X'_+ =
\begin{cases}
(w^2 x_0 (x_0 + f) + w g + h = 0), & \text{if $i \in I'$}, \\
(w^3 x_0^2 + w^2 x_0 f + w g + h = 0), & \text{if $i \in I''$}
\end{cases}
\]
in $\mbP (a_0,\dots,a_3,b)$.
Note that $b = 2$ or $4$ depending on whether $i \in I_{F,cAx/2}$ or $i \in I_{F,cAx/4}$.
Let $X = X_+$ be the birational counterpart of $X'$, that is, it is a weighted complete intersection in $\mbP (a_0,\dots,a_4,a_5)$ with homogeneous coordinates $x_0,\dots,x_3,u$ and $v$ defined as
\[
X = X_+ := 
\begin{cases}
(v x_0 + u (x_0 + f) + g = v u - h = 0), & \text{if } i \in I', \\
(v x_0 + u^2 + u f + g = v u - h = 0), & \text{if } i \in I'',
\end{cases}
\]
where $a_4 = a_0 + b$, $a_5 = d-(a_0+b)$.
In \cite[Section 4.2]{Okada}, Sarkisov links between $X$ and $X'$ are constructed and we recall them from now on.
We define the weighted hypersurface in $\mbP (a_0,\dots,a_4)$ with homogeneous coordinates $x_0,\dots,x_3$ and $u$ as follows
\[
Z_+ := 
\begin{cases}
(u (u (x_0 + f) + g) + x_0 h = 0), & \text{if } i \in I', \\
(u (u^2 + u f + g) + x_0 h = 0), & \text{if } i \in I''.
\end{cases}
\]
We have the following Sarkisov link
\[
\xymatrix{
Y'_+ \ar[d]_{\varphi'_+} \ar[rd]^{\psi'_+} \ar@{-->}[rr]^{\tau_+} &  & Y_+ \ar[d]^{\varphi_+} \ar[ld]_{\psi_+} \\
X'_+ & Z_+ & X,}
\]
where $\varphi'_+$ is the weighted blowup of $X'$ with $\wt (x_0,x_1,x_2,x_3) = \frac{1}{b} (a_4,a_1,a_2,a_3)$ at the $cAx/2$ or $cAx/4$ point, $\varphi_+$ is the Kawamata blowup of $X_+$ at $\msp_5 \in X_+$, $\psi'_+$ and $\psi_+$ are flopping contractions, and $\tau_+$ is the flop. 
Let $\sigma_+ \colon X' = X'_+ \ratmap X$ be the induced birational map.

In the case where the singularity of $X'$ at the $cAx/2$ or $cAx/4$ point is of square type, there is another extremal divisorial contraction $\varphi'_- $ centered at the point.
We shall explain that $\varphi'_-$ leads to a Sarkisov link to the same variety $X$.

We first consider the case $i \in I'$.
We define
\[
X'_- := (w^2 (x_0 - f) x_0 + w g_- + h_- = 0) \subset \mbP (a_0,a_1,a_2,a_3,b),
\]
where 
\[
\begin{split}
g_- (x_0,x_1,x_2,x_3) &= g (x_0 - f,x_1,x_2,x_3), \\
h_- (x_0,x_1,x_2,x_3) &= h (x_0 - f,x_1,x_2,x_3).
\end{split}
\]
We see that $X'_-$ is a member of $\mcG'_i$ with a standard defining equation and there is an isomorphism $j' \colon X' \to X'_-$ defined by ${j'}^* x_0 = x_0 + f$, ${j'}^*x_i = x_i$ for $i = 1,2,3$ and ${j'}^*w = w$.
After identifying $X'$ and $X'_-$ via $j'$, there is a Sarkisov link
\[
\xymatrix{
Y'_- \ar@{-->}[rr]^{\tau_-} \ar[d]_{\varphi'_-} \ar[rd]^{\psi'_-} &  & Y_- \ar[d]^{\varphi_-} \ar[ld]_{\psi_-} \\
X'_- & Z_-& X_-,}
\]
whose construction is the same as that of $X'_+$.
We see that
\[
X_- = (v (x_0 - f) + u x_0 + g_- = v u - h_- = 0) \subset \mbP (a_0,\dots,a_5).
\] 
The induced birational map $\iota := \sigma_-^{-1} \circ \sigma_+ \colon X_+ \ratmap X_-$ is an isomorphism such that $\iota^* x_0 = x_0 + f$, $\iota^*x_i = x_i$ for $i = 1,2,3$, $\iota^*u = v$ and $\iota^*v = u$.
We identify $X = X_+$ and $X_-$ via the isomorphism $\iota$.
We see that divisorial extractions $\varphi'_+$ and $\varphi'_-$ lead to Sarkisov links to $X$.
Note that in the case where the $cAx/2$ or $cAx/4$ of $X'$ is of non-square type, which is equivalent to saying $f = 0$ as a polynomial, we have $g = g_-$ and $h = h_-$, hence two links $\sigma_{\pm}$ define the same link.
In this case, we set $\sigma := \sigma_+ = \sigma_- \colon X' \ratmap X$, $Y' = Y_+ = Y_-$ and $\varphi' = \varphi'_+ = \varphi'_-$.

We consider the case $i \in I''$.
Since $\deg g < 2 a_1$, we may write $g = x_1 (\alpha + x_0 g') + g''$, $h = x_1^2 + x_1 h' + h''$, where $\alpha \in \mbC [x_2,x_3]$ and $g',g'',h',h'' \in \mbC [x_0,x_2,x_3]$.
Without loss of generality we may assume $h' = 0$ by replacing $x_1 \mapsto x_1 - h'/2$.
Recall that the defining equation of $X'$ is $w^3 x_0^2 + w^2 x_0 f + w g + h = 0$.
We set $f_- := f - \alpha g'$ and $g_- = x_1 (-\alpha + x_0 g') + g''$.
 $\alpha_- := -\alpha + x_0 g'$ and $g''_- := g'' - \alpha h'$, and define
\[
X'_- :=
(w^3 x_0^2 + w^2 x_0 f_- + w g_- + h = 0) \subset \mbP (a_0,\dots,a_3,b).
\]
We see that $X'_-$ is a member of $\mcG'_i$ with a standard defining equation and there is an isomorphism $j' \colon X' \to X'_-$ defined by ${j'}^*x_1 = x_1 + w \alpha$, ${j'}^*x_i = x_i$ for $i = 0,2,3$ and ${j'}^*w = w$.
After identifying $X'$ and $X'_-$ via $j'$, there is a Sarkisov link
\[
\xymatrix{
Y'_- \ar@{-->}[rr]^{\tau_-} \ar[d]_{\varphi'_-} \ar[rd]^{\psi'_-} &  & Y_- \ar[d]^{\varphi_-} \ar[ld]_{\psi_-} \\
X'_- & Z_- & X_-,}
\]
whose construction is the same as that of $X'_+$.
We see that
\[
X_- = (v x_0 + u^2 + u f_- + g_- = v u - h = 0) \subset \mbP (a_0,\dots,a_5).
\]
We have an isomorphism $j \colon X = X_+ \to X_-$ defined by $j^*x_1 = - x_1 - u g'$, $j^*v = v + 2 x_1 g' + u {g'}^2$, $j^*x_i = x_i$ for $i = 0,2,3$ and $j^*u = u$.
We identify $X$ and $X_-$ via $j$ and denote by $\sigma_- \colon X' \ratmap X$ the Sarkisov link $\varphi_- \circ \tau_- \circ {\varphi'}^{-1} \circ j' \colon X' \cong X'_- \ratmap X$.
Note that in the case where the $cAx/2$ or $cAx/4$ point of $X'$ is of non-square type, which is equivalent to saying $\alpha = 0$ as a polynomial, we have $f_- = f$, $\alpha_- = \alpha$ and $\beta_- = \beta$, hence $X_+ = X_-$ and two links $\sigma_{\pm}$ define the same link. 
In this case we set $X := X_+ (= X_-)$, $\sigma := \sigma_+ (= \sigma_-) \colon X' \ratmap X$  and $\varphi' := \varphi'_+ (= \varphi'_-)$.
Therefore, we have the following result.

\begin{Prop}
Let $X'$ be a member of the family $\mcG'_i$ with $i \in I_{F,cAx/2} \cup I_{F,cAx/4}$, $X$ the birational counterpart of $X'$ which is a member of $\mcG_i$ and $\msp'$ the $cAx/2$ or $cAx/4$ singular point of $X'$.
\begin{enumerate}
\item If $\msp'$ is of square type, then there are two divisorial extractions $\varphi'_{\pm} \colon Y'_{\pm} \to X'$ centered at $\msp'$ and they lead to Sarkisov links $\sigma_{\pm} \colon X' \ratmap X$.
\item If $\msp'$ is of non-square type, then there is a unique divisorial extraction $\varphi' \colon Y' \to X'$ centered at $\msp'$ and it leads to the Sarkisov link $\sigma \colon X' \ratmap X$.
\end{enumerate}
\end{Prop}

\subsection{Birational involutions} 
\label{sec:birinv}

In this subsection we show that there is a birational involution of a member $X'$ of the family $\mcG'_i$, $i \in I_{cAx/2} \cup I_{F,cAx/4}$, which is a Sarkisov link centered at a terminal quotient singular point $\msp$ marked Q.I. or E.I. in the third column of the table.
We construct the anticanonical model $Z$ of $Y'$ and observe that the anticanonical map $Y' \ratmap Z$ is a morphism and that $Z$ is a double cover of a suitable weighted projective $3$-space.
This can be done by the same argument as that of \cite[Section 4]{CPR}.
The following lemma implies that this construction either excludes $\msp$ as a maximal center or leads to a Sarkisov link centered at $\msp$.

\begin{Lem} \label{lem:birinvSlink}
Let $X$ be a $\mbQ$-Fano variety with Picard number one and $\varphi \colon Y \to X$ an extremal divisorial extraction.
Assume that $-K_Y$ is nef and big but not ample, and there is a double cover $\pi \colon Z \to V$ from the anticanonical model $Z$ of $Y$ to a normal projective $\mbQ$-factorial variety $V$ with Picard number one.
Let $\iota_Z \colon Z \to Z$ be the biregular involution which interchanges the fibers of $\pi$ and $\iota_Y = \psi^{-1} \circ \iota_Z \circ \psi \colon Y \ratmap Y$ the birational involution induced by $\iota_Z$, where $\psi \colon Y \to Z$ is the anticanonical morphism.   
Then one of the following holds.
\begin{enumerate}
\item $\psi$ is divisorial.
In this case $\varphi$ is not a maximal singularity.
\item $\psi$ is small.
In this case $\iota_Y$ is the flop of $\psi$ and the diagram
\[
\xymatrix{
Y \ar[d]_{\varphi} \ar[rd]^{\psi} \ar@{-->}[rrr]^{\iota_Y} & & & Y \ar[d]^{\varphi} \ar[ld]_{\psi} \\
X & Z \ar[r]^{\iota_{Z}} & Z & X}
\]
gives a Sarkisov link $\iota := \varphi \circ \iota_Y \circ \varphi^{-1} \colon X \ratmap X$ starting with $\varphi$.
\end{enumerate}
\end{Lem}

\begin{proof}
In the above statement, we implicitly use the fact that $-K_Y$ is semiample.
This holds true by the base point free theorem (cf.\ \cite[Theorem 3.3]{KM}) since $-K_Y$ is nef and big. 
For this reason, we have the anticanonical morphism $\psi \colon Y \to Z$.

We see that $\psi$ is not an isomorphism and is either divisorial or small.
Assume that $\psi$ is divisorial.
Let $E$ be the exceptional divisor of $\varphi$ and $C$ an irreducible curve on $Y$ which is contracted by $\psi$.
Then we have $(-K_Y \cdot C) = 0$.
In particular, $C$ is not contracted by $\varphi$ since $\varphi$ is a $K_Y$-negative contraction.
It follows that 
\[
a_E (K_X) (E \cdot C) = (-\varphi^* K_X \cdot C) + (K_Y \cdot C)
= (-K_X \cdot \varphi_* C) > 0.
\]
Thus $(E \cdot C) > 0$.
This shows that there are infinitely many curves on $Y$ which intersect $-K_Y$ non-positively and $E$ positively.
By Lemma \ref{crispnegdef}, $\msp$ is not a maximal singularity.

Assume that $\psi$ is small.
We need to show that $\iota_Y$ is indeed the flop of $\psi$, which is equivalent to showing that $\iota_Y$ is not biregular.
Note that the Picard group of $Y$ is generated by $-K_Y$ and $E$.
Since $\iota_Y$ is an isomorphism in codimension one, we have $\iota_Y^* (-K_Y) = -K_Y$.
Hence, by Lemma \ref{lem:isomrk2}, $\iota_Y$ is an isomorphism if and only if $\iota_Y^* (E) = E$.
Set $E_Z = \psi (E) \subset Z$.
We have $E_Z + \iota_Z^* E_Z \sim_{\mbQ} \pi^* D$ for some non-zero effective divisor $D$ on $V$.
It follows that $\psi^* \pi^* D = - l K_Y$ for some $l > 0$ since $\rho (V) = \rho (Z) = 1$ and $\psi$ is defined by $- m K_Y$ for some $m > 0$.
Thus 
\[
\iota_Y^* (E) = \psi^* (\iota_Z^* (E_Z)) \sim_{\mbQ} \psi^* \pi^*D - \psi^*E_Z
= - l K_Y - E.
\]
This shows that $\iota_Y$ cannot be biregular and thus $\iota_Y$ is the flop.
\end{proof}

\begin{Thm} \label{birinv}
Let $X'$ be a member of the family $\mcG'_i$ and $\msp$ a terminal quotient singular point labeled Q.I. in the third column.
For the family $\mcG'_{30}$, we assume that the monomial $y^2 z$ appears in the defining equation of $X'$.
Then there exists a birational involution $\iota'_{\msp}$ of $X'$ which is a Sarkisov link centered at $\msp$.
\end{Thm}

\begin{proof}
Let $\mbP (a_0,\dots,a_4)$ be the ambient space of $X'$ with homogeneous coordinates $x_0,\dots,x_4$.
After replacing coordinates, we can assume $\msp = \msp_4$ and the defining polynomial of $X'$ is of the form $x_4^2 x_3 + x_4 f + g = 0$ for some $f, g \in \mbC [x_0,x_1,x_2,x_3]$.
Moreover we see that $x_0,x_1,x_2,x_3$ lift to anticanonical sections of $Y'$, where $Y' \to X'$ is the Kawamata blowup of $X'$ at $\msp$.
Then we have the anticanonical morphism $\psi \colon Y' \to Z$, where 
\[
Z = (y^2 + y f + g = 0) \subset \mbP (a_0,a_1,a_2,a_3,b),
\]
where $b = a_3 + a_4$.
Here $y$ is the coordinate of degree $b$ and $\psi$ is given by the identification $y = x_4 x_3$.
Note that $Z$ is a double cover of $\mbP (a_0,a_1,a_2,a_3)$.
We refer the readers to \cite[Section 4.9]{CPR} for the above construction.

We shall show that $\psi$ is small.
We see that $\psi$ contracts the proper transform of $(x_3 = f = g = 0) \subset X$.
Assume that $\psi$ contracts a divisor.
Then there is a homogeneous polynomial $h \in \mbC [x_0,x_1,x_2]$ which divides both $f (x_0,x_1,x_2,0)$ and $g (x_0,x_1,x_2,0)$.
Then the defining polynomial of $X'$ can be written as $x_3 F + h G$ for some $F, G \in \mbC [x_0,\dots,x_4]$.
This is a contradiction since $X'$ is $\mbQ$-factorial.
Therefore $\psi$ is small and there is a birational involution of $X'$ which is a Sarkisov link centered at $\msp$ by Lemma \ref{lem:birinvSlink}.
\end{proof}

\begin{Rem}
In the case where $y^2 z$ does not appear in the defining polynomial of $X' \in \mcG'_{30}$, the corresponding point is excluded as maximal center in Proposition \ref{prop:G'30-3sp}. 
\end{Rem}

Let $X' = X'_8 \subset \mbP (1,1,2,3,2)$ be a member of $\mcG_{19}$ and $\msp$ a point of type $\frac{1}{2} (1,1,1)$.
This is the unique singular point marked E.I. in the table.
The defining polynomial of $X'$ is of the form $F = w^2 y (y + f_2) + w g_6 + h_8$, where $f_2, g_6, h_8 \in \mbC [x_0,x_1,y,z]$.
First, we suppose that there is no WCI curve on $X'$ of type $(1,1,2)$ passing through $\msp$, which is equivalent to $y z^2 \in h_8$.
In this case, we shall construct a birational involution centered at $\msp$.
By a suitable change of coordinates $y,z,w$, we can assume $\msp = \msp_2$ and write
\[
F = y z^2 + a_5 z - w y^3 - b_4 y^2 - c_6 y + d_8,
\]
where $a_5, b_4, c_6, d_8 \in \mbC [x_0,x_1,w]$. 
Let $\varphi \colon Y' \to X'$ be the Kawamata blowup of $X'$ at $\msp$.
Then, by \cite[Theorem 4.13]{CPR}, the sections
\[
u := z^2 - w y^2 - b_4 y - c_6 \text{ and } 
v := u z + a_5 w y + a_5 b_4 
\]
lift to plurianticanonical sections of $Y'$.
Moreover, the anticanonical model $Z'$ of $Y'$ is the weighted hypersurface defined by the equation
\[
G := -v^2 + a_6 b_4 v + u^3 + u^2 c_6 - (b_4 d_8 + a_5^2 w)v + (-a_5^2 c_6 + d_8^2) w = 0
\]
in the weighted projective space $\mbP (1,1,2,6,9)$ with coordinates $x_0$, $x_1$, $w$, $u$, $v$, and the corresponding map $\psi \colon Y' \ratmap Z'$ is a morphism.
$Z'$ is a double cover of $\mbP (1,1,2,6)$.
Thus we can apply Lemma \ref{lem:birinvSlink} and obtain the following.

\begin{Thm} \label{thm:19EI}
Let $X'$ be a member of $\mcG'_{19}$ and $\msp$ a singular point of type $\frac{1}{2} (1,1,1)$.
Assume that there is no WCI curve of type $(1,1,2)$ on $X'$ passing through $\msp$.
Then either $\msp$ is not a maximal center or there is a birational involution $\iota \colon X' \ratmap X'$ centered at $\msp$ which is a Sarkisov link.
\end{Thm}

\begin{proof}
This follows from Lemma \ref{lem:birinvSlink} and the preceding argument.
\end{proof}

\begin{Rem}
The case where there is a WCI curve of type $(1,1,2)$ on $X'$ passing through $\msp$ will be treated in Section \ref{sec:spfam19}.
\end{Rem}

\section{Excluding curves} 
\label{sec:curve}

In this section we exclude curves as maximal center.

\subsection{Most of the curves}

We can exclude most of the curves by Lemma \ref{exclcurvelowdeg}.

\begin{Prop} \label{exclmostC}
Let $X'$ be a member of the family $\mcG'_i$ for $i \in I_{F,cAx/2} \cup I_{F,cAx/4}$ and $\Gamma$ an irreducible and reduced curve on $X'$.
Then $\Gamma$ is not a maximal center except possibly for the following cases.
\begin{itemize}
\item $X' \in \mcG'_{17}$ or $\mcG'_{19}$, and $\deg \Gamma = \frac{1}{2}$.
\item $X' \in \mcG'_{23}$ and $\deg \Gamma = \frac{1}{4}$.
\end{itemize}
Here, in the above exceptions, $\Gamma$ passes through the $cAx/2$ or $cAx/4$ point of $X'$ and does not pass through any terminal quotient singular point of $X'$.
\end{Prop}

\begin{proof}
Assume that $\Gamma$ is a maximal center.
By Lemma \ref{exclcurvelowdeg}, $\deg \Gamma < (A^3)$.
If $\Gamma$ is contained in the nonsingular locus of $X'$, then $\deg \Gamma \ge 1$ and this implies $1 < (A^3)$.
This cannot happen since $(A^3) \le 1$ for any $X' \in \mcG_i$ with $i \in I_{F,cAx/2} \cup I_{F,cAx/4}$ satisfies $(A^3) \le 1$. 
Thus $\Gamma$ passes through a singular point.
We see that $\Gamma$ does not pass through a terminal quotient singular point since there is no extremal divisorial contraction centered along a curve through a terminal quotient singular point.
It follows that $\Gamma$ passes through the $cAx/2$ (resp.\ $cAx/4$) point but does not pass through a terminal quotient singular point, and this implies that $\deg \Gamma \in \frac{1}{2} \mbZ$  (resp.\ $\deg \Gamma \in \frac{1}{4} \mbZ$).
As a summary, we have $\frac{1}{2} \le \deg \Gamma < (A^3)$ (resp.\ $\frac{1}{4} \le \deg \Gamma < (A^3)$) if $i \in I_{F,cAx/2}$ (resp.\ $i \in I_{F,cAx/4}$).
This happens precisely when $\Gamma$ is one of the exceptional curves in the statement.
\end{proof}

\subsection{The remaining curves}

Let $X'$ be a member of $\mcG'_{19}$ (resp.\ $\mcG'_{23}$), and $\Gamma \subset X'$ a curve of degree $1/2$ (resp.\ $1/4$)  passing through the $cAx/2$ (resp.\ $cAx/4$) point.
We shall show that $\Gamma$ is not a maximal center.

The ambient weighted projective space of $X'$ is $\mbP (1,1,2,3,b)$, with  coordinates $x_0,x_1,y,z$ and $w$, where $b = 2$ (resp.\ $b = 4$) if $X' \in \mcG'_{19}$ (resp.\ $X' \in \mcG'_{23}$).
The curve $\Gamma$ is contracted to a point by the projection $\mbP (1,1,2,3,b) \ratmap \mbP (1,1,2,3)$ to the first four coordinates, so that it is contained in $(y = 0) \cup (y + f_2 = 0)$ (resp.\ $(x_0 = 0) \cup (x_0 + f_1=0)$).
Replacing $y \mapsto y - f_2$ (resp.\ $x_0 \mapsto x_0 - f_1$) and other coordinates $x_1,y,z$ if necessary, we may assume $\Gamma = (x_0 = y = z = 0)$.

We first consider a member
\[
X' = (w^2 y (y+f_2) + w g_6 + h_8 = 0) \subset \mbP (1,1,2,3,2)
\]
of $\mcG'_{19}$, where $f_2 = f_2 (x_0,x_1)$, $g_6 = g_6 (x_0,x_1,y,z)$ and $h_8 = h_8 (x_0,x_1,y,z)$.
We put $b_3 := (\prt g_6/\prt z) (x_0,x_1,0,0)$ and $c_6 := g_6 (x_0,x_1,0,0)$.
Since $X'$ contains $\Gamma = (x_0 = y = z = 0)$,  $h_8 (0,x_1,0,0) = 0$ and $c_6$ is divisible by $x_0$. 
We write $c_6 = x_0 c_5$.

\begin{Lem}
At least one of $f_2$, $b_3$ and $c_5$ is not divisible by $x_0$.
\end{Lem}

\begin{proof}
Let $X = X_{6,8} \subset \mbP (1,1,2,3,4,4)$ be the member of $\mcG_{19}$ which is the birational counterpart of $X'$.
Defining polynomials of $X$ are written as $F_1 = s_1 y + s_0 (y+f_2) + g_6$ and $F_2 = s_1 s_0 - h_8$, where $s_0,s_1$ are coordinates of degree $4$.
We see that $X$ contains the curve $C := (x_0 = y = z = s_0 = 0)$.
If all of $f_2$, $b_3$ and $c_5$ are divisible by $x_0$, then the restriction of $\prt F_1/ \prt x_0$, $\prt F_1 / \prt x_1$, $\prt F_1/\prt z$, $\prt F_1/ \prt s_0$ and $\prt F_1 / \prt s_1$ to $C$ are identically zero.
It follows that $X$ is not quasismooth at the point $C \cap (\prt F_1/\prt y = 0)$.
This is a contradiction since $X$ is quasismooth.
\end{proof}

Let $\mcM \subset \left| \mcI_{\Gamma} (3A) \right|$ be the linear system on $X'$ generated by the sections $x_0^3$, $x_0^2 x_1$, $x_0 x_1^2$, $y x_0$, $y x_1$, $z$, and let $S$ be a general member of $\mcM$.
The base locus of $\mcM$ is the union of $\Gamma$ and the set of points of type $\frac{1}{2} (1,1,1)$.
We denote by $\msp = \msp_4$ the $cAx/2$ point of $X'$.

\begin{Lem}
The surface $S$ is nonsingular along $\Gamma \setminus \{\msp\}$.
\end{Lem}

\begin{proof}
We write 
\[
f_2 = \alpha x_0 x_1 + \beta x_1^2 + f_2', 
g_6 = \gamma x_0 x_1^5 + \delta y x_0^4 + \varepsilon z x_1^3 + g'_6,
h_8 = \eta x_0 x_1^7 + \zeta y x_1^6 + \theta z x_1^5 + h'_8,
\]
where $\alpha,\beta,\dots,\theta \in \mbC$ and $f_2', g'_6, h'_8$ are contained in the ideal $(x_0,y,z)^2$.
We see that $S$ is cut out on $X'$ by the section of the form $\lambda x_0 x_1^2 + \mu y x_1 + \nu z + (\text{other terms})$.
Then the restriction of Jacobian matrix of the affine cone of $S$ to $\Gamma$ can be computed as
\[
J_{C_S}|_{\Gamma} =
\begin{pmatrix}
\gamma w x_1^5 + \eta x_1^7 & 0 & \beta w^2 x_1^2 + \delta w x_1^4 + \zeta x_1^6 & \varepsilon w x_1^3 + \theta x_1^5 & 0 \\
\lambda x_1^2 & 0 & \mu x_1 & \nu & 0
\end{pmatrix}
\]
It is enough to show that $J_{C_S}$ is of rank $2$ at every point of $\Sigma := \Gamma \setminus \{\msp\}$.

Suppose that $\varepsilon = \theta = 0$.
Then the 1st row of $J_{C_S} (\msq)$ is of rank $1$ for any $\msq \in \Sigma$ since $X'$ is quasismooth at $\msq$.
Thus $\rank J_{C_S} (\msq) = 2$  for any $\msq \in \Sigma$ and we are done.

We continue the proof assuming that $(\varepsilon,\theta) \ne (0,0)$. 
We set
\[
M_1 =
\begin{pmatrix}
\gamma w x_1^5 + \eta x_1^7 & \varepsilon w x_1^3 + \theta x_1^5 \\
\lambda x_1^2 & \nu
\end{pmatrix},
M_2 =
\begin{pmatrix}
\beta w^2 x_1^2 + \delta w x_1^4 + \zeta x_1^6 & \varepsilon w x_1^3 + \theta x_1^5 \\
\mu x_1 & \nu
\end{pmatrix}
\]
Then, it is enough to show that the set $(\det M_1 = \det M_2 = 0) \cap \Sigma$ is empty.
We have
\[
\begin{split}
\det M_1 &= x_1^5 ((\gamma \nu - \varepsilon \lambda) w - (\theta \lambda - \eta \nu) x_1^2), \\
\det M_2 &= x_1^2 (\beta \nu w^2 + (\delta \nu - \varepsilon \mu) w x_1^2 + (\zeta \nu - \theta \mu) x_1^4).
\end{split}
\]
Note that $x_1 \ne 0$ on $\Sigma$.
If $(\gamma,\varepsilon) = (0,0)$, then the set $(\det M_1 = 0) \cap \Sigma$ is already empty.
Hence we may assume $(\gamma,\varepsilon) \ne (0,0)$.
In this case, we may assume $\gamma \nu - \varepsilon \lambda \ne 0$ since $\lambda,\nu$ are general.
We set $w_0 = (\theta \lambda - \eta \nu)/(\gamma \nu - \varepsilon \lambda) \in \mbC$ and $\msq = (0\!:\!1\!:\!0\!:\!0\!:\!w_0)$.
Then we see that $(\det M_1 = 0) \cap \Sigma = \{\msq\}$.
We have
\[
(\det M_2) (1,w_0) = \mu (\varepsilon w_0 + \theta) - \nu (\beta w_0^2 + \delta w_0 + \zeta).
\] 
Since $w_0$ does not involve $\mu$, we see that $(\det M_2) (1,w_0) \ne 0$ for a general $\mu$.
This shows $(\det M_1 = \det M_2 = 0) \cap \Sigma = \emptyset$ and the proof is completed.
\end{proof}

\begin{Lem} \label{selfint19}
We have $(\Gamma^2) \le - 3/2$.
\end{Lem}

\begin{proof}
We work on the open subset on which $w \ne 0$.
Let $\Phi$ be the weighted blowup of $\mbP (1,1,2,3,2)$ at $\msp$ with $\wt (x_0,x_1,y,z) = \frac{1}{2} (1,1,4,3)$ and $\varphi \colon T \to S$ the birational morphism induced by $\Phi$.
Let $\tilde{\Gamma}$ be the proper transform of $\Gamma$ by $\varphi$.
We denote by $E$ the restriction of the exceptional divisor of $\Phi$ to $T$.

We claim that $E = E_1 + E_2$, where $E_1$ is an irreducible and reduced curve, $E_2$ does not contain $E_1$ as a component, $(\tilde{\Gamma} \cdot E_1) = 1$ and $\tilde{\Gamma}$ is disjoint from $E_2$.
Let $s = s (x_0,x_1,y,z)$ be the section which cuts out $S$.
We can write $s = z + x_0 q + (\text{terms divisible by } y)$, where $q = q (x_0,x_1)$ is a quadric.
We have the isomorphisms
\[
\begin{split}
E &\cong (y f_2 + z^2 + z b_3 + x_0 c_5 = z + x_0 q = 0) \subset \mbP (1,1,4,3), \\
&\cong (y f_2 + x_0^2 q^2 - x_0 q b_3 + x_0 c_5 = 0) \subset \mbP (1,1,4).
\end{split}
\]
If $f_2$ is not divisible by $x_0$, then $E = E_1$ is an irreducible and reduced curve and it intersects $\tilde{\Gamma}$ transversally at a nonsingular point.
Assume that $f_2$ is divisible by $x_0$ and write $f_2 = x_0 f_1$.
In this case we have $E = E_1 + E_2$, where $E_1 = (x_0 = 0)$ and $E_2 = (y f_1 + x_0 q^2 - q b_3 + c_5 = 0)$.
Note that $E$ intersects $\tilde{\Gamma}$ transversally at a nonsingular point.
Since one of $b_3$ and $c_5$ is not divisible by $x_0$, the support of $E_2$ does not contain $E_1$ as a component and $E_2$ is disjoint from $\tilde{\Gamma}$.
It follows that $(E_1 \cdot \tilde{\Gamma}) = 1$.
This proves the claim.

We write $\varphi^*\Gamma = \tilde{\Gamma} + r E_1 + F$ for some rational number $r$ and an effective $\mbQ$-divisor $F$ whose support is contained in $\operatorname{Supp} E_2$.
We claim that $r \le 1/2$.
Indeed, we have $(x_0 = 0)_S = \Gamma + \Delta$ for some curve $\Delta$ and $x_0$ vanishes along $E$ to order $1/2$.
It follows that $\varphi^* (x_0 = 0)_S = \varphi^*\Gamma + \varphi^*\Delta$ and the coefficient of $E_1$ in $\varphi^*(x_0 = 0)_S$ is $1/2$, which shows $r \le 1/2$.
The induced birational morphism $\Phi|_{Y'} \colon Y' \to X'$, where $Y'$ is the proper transform of $X'$ via $\Phi$, is an extremal divisorial extraction centered at $\msp$ (see Section \ref{sec:SLWHWCI}).
In particular $K_{Y'} = (\Phi|_{Y'})^*K_{X'} + \frac{1}{2} G$, where $G$ is the exceptional divisor of $\Phi|_{Y'}$.
We have $(\Phi||_{Y'})^* S = T - \frac{3}{2} G$ and $G|_T = E$.
By adjunction, we conclude $K_T = \varphi^*K_S - E$.
We have 
\[
(\Gamma^2) = (\tilde{\Gamma}^2) + (r E_1 + F \cdot \tilde{\Gamma}) = (\tilde{\Gamma}^2) + r
\]
and, by adjunction,
\[
(\tilde{\Gamma}^2) = - (K_T \cdot \tilde{\Gamma}) -2 = - (K_S \cdot \Gamma) + (E \cdot \tilde{\Gamma}) - 2.
\]
Combining these with $(K_S \cdot \Gamma) = 2 \deg \Gamma = 1$, we get $(\Gamma^2) = -2 + r \le - 3/2$.
\end{proof}

We consider a member 
\[
X' = (w^2 x_0 (x_0 + f_1) + w g_6 + h_{10} = 0) \subset \mbP (1,1,2,3,4)
\]
of $\mcG'_{23}$, where $f_1 = f_1 (x_1)$ and $g_6, h_{10} \in \mbC [x_0,x_1,y,z]$.
We write $g_6 (0,x_1,y,z) = z^2 + z b_3 + b_6$, where $b_j = b_j (x_1,y)$.
Since $X'$ contains $\Gamma = (x_0 = y = z = 0)$, $h_{10} (0,x_1,0,0) = b_6 (x_1,0) = 0$ and we write $b_6 = y b_4$.

\begin{Lem} \label{divy23}
If $f_1 = 0$, then at least one of $b_3$ and $b_4$ is not divisible by $y$.
\end{Lem}

\begin{proof}
Let $X \subset \mbP (1,1,2,3,5,5)$ be a member of $\mcG_{23}$ which is the birational counterpart of $X'$.
Defining polynomials of $X$ are written as $F_1 = s_1 x_0 + s_0 (x_0 + f_1) + g_6$ and $F_2 = s_1 s_0 - h_{10}$.
We see that $X$ contains the curve $C := (x_0 = y = z = s_0 = 0)$.
Assume that $f_1 = 0$ and both $b_3$ and $b_4$ are divisible by $y$.
We see that $(\prt F_1/\prt x_0 = 0) \cap C$ consists of a single point at which $X$ is not quasismooth.
This is a contradiction since $X$ is quasismooth.
\end{proof}

Let $\mcM$ be the linear system $\left| \mcI_{\Gamma} (3 A) \right|$ on $X'$.
The base locus of $\mcM$ is the union of $\Gamma$ and finitely many closed points as a set.
It is straightforward to see that $S$ is nonsingular along $\Gamma \setminus \{\msp\}$, where $\msp = \msp_4$ is the $cAx/4$ point of $X'$.

\begin{Lem} \label{selfint23}
We have $(\Gamma^2) \le - 1$.
\end{Lem}

\begin{proof}
We work on the open set on which $w \ne 0$.
Let $\Phi$ be the weighted blowup of $\mbP (1,1,2,3,4)$ at $\msp$ with $\wt (x_0,x_1,y,z) = \frac{1}{4} (5,1,2,3)$ and $\varphi \colon T \to S$ the birational morphism induced by $\Phi$.
Let $\tilde{\Gamma}$ be the proper transform of $\Gamma$ by $\varphi$.
We denote by $E$ the restriction of the exceptional divisor of $\Phi$ to $T$.

We shall show that $E = E_1 + E_2$, where $E_1$ is an irreducible and reduced curve, $E_2$ does not contain $E_1$ as a component, $(\tilde{\Gamma} \cdot E_1) = 1$ and $\tilde{\Gamma}$ is disjoint from $E_2$.
The section $s$ which cut out $S$ can be written as $s = z + \alpha y x_1 + (\text{terms divisible by } x_0)$, where $\alpha \in \mbC$.
We have the following isomorphisms:
\[
\begin{split}
E &\cong (x_0 f_1 + z^2 + z b_3 + y b_4 = z + \alpha y x_1 = 0) \subset \mbP (5,1,2,3) \\
&\cong (x_0 f_1 + \alpha^2 y^2 x_1^2 - \alpha y x_1  b_3 + y b_4 = 0) \subset \mbP (5,1,2).
\end{split}
\]
Assume first that $f_1 \ne 0$.
In this case we may assume $f_1 = x_1$ after re-scaling $x_1$.
If $b_4$ is not divisible by $x_1$, then $E \subset \mbP (5,1,2)$ is irreducible and reduced.
In this case, we put $E = E_1$ and $E_2 = 0$.
If $b_4$ is divisible by $x_1$, then we can write $b_4 = x_1^2 b_2$ and we have $E = E_1 + E'$, where $E_1 = (x_0 + \alpha^2 y^2 x_1 - \alpha y b_3 + y x_1 b_2 = 0)$ and $E_2 = (x_1 = 0)$.
Assume next that $f_1 = 0$.
In this case we have $E = E_1 + E_2$, where $E_1 = (y = 0)$ and $E_2 = (\alpha^2 y x_1^2 - \alpha x_1 b_3 + b_4 = 0)$.
Since at least one of $b_3$ and $b_4$ is not divisible by $y$ by Lemma \ref{divy23}, $E'$ does not contain $E_1$ as a component and $E_2$ is disjoint from $\Delta$.
In any of the above cases, $\tilde{\Gamma}$ is disjoint from $E_2$ and it intersects with $E_1$ transversally at one point which is a nonsingular point of $T$.
It follows that $(\tilde{\Gamma} \cdot E_1) = 1$.

We write $\varphi^*\Gamma = \tilde{\Gamma} + r E_1 + F$ for some rational number $r$ and an effective $\mbQ$-divisor $F$ whose support is contained in $\operatorname{Supp} E_2$.
Note that $(\tilde{\Gamma} \cdot F) = 0$ by the above argument.
We have $r \le 1/2$ since the section $y$ on $S$ cut out $\Gamma$ and some other curve, and $y$ vanishes along $E_1$ to order $2/4 = 1/2$.
The morphism $\Phi|_{Y'} \colon Y' \to X'$, where $Y'$ is the proper transform of $X'$ via $\Phi$, is an extremal divisorial extraction centered at $\msp$ and we have $K_{Y'} = (\Phi|_{Y'})^* K_{X'} + \frac{1}{4} G$, where $G$ is the exceptional divisor of $\Phi|_{Y'}$.
Since $(\Phi|_{Y'})^*S = T - \frac{1}{2} G$, we have $K_T = \varphi^*K_S - \frac{1}{2} E$ by adjunction.
We have 
\[
\begin{split}
(\Gamma^2) &= (\varphi^*\Gamma)^2 = (\tilde{\Gamma} \cdot \varphi^*\Gamma) = (\tilde{\Gamma}^2) + (r E_1 + F \cdot \tilde{\Gamma}) = (\tilde{\Gamma}^2) + r (E_1 \cdot \tilde{\Gamma}), \\
(\tilde{\Gamma}^2) &= - (K_T \cdot \tilde{\Gamma}) - 2 = - (K_S \cdot \Gamma) + \frac{1}{2} (E \cdot \tilde{\Gamma}) - 2,
\end{split}
\]
and $(K_S \cdot \Gamma) = 2 \deg \Gamma = 1/2$.
Thus, we have $(\Gamma^2) = -3/2 + r \le - 1$.
\end{proof}

\begin{Prop} \label{exclC1923}
Let $X'$ be a member of $\mcG'_{19}$ $($resp.\ $\mcG'_{23}$$)$ and $\Gamma$ a curve of degree $1/2$ $($resp.\ $1/4$$)$ passing through the $cAx/2$ $($resp.\ $cAx/4$$)$ point.
Then $\Gamma$ is not a maximal center.
\end{Prop}

\begin{proof}
Let $\mcH \subset |n A|$ be a movable linear system on $X'$.
We can choose a general $S \in \mcM$ so that we have
\[
A|_S \sim \frac{1}{n} \mcH |_S = \frac{1}{n} \mcL + \gamma \Gamma,
\]
where $\mcL$ is a movable linear system on $S$ and $\gamma \ge 0$.
We shall prove $\gamma \le 1$.
We may assume that $\gamma > 0$ because otherwise there is nothing to prove.
We have $(\mcL^2) \ge 0$ since $\mcL$ is nef.
We get
\[
(\mcL^2) = (A|_S - \gamma \Gamma)^2 = 3 (A^3) - 2 (\deg \Gamma) \gamma + (\Gamma^2) \gamma^2.
\]
To show that $\gamma \le 1$, it is enough to show that
\[
3 (A^3) - 2 (\deg \Gamma) + (\Gamma^2) \le 0
\]
since $(\Gamma^2) < 0$, $\deg \Gamma > 0$ and $(\mcL^2) \ge 0$.
By Lemmas \ref{selfint19} and \ref{selfint23}, the left-hand side of the last displayed equation can be computed as 
\[
3 (A^3) - 2 (\deg \Gamma) + (\Gamma^2) \le
\begin{cases}
-1/2, & \text{if } X' \in \mcG'_{19}, \\
-1/4, & \text{if } X' \in \mcG'_{23}.
\end{cases}
\]
This shows that $\gamma \le 1$ and thus $\Gamma$ is not a maximal center.
\end{proof}

In the rest of this subsection, we consider a member $X' = X'_8 \subset \mbP (1,1,1,4,2)$ of $\mcG'_{17}$ and let $\Gamma \subset X$ be a curve of degree $1/2$ passing through the $cAx/2$ point $\msp = \msp_4$.
Let $S$ and $S'$ be general members of the pencil $\left| \mcI_{\Gamma} (1) \right|$.
We see that $X'$ is defined by the equation
\[
w^3 x_0^2 + w^2 x_0 f_3 + w (y b_2 + b_6) + y^2 + y c_4 + c_8 = 0
\]
in $\mbP (1^3,4,2)$, where $f_i, b_i, c_i \in \mbC [x_0,x_1,x_2]$.
By choice of coordinates, we may assume $\Gamma = (x_0 = x_1 = y = 0) \subset X$, which is equivalent to the condition that $b_6, c_8 \in (x_0,x_1)$.

\begin{Lem} \label{lem:intspcurve17}
\begin{enumerate}
\item If either $b_2 \notin (x_0,x_1)$ or $c_4 \notin (x_0,x_1)$, then the scheme-theoretic intersection $S \cap S'$ consists of the union of two distinct irreducible and reduced curves $\Gamma$ and $\Delta$, both of which are smooth rational curves of degree $1/2$ passing through $\msp_4$.
\item If $b_2, c_4 \in (x_0,x_1)$, then $S' |_S = 2 \Gamma$ as a divisor on $S$.
\end{enumerate}
\end{Lem}

\begin{proof}
The intersection $S \cap S'$ is the scheme $(x_0 = x_1 = 0)_X$, which is isomorphic to
\[
(w y \bar{b}_2 + y^2 + y \bar{c}_4 = 0) \subset \mbP (1,4,2),
\]
where $\bar{b}_2 = b_2 (0,0,x_2)$ and $\bar{c}_4 = c_4 (0,0,x_2)$.
It follows that $S \cap S'$ is the union of $\Gamma$ and $\Delta = (x_0 = x_1 = w \bar{b}_2 + y + \bar{c}_4 = 0)$.
We see that $\Gamma = \Delta$ if and only if $\bar{b}_2 = \bar{c}_4 = 0$ as a polynomial in $x_2$.
The assertions follow immediately from this observation. 
\end{proof}

\begin{Lem} \label{selfint17}
If either $b_2 \notin (x_0,x_1)$ or $c_4 \notin (x_0,x_1)$, then we have $(\Gamma \cdot \Delta) \ge \deg \Delta$.
\end{Lem}

\begin{proof}
After replacing homogeneous coordinates, we may assume that $S$ is cut out on $X$ by the section $x_1$.
Thus,
\[
S = (w^3 x_0^2 + w^2 x_0 f'_3 + w (y b'_2 + b'_6) + y^2 + y c'_4 + c'_8 = 0) \subset \mbP (1,1,4,2),
\]
where $f'_3 = f_3 (x_0,0,x_2)$, $b'_i = b_i (x_0,0,x_2)$ and $c'_i = c_i (x_0,0,x_2)$.
The section $x_0$ cuts out on $S$ two curves $\Gamma$ and $\Delta = (x_0 = x_1 = w b_2 + y + c_4 = 0)$.
We set $\bar{f}_3 = f_3 (0,0,x_2)$ and $\bar{b}_2 = b_2 (0,0,x_2)$.
If $\bar{b}_2 \ne 0$, then $\Gamma$ intersects $\Delta$ at a nonsingular point and thus $(\Gamma \cdot \Delta) \ge 1 > \deg \Gamma$.
In the following we assume $\bar{b}_2 = 0$.
The surface $S$ is nonsingular outside $\msp$.
The singularity of $S$ at $\msp$ is analytically equivalent to $(x_0^2 + x_1^m + y^2 = 0)/\mbZ_2 (1,1,0)$, where $m = 6$ or $8$ depending on whether $\bar{f}_3 \ne 0$ or $\bar{f}_3 = 0$, and it is a du Val singularity of type $D_m$. 

Let $\varphi \colon T \to S$ be the weighted blowup of $S$ at $\msp$ with weight $\wt (x_0,x_2,y) = \frac{1}{2} (3,1,4)$ and $E$ its exceptional divisor.
We see that $K_T = \varphi^*K_S$ and there is an isomorphism
\[
E \cong (x_0^2 + x_0 \bar{f}_3 = 0) \subset \mbP (3,1,4).
\] 
Let $\msq$ be the point $(0 \!:\! 0 \!:\! 1) \in E \cong \mbP (3,1,4)$.
We have $E = E_1 + E_2$, where $E_1 = (x_0 = 0)$ and $E_2 = (x_0 + \bar{f}_3 = 0)$.
Note that if $\bar{f}_3 = 0$, then $E_1 = E_2$.
We see that $T$ is smooth outside the point $\msq$.
The proper transforms $\tilde{\Gamma}$ and $\tilde{\Delta}$ of $\Gamma$ and $\Delta$, respectively, avoid the point $\msq$, are disjoint from $E_2$ and intersect $E_1$ transversally at one point.

We compute the intersection number $(\Gamma \cdot \Delta)$.
Since $A|_S \sim \Gamma + \Delta$, we have $1/2 = \deg \Gamma = (\Gamma^2) + (\Gamma \cdot \Delta)$ and $1/2 = \deg \Delta = (\Gamma \cdot \Delta) + (\Delta^2)$.
In particular, $(\Gamma^2) = (\Delta^2)$.
Since $\tilde{\Gamma}$ and $\tilde{\Delta}$ are nonsingular rational curve contained in the smooth locus of $T$ and $K_T = \varphi^* K_S \sim 0$, we have $(\tilde{\Gamma}^2) = (\tilde{\Delta}^2) = -2$.
Let $l_1,l_2,r_1,r_2$ be rational numbers such that $\varphi^*\Gamma = \tilde{\Gamma} + l_1 E_1 + l_2 E_2$ and $\varphi^*\Delta = \tilde{\Delta} + r_1 E_1 + r_2 E_2$.
We have 
\[
(\Gamma^2) = (\tilde{\Gamma}^2) + l_1 (E_1 \cdot \tilde{\Gamma}) + l_2 (E_2 \cdot \tilde{\Gamma}) = -2 + l_1.
\] 
Similarly, we have $(\Delta^2) = -2 + r_1$.
Since $(x_0 = 0)_S = \Gamma + \Delta$ and $x_0$ vanishes along $E_1$ with order $3/2$ or $3$, depending on $\bar{f}_3 \ne 0$ or $\bar{f}_3 = 0$, we have $l_1 + r_1 = 3/2$ or $3$.
Hence $l_1 = r_1 = 3/4$ or $3/2$, depending on whether $\bar{f}_3 \ne 0$ or $\bar{f}_3 = 0$.
Thus $(\Gamma^2) = -2 + l_1 = -5/4$ or $-1/2$.
It follows that $(\Gamma \cdot \Delta) = 1/2 - (\Gamma^2) = 7/4$ or $1$ and we obtain $(\Gamma \cdot \Delta) \ge \deg \Gamma = 1/2$.
\end{proof}

\begin{Prop} \label{exclC17}
Let $X'$ be a member of $\mcG'_{17}$ and $\Gamma$ a curve of degree $1/2$ passing through the point $\msp_4$.
Then $\Gamma$ is not a maximal center.
\end{Prop}

\begin{proof}
We consider the case where either $b_2 \notin (x_0,x_1)$ or $c_4 \notin (x_0,x_1)$.
In this case, in view of Lemmas \ref{lem:intspcurve17} and \ref{selfint17}, we can apply Lemma \ref{exclcurvelowdeg} for the divisor $S \sim_{\mbQ} A$ and the movable linear system $\mcM = \left| \mcI_{\Gamma} (1) \right|$.

We consider the case where $b_2, c_4 \in (x_0,x_1)$.
We assume to the contrary that $\Gamma$ is a maximal center.
Then there is a movable linear system $\mcH \subset \left| n A \right|$ which has a maximal singularity along $\Gamma$.
Let $\gamma$ be the multiplicity of $\mcH$ along $\Gamma$.
Then, since $\mult_{\Gamma} (\mcH|_S) = m \mult_{\Gamma} \mcH = m \gamma$ for some $m \ge 2$, we have
\[
A|_S \sim_{\mbQ} \frac{1}{n} \mcH|_S = \frac{1}{n} \mcL + m \gamma \Gamma,
\]
where $\mcL$ is a movable linear system on $S$.
On the other hand, we have $A|_S \sim_{\mbQ} S'|_S = 2 \Gamma$.
It follows that $\mcL \sim_{\mbQ} n (2-m \gamma) \Gamma$.
Since $\mcL$ is nef, we have $\gamma \le 2/m \le 1$, which is a contradiction.
This completes the proof.
\end{proof}

Combining Propositions \ref{exclmostC}, \ref{exclC1923} and \ref{exclC17}, we get the following. 

\begin{Thm}
Let $X'$ be a member of $\mcG'_i$ with $i \in I_{F,cAx/2} \cup I_{F,cAx/4}$.
Then no curve on $X'$ is a maximal center.
\end{Thm}

\section{Excluding nonsingular points}  \label{sec:nspts}

In this section we exclude nonsingular points of $X'$ as maximal center.
For a member $X' \subset \mbP (a_0,\dots,a_4)$ of $\mcG'_i$ and $m = 0,\dots,4$, we denote by $\pi_m$ the restriction of the projection from the point $\msp_m$ to $X'$ and by $\Exc (\pi_m) \subset X'$ the locus contracted by $\pi_m$.

\begin{Prop} \label{isolnspt}
Let $\msp = (\xi_0 \!:\! \dots \!:\! \xi_4)$ be a nonsingular point of $X'$.
The following assertions hold.
\begin{enumerate}
\item If $\xi_j \ne 0$, then 
\[
\left( \max_{0 \le k \le 4, k \ne j} \lcm \{ a_j, a_k\} \right) A
\] 
isolates $\msp$.
\item If $\xi_j \ne 0$ and $\msp \notin \Exc (\pi_m)$ for some $m \ne j$, then 
\[
\left( \max_{0 \le k \le 4, k \ne j,m} \lcm \{ a_j,a_k\} \right) A
\] 
isolates $\msp$.
\end{enumerate}
\end{Prop}

\begin{proof}
See \cite[Theorem 5.6.2]{CPR} or \cite[Proposition 7.4]{Okada}.
\end{proof}

\begin{Prop} \label{exclnsptmost}
Let $X'$ be a member of $\mcG'_i$ with $i \in I_{F,cAx/2} \cup I_{F,cAx/4} \setminus \{30\}$.
Then no nonsingular point of $X'$ is a maximal center.
\end{Prop}

\begin{proof}
We shall show that $m A$ isolates a nonsingular point of $X'$ for some $m \le 4/(A^3)$, which will completes the proof by Lemma \ref{exclmostnspt}.
Let $X'$ be a member of $\mcG'_i$ with defining polynomial $F$, where
\[
i \in \{17,29,41,49,55,69,74,77,82\}.
\]
Then we have $x_3^2 \in F$.
It follows that $\Exc (\pi_3) = \emptyset$.
Explicit calculation in each instance shows that
\[
m_1 := \max_{j,k \ne 3} \{\lcm \{a_j,a_k\} \} \le \frac{4}{(A^3)}.
\]
Now let $\msp = (\xi_0 \!:\! \cdots \!:\! \xi_4)$ be a nonsingular point of $X'$.
Then $\xi_j \ne 0$ for some $j \ne 3$.
It follows from Proposition \ref{isolnspt} that $m_1 A$ isolates $\msp$ and $m_1 \le 4/(A^3)$.

Let $X'$ be a member of $\mcG'_{42}$.
Then we have $x_2^3 \in F$.
It follows that $\Exc (\pi_2) = \emptyset$ and we have
\[
m_2 := \max_{j,k \ne 2} \{\lcm (a_j,a_k)\} = 10 < \frac{4}{(A^3)} = \frac{40}{3}.
\]
Thus $m_2 A$ isolates nonsingular points and $m_2 \le 4/(A^3)$.

It remains to consider a member $X'$ of $\mcG'_i$ with $i \in \{19,23,50\}$.
Let $X'$ be a member of $\mcG'_{19}$.
After replacing $w \mapsto w + \alpha y$ for some $\alpha \in \mbC$ if necessary, we may assume that $y^4 \in F$.
Then $\Exc (\pi_2) = \emptyset$ and we compute
\[
\max_{j,k \ne 2} \{\lcm (a_j,a_k)\} = 6 = \frac{4}{(A^3)}.
\]
Thus $6 A$ isolates any nonsingular point and $6 \le 4/(A^3)$.

Let $X'$ be a member of $\mcG'_{50}$.
After replacing $w$ if necessary, we may assume that $y^7 \in F$.
Then we have $\Exc (\pi_1) = \emptyset$ and we compute
\[
\max_{j,k \ne 1} \{\lcm (a_j,a_k)\} = 20 < \frac{4}{(A^3)} = \frac{240}{7}.
\]
Thus $20 A$ isolates any nonsingular point and $20 \le 4/(A^3)$.

Let $X' = X'_{10} \subset \mbP (1,1,2,3,4)$ be a member of $\mcG'_{23}$ and $\msp = (\xi_0 \!:\! \xi_1 \!:\! \upsilon \!:\! \zeta \!:\! \omega)$ a nonsingular point of $X'$.
If $\xi_0 \ne 0$ or $\xi_1 \ne 0$ then $4A$ isolates $\msp$ by Proposition \ref{isolnspt}.
We assume that $\xi_0 = \xi_1 = 0$.
In this case, we have $\upsilon \ne 0$ and $\zeta \ne 0$ because otherwise $\msp$ would be a singular point.
Then the set
\[
\{ x_0, x_1, \upsilon^3 z^2 - \zeta^2 y^3, \upsilon^2 w - \omega y^2 \}
\]
isolates $\msp$.
Hence $6 A$ isolates $\msp$ (see \cite[Lemma 5.6.4]{CPR}) and we have $6 < 48/5 = 4/(A^3)$. 
This completes the proof.
\end{proof}

\begin{Prop} \label{exclnsptno30}
Let $X'$ be a member of $\mcG'_{30}$.
Then, no nonsingular point of $X'$ is a maximal center.
\end{Prop}

\begin{proof}
The defining polynomial of $X' = X'_{10} \subset \mbP (1,1,3,4,2)$ can be written as $w^2 y (y + f_3) + w g_8 + h_{10}$, where $f_3 = f_3 (x_0,x_1)$ and $g_8, h_{10} \in \mbC [x_0,x_1,y,z]$.
If $\msp \in (x_0 \ne 0) \cup (x_1 \ne 0)$ (resp.\ $\msp \in (w \ne 0)$), then $4 A$ (resp.\ $6 A$) isolates $\msp$.
By the inequality $4 < 6 < 4/(A^3) = 48/5$, $\msp$ is not a maximal center. 

Let $\msp$ be a nonsingular point of $X$ contained in $(x_0 = x_1 = w = 0)$.
If $y^2 z \in h_{10}$, then $(x_0 = x_1 = w = 0)_X$ consists of singular points of $X'$ and there is nothing to prove.
We assume that $y^2 z \notin h_{10}$.
In this case, $X'$ contains the curve $\Gamma = (x_0 = x_1 = w = 0)$.
Let $S$ be a general member of the pencil $|\mcI_{\Gamma} (1)|$.
The base locus of the linear system is $(x_0 = x_1 = 0)_X = \Gamma \cup \Delta$, where $\Delta = (x_0 = x_1 = w y + z^2 = 0)$ is a curve.
We claim that $S$ is nonsingular along $\Bs |\mcI_{\Gamma} (1)| \setminus \{\msp_2, \msp_3, \msp_4\}$.
Indeed, the defining polynomial of $X'$ can be written as $w^2 y^2 + w (z^2 + y z a_1) + y^3 b_1 + c_{10}$, where $a_1, b_1 \in \mbC [x_0,x_1]$ is of degree $1$ and $c_{10} = c_{10} (x_0,x_1,y,z,w) \in (x_0,x_1)^2$.
The surface $S$ is cut out on $X$ by $\alpha_0 x_0 + \alpha_1 x_1$, where $\alpha_0,\alpha_1 \in \mbC$.
The restriction of Jacobian matrix $J_{C_S}$ of the affine cone $C_S$ of $S$ to $(x_0 = x_1 = 0)$ can be computed as
\[
\begin{pmatrix}
w y z \dfrac{\prt a_1}{\prt x_0} +y^3 \dfrac{\prt b_1}{\prt x_0} & w y z \dfrac{\prt a_1}{\prt x_1} +y^3 \dfrac{\prt b_1}{\prt x_1} & 2 w^2 y & 2 w z & 2 w y^2 + z^2 \\[3mm]
\alpha_0 & \alpha_1 & 0 & 0 & 0
\end{pmatrix}.
\]
The intersection of the zero loci of $2 w^2 y$, $2 w z$ and $2 w y^2 + z^2$ is the set $\{\msp_2, \msp_4\}$.
It follows that $S$ ia quasismooth along $(x_0 = x_1 = 0) \setminus \{ \msp_2, \msp_4 \}$.
Since $\alpha_0, \alpha_1$ are general, we see that $S$ is quasismooth at both $\msp_2$ and $\msp_4$.
Moreover, the singularity of $S$ at $\msp_2$ and $\msp_3$ are of type respectively $\frac{1}{3} (1,2)$ and $\frac{1}{4} (1,3)$.

Since $\Gamma$ is a nonsingular rational curve passing through two singular points $\msp_2, \msp_3$ of $S$, we can compute 
\[
(\Gamma^2) = -2 + \frac{2}{3} + \frac{3}{4} = - \frac{7}{12},
\] 
where the terms $2/3$ and $3/4$ come from the different.
Let $S'$ be another general member of $|\mcI_{\Gamma} (1)|$.
We have $A|_S \sim S'|_S = \Gamma + \Delta$.
From this relation, we have $(\Gamma \cdot A|_S) = (\Gamma^2) + (\Gamma \cdot \Delta)$ and $(\Delta \cdot A|_S) = (\Gamma \cdot \Delta) + (\Delta^2)$.
Since $(\Gamma \cdot A|_S) = \deg \Gamma = 1/12$ and $(\Delta \cdot A|_S) = \deg \Delta = 1/3$, we get $(\Gamma \cdot \Delta) = 2/3$ and $(\Delta^2) = -1/3$.

Now we assume that $\msp$ is a maximal center of a movable linear system $\mcH \subset |n A|$.
We can write
\[
A|_S \sim \frac{1}{n} \mcH |_S = \frac{1}{n} \mcL + \gamma \Gamma + \delta \Delta,
\]
where $\mcL$ is a movable linear system on $S$ and $\gamma, \delta \ge 0$.
We shall show that $\gamma \le 1$.
Assume to the contrary that $\gamma > 1$.
By the linear equivalences $A|_S \sim \Gamma + \Delta$ and $A|_S \sim (1/n)\mcL + \gamma \Gamma + \delta \Delta$, we have $(1-\gamma)A|_S \sim L + (\delta-\gamma) \Delta$.
By taking intersection with $\Delta$, we have $(1-\gamma) (A|_S \cdot \Delta) \ge  (\delta - \gamma) (\Delta^2)$ and thus $(1-\gamma)/3 \ge - (\delta - \gamma)/3$.
It then follows that $\delta \ge 2 \gamma -1 > 1$.
On the other hand,
\[
\frac{1}{n} \mcL \sim_{\mbQ} A|_S -\gamma \Gamma - \delta \Delta \sim (1-\gamma) \Gamma + (1-\delta) \Delta.
\]
This is a contradiction since $\gamma, \delta > 1$ and $\mcL$ is nef.
Hence we have $\gamma \le 1$.

Since $\msp \notin \Delta$ and $\gamma \le 1$, we have $(1/n^2)(\mcL^2) > 4 (1-\gamma)$ by (1) of Theorem \ref{multineq2}.
On the other hand, we have
\[
\begin{split}
4 (1-\gamma) - \frac{1}{n^2} (\mcL^2) 
&= 4 (1-\gamma) - (A|_S - \gamma \Gamma - \delta \Delta)^2 \\
&= 4 (1-\gamma) - \left(\frac{5}{12} - \frac{1}{6} \gamma - \frac{2}{3} \delta + \frac{4}{3} \gamma \delta - \frac{7}{12} \gamma^2 - \frac{1}{3} \delta^2 \right) \\
&= \frac{1}{3} \left( \delta - (2 \gamma-1) \right)^2 - \frac{1}{4} (\gamma -1)(3 \gamma + 13) \\
&\ge - \frac{1}{4} (\gamma -1)(3 \gamma + 13) \ge 0,
\end{split}
\]
where the last inequality follows from $0 \le \gamma \le 1$.
This is a contradiction and $\msp$ is not a maximal center.
\end{proof}

\begin{Thm}
Let $X'$ be a member of $\mcG'_i$ with $i \in I_{F,cAx/2} \cup I_{F,cAx/4}$.
Then no nonsingular point of $X'$ is a maximal center.
\end{Thm}

\begin{proof}
This follows from Propositions \ref{exclnsptmost} and \ref{exclnsptno30}.
\end{proof}

\section{Excluding terminal quotient singular points} \label{sec:singpts}

Throughout this section, let $\msp$ be a terminal quotient singular point of a member $X'$ of $\mcG'_i$, which is not a center of a birational involution constructed in Section \ref{sec:birinv}.
The aim of this section is to exclude such a point $\msp$ as maximal center.
We denote by $\varphi' \colon Y' \to X'$ the Kawamata blowup at $\msp$ with exceptional divisor $E$ and put $B = -K_{Y'}$ as before.
Let $S \in |B|$ be a general member which is the proper transform of a general member of $|A|$.

We use the following fact frequently: for the exceptional divisor of the Kawamata blowup $\varphi \colon W \to V$ at a terminal quotient singular point $\msp \in V$ of type $\frac{1}{r} (1,a,r-a)$, we have $(E^3) = r^2/a (r-a)$ (see \cite[Proposition 3.4.6]{CPR} for a detail).

\subsection{The case $(B^3) \le 0$} \label{sec:singptsnonpos}

In this subsection we treat terminal quotient singular points with $(B^3) \le 0$ and exclude them as maximal center.
We introduce the following assumption.

\begin{Assumption} \label{assumpsing}
\begin{enumerate}
\item If $X'$ is a member of $\mcG'_{23}$ and $\msp$ is the point of type $\frac{1}{2} (1,1,1)$, then there is no WCI curve of type $(1,1,4)$ on $X'$ passing through $\msp$.
\item If $X'$ is a member of $\mcG'_{50}$ and $\msp$ is the point of type $\frac{1}{2} (1,1,1)$ (resp.\ $\frac{1}{3} (1,1,2)$), then there is no WCI curve of type $(1,3,4)$ on $X'$ through $\msp$ (resp.\ $z^3 t$ appears in the defining polynomial of $X'$ with non-zero coefficient).
\end{enumerate}
\end{Assumption}

\begin{Prop} \label{singptsurf}
Let $X'$ be a member of $\mcG'_i$ and $\msp$ a terminal quotient singular point of $X'$ with $(B^3) \leq 0$ satisfying \emph{Assumption \ref{assumpsing}}.
Then there exists a surface $T \in |B|$ on $Y'$ with the following properties.
\begin{enumerate}
\item The scheme theoretic intersection $\Gamma := S \cap T$ is an irreducible and reduced curve on $Y$.
\item $(T \cdot \Gamma) \leq 0$.
\end{enumerate}
In particular $\msp$ is not a maximal center.
\end{Prop}

\begin{table}[b]
\begin{center}
\caption{Defining polynomials of $\Gamma$}
\label{table:gammaqsm}
\begin{tabular}{ccc}
\hline
No. & $G$ & Condition \\
\hline
23 & $w (y^3 + z^2) + \alpha y^2 z^2$ & $\alpha \ne 0$ \\
29 & $w^3 y^2 + \beta w^2 y^3 + \gamma w y^4 + \delta y^5 + z^2$ & \\
42 & $w^3 y^2 + w z^2 + y^3$ & \\
49 & $w y^5 + \beta y^7 + t^2$ & \\
50 & $w^2 z^2 + w t^2 + \alpha z^3 t$ & $\alpha \ne 0$ \\
55 & $w^3 y^2 + \beta w y^3 + z^2$ & \\
74 & $w^3 z^2 + \beta z^6 + \gamma z^3 t + t^2$ & \\
77 & $w^3 y^2 + y^3 + z^2$ & \\
82 & $w^3 z^2 + t^2$ & 
\end{tabular}
\end{center} 
\end{table}

\begin{proof}
Let $\mbP (a_0,\dots,a_3,b)$, $a_0 \le \cdots \le a_3$, be the ambient weighted projective space of $X'$ with homogeneous coordinates $x_0,\dots,x_3,w$ and $F$ the defining polynomial of $X'$.
We see that both $x_0$ and $x_1$ vanish at $\msp$ and lift to plurianticanonical sections on $Y$. 
We set $T := (x_1 = 0)$. 
Then $\Gamma = S \cap T$ is the proper transform of $(x_0 = x_1 = 0)_{X'}$ which is isomorphic to a weighted complete intersection defined by the equation $G := F (0,0,x_2,x_3,w) = 0$ in $\mbP (a_2,a_3,b)$.
We see that the polynomial $G$ is of the form listed in Table \ref{table:gammaqsm} for some complex numbers $\alpha,\beta, \gamma, \delta$. 
The condition $\alpha \ne 0$ follows from Assumption \ref{assumpsing} (see Examples \ref{ex:singpt23} and \ref{ex:singpt50} below).
This implies that $\Gamma$ is irreducible and reduced.
Finally, since $S \in \left| B \right|$ and $T \in \left| a_1 B \right|$, we have 
\[
(T \cdot \Gamma) = (T \cdot S \cdot T) = a_1^2 (B^3) \le 0.
\]
By Lemma \ref{excltestsurf}, $\msp$ is not a maximal center.
\end{proof}

In the following examples, let $F$ be the defining polynomial of $X'$ and $S$, $T$ be the proper transforms on $Y'$ of the divisors $(x_0 = 0)_{X'}$, $(x_1 = 0)_{X'}$, respectively.
We explain that the polynomial $G := F (0,0,x_2,x_3,w)$ is of the form given in Table \ref{table:gammaqsm}.

\begin{Ex} \label{ex:singpt23}
Let $X' = X'_{10} \subset \mbP (1,1,2,3,4)$ be a member of $\mcG'_{23}$ with $F = w^2 x_0 (x_0 + f_1) + w g_6 + h_{10}$ and $\msp$ a point of type $\frac{1}{2} (1,1,1)$.
If $y^3 \notin g_6$, then there is no point of type $\frac{1}{2} (1,1,1)$.
Hence we may assume that the coefficient of $y^3$ in $g_6$ is one.
After replacing $y$, we may assume that $\msp = \msp_2$.
It follows that $y^5 \notin h_{10}$.
Then we have $G = w (z^2 + y^3) + \alpha z^2 y^2$ for some $\alpha \in \mbC$.
We have $\alpha \ne 0$ by Assumption \ref{assumpsing}.
\end{Ex}

\begin{Ex} \label{ex:singpt42}
Let $X' = X'_{12} \subset \mbP (1,1,4,5,2)$ be a member of $\mcG'_{42}$ with $F = w^2 y (y + f_4) + w g_{10} + h_{12}$ and $\msp$ a point of type $\frac{1}{2} (1,1,1)$.
If $y^3 \notin h_{12}$, then there is no point of type $\frac{1}{2} (1,1,1)$.
Hence we may assume that the coefficient of $y^3$ in $h_{12}$ is one.
Then we have $G = w^2 y^2 + w z^2 + y^3$.
\end{Ex}

\begin{Ex} \label{ex:singpt50}
Let $X' = X'_{14} \subset \mbP (1,2,3,5,4)$ be a member of $\mcG'_{50}$ with $F = w^2 z (z + f_3) + w g_{10} + h_{14}$ and $\msp = \msp_2$ the point of type $\frac{1}{3} (1,1,2)$.
Note that $T \in |2 B|$ is the proper transform of $(y=0)_{X'}$.
We have $G = w^2 z^2 + w t^2 + \alpha t z^3$ for some $\alpha \in \mbC$. 
We have $\alpha \ne 0$ by Assumption \ref{assumpsing}.
\end{Ex}

The remaining singular points with $(B^3) < 0$ will be excluded by constructing a suitable nef divisor on the Kawamata blowup of $X'$ at the point.

\begin{Lem} \label{lem:critc}
Let $V$ be a $\mbQ$-Fano variety, $\msp \in V$ a point and $\varphi \colon W \to V$ an extremal divisorial extraction centered at $\msp$ with exceptional divisor $E$.
Assume that there are prime Weil divisors $D_1, \dots,D_k$ on $V$ with the following properties.
\begin{enumerate}
\item The intersection $D_1 \cap \cdots \cap D_k$ does not contain a curve passing through $\msp$.
\item For $i = 1,2,\dots,k$, the proper transform $\tilde{D}_i$ of $D_i$ on $W$ is $\mbQ$-linearly equivalent to $- b_i K_W + e_i E$ for some $b_i > 0$ and $e_i \ge 0$.
\item We have $c \le a_E (K_V)$, where $c := \max \{e_i/b_i\}$ and $a_E (K_V)$ is the discrepancy of $K_V$ along $E$.
\end{enumerate}
Then, the divisor $M = - K_W + c E$ is nef.
\end{Lem}

\begin{proof}
Set $A := -K_V$.
We have $K_W = \varphi^*K_V + a_E (K_V) E$, thus $M = -K_W + c E \sim_{\mbQ} \varphi^*A  + d E$, where $d = c- a_E (K_V)$.
Let $m$ be a sufficiently large and divisible positive integer.
The complete linear system $|m M|$ contains a suitable positive multiple of $\tilde{D}_i + l_i E_i$ for some $l_i \ge 0$.
It follows from the assumption (1) that the base locus of $|m M|$ does not contain a curve which intersects $E$ but is not contained in $E$.

Let $C$ be an irreducible and reduced curve on $W$.
If $C$ is disjoint from $E$, then $(M \cdot C) = (A \cdot \varphi_*C) > 0$ since $A$ is ample.
If $C$ intersects $E$ but is not contained in $E$, then $(M \cdot C) \ge 0$ since $C$ is not contained in the base locus of $|m M|$.
Assume that $C$ is contained in $E$.
Note that we have $(E \cdot C) < 0$ since $-E$ is $\varphi$-ample. 
Then we have
\[
(M \cdot C) = (\varphi^*A + d E \cdot C)
= d (E \cdot C) \ge 0
\]
since $d = c - a_E (K_V) \le 0$ by the assumption (3).
This shows that $M$ is nef.
\end{proof}

Let $X'$ be a member of $\mcG'_i$ and $\msp \in X'$ a point.
We say that a set of homogeneous polynomials $\{g_1,\dots,g_m\}$ {\it isolates} $\msp$ if the intersection of divisors $(g_1 = 0)_{X'}, \dots, (g_m = 0)_{X'}$ does not contain a curve passing through $\msp$.

\begin{Prop} \label{prop:singBnegtc}
Let $X'$ be a member of $\mcG_i$ for $i = 50,74,82$ and $\msp$ a terminal quotient singular point of type $\frac{1}{2} (1,1,1)$ satisfying \emph{Assumption \ref{assumpsing}}.
Then the divisor $M := b B + e E$ in the second column of the table is a nef divisor such that $(M \cdot B^2) \le 0$.
In particular, $\msp$ is not a maximal center.
\end{Prop}

\begin{proof}
Let $X' = {X'}_{14} \subset \mbP (1,2,3,5,4)$ be a member of $\mcG'_{50}$ with the defining polynomial $w^2 z (z+f_3) + w g_{10} + h_{14}$.
We may assume that $y^5 \in g_{10}$ because otherwise $X'$ does not contain a singular point of type $\frac{1}{2} (1,1,1)$.
After replacing $w$, we may assume that $\msp = \msp_1$ or equivalently $y^7 \notin h_{14}$.
Then either $\{x,z,w\}$ isolates $\msp$ or $X'$ contains the WCI curve $(x = z = w = 0)$ of type $(1,3,4)$.
The latter cannot happen by Assumption \ref{assumpsing}.
It follows that $\{x,z,w\}$ isolates $\msp$.
We see that $x,z,w$ vanish along $E$ to order respectively $1/2,1/2,2/2$, so that they lift to sections of $B, 3 B + E, 4 B + E$, respectively (see Remark \ref{rem:compvan} below).
It follows from Lemma \ref{lem:critc} that $M = 3 B + E$ is a nef divisor on $Y'$ and we compute
\[
(M \cdot B^2) = 3 (A^3) - \frac{1}{2^3} (E^3) = 3 \times \frac{7}{60} - \frac{1}{2} < 0.
\]

Let $X' = X'_{18} \subset \mbP (1,2,3,9,4)$ be a member of $\mcG'_{74}$ with defining polynomial $w^3 z^2 + w^2 z f_7 + w g_{14} + h _{18}$.
We may assume that $y^7 \in g_{14}$ and that $\msp = \msp_1$.
We see that $(x = z = w = 0)_{X'} = \{\msp\}$ since $t^2 \in h_{14}$ by quasismoothness of $X'$ at $\msp_3$.
This shows that $\{x,z,w\}$ isolates $\msp$ and the birational transform of $(x = 0)_{X'}$, $(z = 0)_{X'}$ and $(w = 0)_{X'}$ on $Y$ are divisors $B$, $3 B + E$ and $4 B + E$, respectively.
It follows from Lemma \ref{lem:critc} that $3 B + E$ is a nef divisor and we compute 
\[
(M \cdot B^2) = 3 (A^3) - \frac{1}{2^3} (E^3) = 3 \times \frac{1}{12} - \frac{1}{2} < 0.
\]

Finally, let $X' = X'_{22} \subset \mbP (1,2,5,11,4)$ be a member of $\mcG'_{82}$ with defining polynomial $w^3 z^2 + w^2 z f_9 + w g_{18} + h_{22}$.
We may assume that $y^9 \in g_{18}$ and may assume that $\msp = \msp_1$ after replacing $w$.
We see that $\{x,z,w\}$ isolates $\msp$ since $t^2 \in h_{22}$ by quasismoothness of $X'$ at $\msp_3$.
We also see that the birational transform of $(x = 0)_{X'}$, $(z = 0)_{X'}$ and $(w = 0)_{X'}$ on $Y$ are divisors $B$, $5B + 2E$ and $4B+E$, respectively.
It follows from Lemma \ref{lem:critc} that $M := 5B + 2E$ is a nef divisor and we compute
\[
(M \cdot B^2) = 5 (A^3) - \frac{1}{2^3} (E^3) = 5 \times \frac{1}{20} - \frac{1}{2} < 0.
\]
The last statement follows from Corollary \ref{tsmethod}.
\end{proof}

\begin{Rem} \label{rem:compvan}
We briefly explain the computation of vanishing order of sections along the exceptional divisor of Kawamata blowup.
For more details, we refer the readers to \cite[Examples 5.7.2-5.7.5]{CPR} and \cite[Example 4.5]{Okada}.

Let $V$ be a weighted hypersurface in $\mbP (a_0,\dots,a_4)$ defined by a homogeneous polynomial $F$.
Assume that $V$ is quasismooth at $\msp = \msp_0 \in V$ and it has terminal singular at $\msp$.
Then, after replacing coordinates, we can write $F = x_0^k x_4 + G$, where $G = G (x_0,\dots,x_4)$ does not contain $x_0^k x_4$.
In this case, the Kawamata blowup $\varphi \colon W \to V$ is the blowup with $\wt (x_1,x_2,x_3) = \frac{1}{r} (1,a,r-a)$, where we assume $r = a_0$, $a_1 \equiv 1, a_2 \equiv a, a_3 \equiv r-a \pmod{r}$.
In this case, the sections $x_1,x_2,x_3$ vanish along the exceptional divisor $E$ to order respectively $1/r, a/r, (r-a)/r$, and $x_0$ does not vanish along $E$.
We compute the vanishing order of the remaining coordinate $x_4$.
Filtering off terms divisible by $x_4$ in the defining equation, we have $u x_4 = - G (x_0,x_1,x_2,x_3,0)$, where $u = x_0^k + (\text{other terms})$.
Now we compute the vanishing order of each monomial in $G(x_0,x_1,x_2,x_3,0)$ and let $b/r$ be the minimum of them.
Then, since $u$ does not vanish at $\msp$, we see that the vanishing order of $x_4$ along $E$ is $b/r$.
\end{Rem}

In the rest of this subsection we exclude terminal quotient singular points with $(B^3) \le 0$ which do not satisfy Assumption \ref{assumpsing}.

Let $X' = X'_{10} \subset \mbP (1,1,2,3,4)$ be a member of $\mcG'_{23}$ and $\msp$ the point of type $\frac{1}{2} (1,1,1)$.
We treat the case where $X'$ contains a WCI curve of type $(1,1,4)$ passing through $\msp$. 
The defining polynomial of $X'$ is of the form $w^2 x_0 (x_0 + f_1) + w g_6 + h_{10}$, where $f_1, g_6, h_{10} \in \mbC [x_0,x_1,y,z]$.
By replacing $y$ and $w$ suitably, we may assume that $\msp = \msp_2$, $g_6 = y^3 + y^2 a_2 + y a_4 + a_6$ and $h_{10} = y^2 b_6 + y b_8 + b_{10}$, where $a_i, b_i \in \mbC [x_0,x_1,z]$.
We see that $X'$ contains a WCI curve of type $(1,1,4)$ passing through $\msp$ if and only if $z^2 \notin b_6$.
In this case the $(1,1,4)$ curve is $(x_0 = x_1 = w = 0)$ and we denote by $\Gamma$ its proper transform on $Y'$.
We see that at least one of the coefficients of $z^3 x_0$ and $z^3 x_1$ in $b_{10}$ is non-zero because otherwise $X'$ has a singular point at $(0 \!:\! 0 \!:\! -1 \!:\! 1 \!:\! 0)$ and this is a contradiction.

\begin{Prop} \label{exclspptNo23}
Let $X'$ be a member of the family $\mcG'_{23}$ and $\msp$ a singular point of type $\frac{1}{2} (1,1,1)$.
Assume that $X'$ contains a WCI curve of type $(1,1,4)$ passing through $\msp$.
Then $\msp$ is not a maximal center.
\end{Prop}

\begin{proof}
We keep notation on defining polynomial of $X'$ as in the preceding argument.
For general $\lambda, \mu \in \mbC$, let $S_{\lambda} \sim_{\mbQ} B$ and $T_{\mu} \sim_{\mbQ} 4 B$ be the birational transform on $Y'$ of the divisors $(x_1 - \lambda x_0 = 0)_{X'}$ and  $(w - \mu x_0^4 = 0)_{X'}$, respectively.
We see that $S_{\lambda} \cap T_{\mu}$ is the proper transform of the curve on $X'$ defined by equations $x_1 - \lambda x_0 = w - \mu x_0^4 = x_0 G_{\lambda,\mu} = 0$, where
\[
G_{\lambda,\mu} =
\mu^2 x_0^8 (x_0 + f_1 (x_0,\lambda x_0)) + \mu x_0^3 g_6 (x_0,\lambda x_0,y,z) + \frac{h_{10} (x_0,\lambda x_0,y,z)}{x_0}.
\]  
Since at least one of $z^3 x_0$ and $z^3 x_1$ appears in $h_{10}$ and $\lambda$ is general, $G_{\lambda,\mu}$ is not divisible by $x_0$.
It follows that $S_{\lambda} \cap T_{\mu} = \Gamma + C_{\lambda,\mu}$, where $C_{\lambda,\mu}$ is the proper transform on $Y'$ of the curve $(x_1 - \lambda x_0 = w - \mu x_0^4 = G_{\lambda,\mu} = 0)$.
We see that $C_{\lambda,\mu}$ does not contain $\Gamma$ as a component and $C_{\lambda,\mu} \cap E \ne \emptyset$.
We shall show that there is a component $C_{\lambda,\mu}^{\circ}$ of $C_{\lambda,\mu}$ such that $(B \cdot C_{\lambda,\mu}^{\circ}) \le 0$ and $(E \cdot C_{\lambda,\mu}^{\circ}) > 0$.

Since $\Gamma$ intersects $E$ transversally at a nonsingular point of $Y'$ at which both $\Gamma$ and $E$ are nonsingular, we have $(E \cdot \Gamma) = 1$.
It follows that 
\[
(B \cdot \Gamma) = (A \cdot \varphi_* \Gamma) - \frac{1}{2} (E \cdot \Gamma) = \frac{1}{6} - \frac{1}{2} = - \frac{1}{3}.
\]
This, together with 
\[
(B \cdot \Gamma + C_{\lambda,\mu}) = (B \cdot S_{\lambda} \cdot T_{\mu}) = 4 (B^3) = -\frac{1}{3}, 
\]
shows that $(B \cdot C_{\lambda,\mu}) = 0$.
We write $C_{\lambda,\mu} = C'_{\lambda,\mu} + C''_{\lambda,\mu}$ where each component of $C'_{\lambda,\mu}$ (resp.\ $C''_{\lambda,\mu}$) intersects $E$ (resp.\ is disjoint from $E$).
Note that $C'_{\lambda,\mu} \ne 0$ since $C_{\lambda,\mu} \cap E \ne \emptyset$ and that $(B \cdot C''_{\lambda,\mu}) = (A \cdot {\varphi'}_* C''_{\lambda,\mu}) \ge 0$.
It follows that 
\[
(B \cdot C'_{\lambda,\mu}) = (B \cdot C_{\lambda,\mu}) - (B \cdot C''_{\lambda,\mu}) \le 0.
\]
Therefore there is at least one component, say $C^{\circ}_{\lambda, \mu}$, of $C'_{\lambda,\mu}$ such that $(B \cdot C^{\circ}_{\lambda,\mu}) \le 0$.
By our choice, we have $(C^{\circ}_{\lambda,\mu} \cdot E) > 0$.
It follows that there infinitely many curves on $Y$ which intersect $B = -K_{Y'}$ non-positively and $E$ positively.
By Lemma \ref{crispinfc}, $\msp$ is not a maximal center.
\end{proof}

We exclude singular points of type $\frac{1}{2} (1,1,1)$ and $\frac{1}{3} (1,1,2)$ on a special member $X' = X'_{14} \subset \mbP (1,2,3,5,4)$ of the family $\mcG'_{50}$ as maximal center. 

\begin{Prop} \label{exclsppt2No50}
Let $X'$ be a member of the family $\mcG'_{50}$ and $\msp$ a point of type $\frac{1}{2} (1,1,1)$.
Assume that there is a WCI curve of type $(1,3,4)$ passing through $\msp$.
Then $\msp$ is not a maximal center.
\end{Prop}

\begin{proof}
The defining polynomial of $X'$ is of the form $w^2 z (z+f_3) + w g_{10} + h_{14}$, where $f_3 = f_3 (x,y)$ and $g_{10}, h_{14} \in \mbC [x,y,z,t]$.
We see that $y^5 \in g_{10}$ because otherwise $X'$ does not contain a point of type $\frac{1}{2} (1,1,1)$.
After replacing $w$, we may assume $\msp = \msp_1$.
$X'$ contains a WCI curve of type $(1,3,4)$ passing through $\msp$ if and only if $y^2 t^2 \notin h_{14}$.
In this case the WCI curve is $(x = z = w = 0)$.
Let $T$ be the proper transform of a general member of the linear system on $X'$ generated by $z, yx, x^3$.
We have $S|_T = \Gamma + C$, where $\Gamma$ and $C$ are the proper transforms of curves $(x = z = w = 0)$ and $(x = z = y^5 + t^2 = 0)$, respectively.
We have $S \in |B|$ and $T \in |3B+E|$.
Note that $C$ is disjoint from the exceptional divisor of $\varphi'$.
We compute 
\[
(\Gamma + C \cdot \Gamma) = (S|_T \cdot \Gamma) = (B \cdot \Gamma) = (A \cdot \varphi'_* C) = \frac{1}{4}
\]
and 
\[
(\Gamma + C)^2 = (S|_T)^2 = (B^2 \cdot 3B+E) = -\frac{3}{20}.
\] 
Since $\Gamma$ and $C$ have a intersection point at which $T$ is smooth, we have $m := (\Gamma \cdot C) \ge 1$.
It follows that the matrix
\[
\begin{pmatrix}
(\Gamma^2) & (\Gamma \cdot C) \\
(\Gamma \cdot C) & (C^2)
\end{pmatrix}
=
\begin{pmatrix}
1/4 - m & m \\
m & -2/5 -m
\end{pmatrix}
\]
is negative-definite.
Thus $\msp$ is not a maximal center by Lemma \ref{crispnegdef}.
\end{proof}

\begin{Prop} \label{exclsppt3No50}
Let $X'$ be a member of the family $\mcG'_{50}$ and $\msp$ the point of type $\frac{1}{3} (1,1,2)$.
Assume that $z^3 t$ does not appear in the defining polynomial of $X'$.
Then $\msp$ is not a maximal center.
\end{Prop}

\begin{proof}
We see that $\msp = \msp_2$ and the defining polynomial of $X'$ is of the form $w^2 z (z + f_3) + w g_{10} + h_{14}$, where $f_3,g_{10},h_{14} \in \mbC [x,y,z,t]$.
After replacing $z$, we may assume that $h_{14} = z^4 y + z^3 b_5 + z^2 b_8 + z b_{11} + b_{14}$, where $b_i \in \mbC [x,y,t]$.
Let $T \in |2B|$ be the proper transform of a general member of $|2A|$ on $Y'$.
The surface $T$ is normal and $S|_T = \Gamma + C$, where $\Gamma$ and $C$ are the proper transforms of curves $(x = y = w = 0)$ and $(x = y = w z^2 + t^2 = 0)$, respectively.
We see that $x,y,t,w$ vanish along $E$ to order respectively $1/3,2/3,2/3,1/3$ and $\varphi'$ can be identified with the embedded weighted blowup of $X' \subset \mbP := \mbP (1,2,3,5,4)$ at $\msp$ with $\wt (x,y,t,w) = \frac{1}{3} (1,2,2,1)$.
Let $A_{\mbP}$ be a Weil divisor on $\mbP$ such that $\mcO_{\mbP} (A_{\mbP}) \cong \mcO_{\mbP} (1)$ and $F$ the exceptional divisor of the weighted blowup of $\mbP$.
Note that $A_{\mbP}|_{Y'} = A$ and $F|_{Y'} = E$.
Then, since $\Gamma$ is the intersection of proper transforms of divisors $(x = 0)$, $(y=0)$ and $(w = 0)$ which are $\mbQ$-linearly equivalent to $A_{\mbP} - \frac{1}{3}$, $2 A_{\mbP} - \frac{2}{3} F$ and $4 A_{\mbP} - \frac{1}{3} F$, we have
\[
\begin{split}
(B \cdot \Gamma) &= (A_{\mbP} - \frac{1}{3} F \cdot A_{\mbP} - \frac{1}{3} \cdot 2 A_{\mbP} - \frac{2}{3} F \cdot 4 A_{\mbP} - \frac{1}{3} F) \\
&= \frac{2 \cdot 4}{1 \cdot 2 \cdot 3 \cdot 5 \cdot 4} - \frac{2}{3^4} \times \frac{3^3}{1 \cdot 2 \cdot 2 \cdot 1} = - \frac{1}{10}.
\end{split}
\]
Hence, $(\Gamma + C \cdot \Gamma) = (S|_T \cdot \Gamma) = (B \cdot \Gamma)  = - 1/10$.
We have $(\Gamma + C)^2 = (S|_T^2) = 2 (B^3) = -1/12$.
Set $m := (\Gamma \cdot C)$.
We have $m \ge 1/2$ since $\Gamma$ intersects $C$ at the $\frac{1}{2} (1,1,1)$ point lying on $E$.
Thus, the matrix
\[
\begin{pmatrix}
(\Gamma^2) & (\Gamma \cdot C) \\
(\Gamma \cdot C) & (C^2)
\end{pmatrix}
=
\begin{pmatrix}
-1/10 - m & m \\
m & 1/60 - m
\end{pmatrix}
\]
is negative-definite.
Therefore $\msp$ is not a maximal center by Lemma \ref{crispnegdef}
\end{proof}

\begin{Thm} \label{exclcqvolnp}
Let $X'$ be a member of $\mcG'_i$ with $i \in I_{F,cAx/2} \cup I_{F,cAx/4}$.
Then no terminal quotient singular point $\msp$ with $(B^3) \le 0$ is a maximal center.
\end{Thm}

\begin{proof}
This follows from Propositions \ref{singptsurf}, \ref{prop:singBnegtc}, \ref{exclspptNo23}, \ref{exclsppt2No50} and \ref{exclsppt3No50}.
\end{proof}

\subsection{The Case $(B^3) > 0$} \label{sec:singptspos}

We treat the case where $(B^3) > 0$.
There are only three instances: $X'$ is a member of one of the families $\mcG'_{30}$, $\mcG'_{55}$ and $\mcG'_{69}$, and $\msp = \msp_2$.
Note that, in the case where $X' \in \mcG'_{30}$, there is a birational involution centered at $\msp = \msp_2$ if $y^2 z$ appears in the defining polynomial with non-zero coefficient (see Section \ref{sec:birinv}), so we treat the case where $y^2 z$ does not appear in the defining polynomial.
For a curve $C$ on $X'$, we denote by $\tilde{C}$ the proper transform of $C$ by the Kawamata blowup $\varphi' \colon Y' \to X'$.

\begin{Prop} \label{prop:G'55-4}
Let $X'$ be a member of $\mcG'_{55}$ and $\msp = \msp_2$ the singular point of type $\frac{1}{4} (1,1,3)$.
Then, $\msp$ is not a maximal center.
\end{Prop}

\begin{proof}
The defining polynomial of $X'= X'_{14} \subset \mbP (1,1,4,7,2)$ is of the form $w^3 y^2 + w^2 y f_6 + w g_{12} + h_{14}$, where $f_6, g_{12}, h_{14} \in \mbC [x_0,x_1,y,z]$.
By quasismoothness of $X'$ at $\msp$, we may assume that the coefficients of $y^3$ in $g_{12}$ and $z^2$ in $h_{14}$ are both $1$ after re-scaling $y$ and $z$.
After replacing $w$, we may assume that there is no monomial divisible by $y^3$ in $h_{14}$.
Thus, after replacing $z$, we may assume that $h_{14} = y^2 a_6 + y a_{10} + z^2 + a_{16}$ for some $a_i \in \mbC [x_0,x_1]$.

For a complex number $\lambda$, we set $C_{\lambda} := (w = x_1 - \lambda x_0 = 0)_{X'}$, which is isomorphic to the weighted hypersurface in $\mbP (1,4,7)$ defined by the equation
\[
h_{14} (x_0,\lambda x_0,y,z) = z^2 + y^2 a_6 (x_0,\lambda x_0) + y a_{10} (x_0, \lambda x_0) + a_{14} (x_0,\lambda x_0) = 0.
\]
We claim that $C_{\lambda}$ is reduced for a general $\lambda \in \mbC$.
If $C_{\lambda}$ is non-reduced for a general $\lambda$, then $a_6 = a_{10} = a_{14} = 0$ as polynomials.
But then $X'$ is not quasismooth along the curve $(w = z = g_{12} = 0)$ and this a contradiction.
It follows that, for a general $\lambda \in \mbC$, either $C_{\lambda}$ is irreducible and reduced or $C_{\lambda}$ splits into the sum of two irreducible and reduced curves $C_{\lambda} = C^1_{\lambda} + C^2_{\lambda}$ such that the proper transforms $\tilde{C}^1_{\lambda}$ and $\tilde{C}^2_{\lambda}$ of $C^1_{\lambda}$ and $C^2_{\lambda}$, respectively, on $Y'$ are numerically equivalent.  
Indeed, if $C_{\lambda}$ is reducible, then $C_{\lambda} = C^1_{\lambda} + C^2_{\lambda}$, where $C^1_{\lambda} = (w = x_1 - \lambda x_0 = z + b_7 = 0)$ and $C^2_{\lambda} = (w = x_1 - \lambda x_0 = z - b_7 = 0)$ for some $b_7 \in \mbC [x_0,y]$.
Since they have the same degree and $(E \cdot \tilde{C}^i_{\lambda}) = 1$ for $i = 1,2$, their proper transforms on $Y'$ are numerically equivalent. 

Let $\tilde{C}_{\lambda}$ be the proper transform of $C_{\lambda}$ on $Y'$.
We compute
\[
(B \cdot \tilde{C}_{\lambda}) = ((A - \frac{1}{4} E)^2 \cdot 2 A - \frac{6}{4} E) = 2 \times \frac{1}{4} - \frac{6}{4^3} \times \frac{4^2}{3} = 0.
\] 
and $(E \cdot \tilde{C}_{\lambda}) > 0$ since $C_{\lambda}$ passes through $\msp$.
In the case where $C_{\lambda}$ is reducible, we have $(B \cdot \tilde{C}^i_{\lambda}) = 0$ and $(E \cdot \tilde{C}^i_{\lambda}) > 0$ since $\tilde{C}^1_{\lambda}$ and $\tilde{C}^2_{\lambda}$ are numerically equivalent.
Thus, there are infinitely many irreducible and reduced curve which intersect $B$ non-positively and $E$ positively.
By Lemma \ref{crispinfc}, $\msp$ is not a maximal center.
\end{proof}

\begin{Prop} \label{prop:G'69-5}
Let $X'$ be a member of the familiy $\mcG'_{69}$ and $\msp = \msp_2$ the singular point of type $\frac{1}{5} (1,2,3)$.
Then $\msp$ is not a maximal center.
\end{Prop}

\begin{proof}
Let $X' = X'_{16} \subset \mbP (1,1,5,8,2)$ be a member of $\mcG'_{69}$ and $\msp = \msp_2$.
The defining polynomial of $X$ can be written as $F = w^3 y^2 +w^2 y f_7 + w g_{14} + h_{16}$, where $f_7,g_{14}, h_{16} \in \mbC [x_0,x_1,y,z]$.
By quasismoothness of $X$ at $\msp$, at least one of $y^3 x_0$ and $y^3 x_1$ appears in $h_{16}$.
After replacing $x_0$ and $x_1$, we may assume that the coefficient of $y^3 x_1$ in $h_{16}$ is $1$ and there is no other monomial divisible by $y^3$ in $F$.

For a complex number $\lambda$, we set $C_{\lambda} = (x_1 = w - \lambda x_0^2 = 0)_{X'}$ which is isomorphic to the weighted hypersurface in $\mbP (1,5,8)$ defined by the equation $G_{\lambda} := F (x_0,0,y,z,\lambda x_0^2) = 0$.
Note that $z^2 \in G_{\lambda}$ for any $\lambda$.
We shall show that the proper transform on $Y'$ of every irreducible and reduced component of $C_{\lambda}$ intersects $B = -K_{Y'}$ trivially and $E$ positively.
Assume that $C_{\lambda}$ is irreducible.
Then,
\[
(B \cdot \tilde{C}_{\lambda}) 
= 2 ((A - \frac{1}{5}E)^2 \cdot A - \frac{6}{5}E)
= 2 \left( \frac{1}{5} - \frac{6}{5^3} \times \frac{5^2}{2 \times 3} \right) = 0
\]
and clearly $(E \cdot \tilde{C}_{\lambda}) > 0$ since $C_{\lambda}$ passes through $\msp$.
Assume that $C_{\lambda}$ is either reducible or non-reduced.
Then any irreducible and reduced component $C'_{\lambda}$ of $C_{\lambda}$ is defined by $x_1 = w - \lambda x_0^2 = z + a_8 = 0$ for some $a_8 \in \mbC [x_1,y]$.
We have $\deg C'_{\lambda} = 1/5$ and $(E \cdot \tilde{C}'_{\lambda}) = 1$.
Hence $(B \cdot \tilde{C}'_{\lambda}) = \deg C'_{\lambda} - 1/5 (E \cdot \tilde{C}'_{\lambda}) = 0$.
By Lemma \ref{crispinfc}, $\msp$ is not a maximal center.
\end{proof}

\begin{Prop} \label{prop:G'30-3sp}
Let $X'$ be a member of the family $\mcG'_{30}$ and $\msp$ the singular point of type $\frac{1}{3} (1,1,2)$.
Assume that $y^2 z$ does not appear in the defining polynomial of $X'$ with non-zero coefficient.
Then $\msp$ is not a maximal center.
\end{Prop}

\begin{proof}
Let $X' = X'_{10} \subset \mbP (1,1,3,4,2)$ be a member of $\mcG'_{30}$ and $\msp$ the point of type $\frac{1}{3} (1,1,2)$.
The defining polynomial of $X'$ is of the form $w^2 y (y + f_3) + w g_8 + h_{10}$, where $f_3 \in \mbC [x_0,x_1,z]$ and $g_8,h_{10} \in \mbC [x_0,x_1,y,z]$.
By the assumption, we have $y^2 z \notin h_{10}$.
By quasismoothness of $X'$ at $\msp$, at least one of $y^3 x_0$ and $y^3 x_1$ appears in $h_{10}$ with non-zero coefficient.
After replacing $x_0$ and $x_1$, we may assume that $h_{10} = y^3 x_0 + y^2 b_4 + y b_7 + b_{10}$, where $b_i = \mbC [x_0,x_1,z]$. 

For a complex number $\lambda$, we consider the curve $(x_0 = w - \lambda x_1^2 = 0)_{X'}$, which is isomorphic to the weighted hypersurface in $\mbP (1,3,4)$ defined by the equation $x_1^2 G_{\lambda} = 0$, where
\[
G_{\lambda} = \lambda^2 x_1^2 y (y + f_3 (0,x_1,z)) + \lambda g_8 (0,x_1,y,z) + \frac{h_{10} (0,x_1,y,z)}{x_1^2}.
\]
Since $z^2 \in g_8$ by quasismoothness of $X'$ at $\msp_3$, $G_{\lambda}$ is not divisible by $x_1$ for a general choice of $\lambda$.
We define 
\[
C_{\lambda} := (x_0 = w - \lambda x_1^2 = G_{\lambda} = 0) \subset X'.
\]

Let $\lambda$ be a general complex number so that the coefficient of $z^2$ in $G_{\lambda}$ is non-zero.
Assume that $C_{\lambda}$ is irreducible and reduced.
Then $\deg C_{\lambda} = 2/3$ and $(E \cdot \tilde{C}_{\lambda}) = 2$, where $\tilde{C}_{\lambda}$ is the proper transform of $C_{\lambda}$ on $Y'$.
Hence $(B \cdot \tilde{C}_{\lambda}) = \deg C_{\lambda} - 1/3 (E \cdot \tilde{C}_{\lambda}) = 0$.
Assume that $C_{\lambda}$ is either reduced or non-reduced.
Then any irreducible and reduced component $C'_{\lambda}$ of $C_{\lambda}$ is defined by the equation $x_0 = w - \lambda x^2_1 = z + c_4 = 0$ for some $c_4 \in \mbC [x_0,x_1,y]$.
We have $\deg C'_{\lambda} = 1/3$ and $(E \cdot \tilde{C}_{\lambda}) = 1$.
Hence $(B \cdot \tilde{C}'_{\lambda}) = \deg C'_{\lambda} - \frac{1}{3} (E \cdot \tilde{C}'_{\lambda}) = 0$. 

Therefore, there are infinitely many irreducible and reduced curves on $Y'$ which intersect $B = -K_{Y'}$ non-positively and $E$ positively.
By Lemma \ref{crispinfc}, $\msp$ is not a maximal center. 
\end{proof}

\begin{Thm} \label{thm:exclsingBpos}
Let $X'$ be a member of one of the families $\mcG'_{30}$, $\mcG'_{55}$ and $\mcG'_{69}$.
We assume that the monomial $y^2 z$ does not appear in the defining polynomial of $X' \in \mcG'_{30}$.
Then the singular point $\msp_2$ is not a maximal center.
\end{Thm}

\begin{proof}
This follows from Propositions \ref{prop:G'30-3sp}, \ref{prop:G'55-4} and \ref{prop:G'69-5}.
\end{proof}

\section{Invisible birational involutions for special members} \label{sec:spinv}

The aim of this section is to construct a birational involution of special members of $\mcG'_{19}$ and $\mcG_{19}$.

\subsection{Construction of invisible involutions in a general setting}

In this subsection we explain a construction of suitable birational involutions in a general setting following the argument of \cite{CP}.
The construction is quite complicated and is less explicit compared to the Q.I. and E.I. cases.

Let $\mbP := \mbP (1,1,a,b_1,\dots,b_n)$ be a weighted projective space with homogeneous coordinates $x_0, x_1, y, z_1,\dots,z_n$ of degree respectively $1,1,a,b_1,\dots,b_n$, and $X$ an anticanonically embedded $\mbQ$-Fano $3$-fold weighted complete intersection in $\mbP$.
Let $\msp \in X$ be a point contained in $(x_0 = x_1 = y = 0)$.
We set $\Gamma := (x_0 = x_1 = y = 0)_X$ and let $\mcH \subset \left| - a K_X \right|$ be the linear system on $X$ generated by $x_0^a, x_0^{a-1} x_1, \dots, x_1^a$ and $y$.
For complex numbers $\lambda, \mu$, we define $S_{\lambda} := (x_1 - \lambda x_0 = 0)_X$ and $T_{\mu} := (y - \mu x_0^a = 0)_X$.
We assume that $\Gamma$ is an irreducible and reduced curve and that $S_{\lambda}$ is normal for a general $\lambda \in \mbC$.
We define $C_{\lambda,\mu}$ to be the component of $S_{\lambda} \cap T_{\mu}$ other than $\Gamma$ so that $T_{\mu} |_{S_{\lambda}} = \gamma \Gamma + C_{\lambda,\mu}$ for some $\gamma > 0$.
Let $\pi \colon X \ratmap \mbP (1,1,a)$ be the projection to the coordinates $x_0,x_1,y$ and $\pi_{\lambda} = \pi |_{S_{\lambda}} \colon S_{\lambda} \ratmap \mbP (1,a) \cong \mbP^1$ its restriction to $S_{\lambda}$.
The restriction $\mcH|_{S_{\lambda}}$ has a base component $\Gamma$ and let $\mcL_{\lambda}$ be the movable part of $\mcH|_{S_{\lambda}}$.
We see that $\mcL_{\lambda}$ is a pencil of curves $C_{\lambda,\mu}$ and it defines $\pi_{\lambda}$.
We shall construct a birational involution of $X$ which is a Sarkisov link centered at $\msp$ under the following condition.

\begin{Cond} \label{condinvbir}
\begin{enumerate}
\item $\Gamma$ is an irreducible and reduced curve and $S_{\lambda}$ is normal for a general $\lambda \in \mbC$.
\item For a general $\lambda \in \mbC$, $C_{\lambda,\mu}$ is irreducible and reduced for every $\mu \in \mbC$.
\item For a general $\lambda \in \mbC$, the indeterminacy locus of $\pi_{\lambda}$ consists of two distinct points $\msp$ and $\msq$ which do not depend on $\lambda$.
\item There are an extremal divisorial extraction $\varphi \colon Y \to X$ cetered at $\msp$ with exceptional divisor $E$ and a birational morphism $\psi \colon W \to Y$ from a normal projective $\mbQ$-factorial variety $W$ with the following properties:
\begin{enumerate}
\item $\psi$ is an isomorphism over $Y \setminus \{\msq_Y\}$, where $\msq_Y = \varphi^{-1} (\msq)$.
\item $(-K_W \cdot \hat{\Gamma}) < 0$ and $(-K_W \cdot \hat{C}_{\lambda,\mu}) = 0$, where $\hat{\Gamma}$ and $\hat{C}_{\lambda,\mu}$ are the proper transforms of $\Gamma$ and $C_{\lambda,\mu}$ on $W$, respectively.
\item We have $(\hat{E} \cdot \hat{C}_{\lambda,\mu}) = m$ for some positive integer $m$, where $\hat{E}$ is the proper transform of $E$ on $W$.
\item There is a prime $\psi$-exceptional divisor $F$ such that $(F \cdot \hat{C}_{\lambda,\mu}) = 1$, and the other $\psi$-exceptional divisors (if exist) are disjoint from $\hat{C}_{\lambda,\mu}$.
\item The proper transform $\mcH_W$ of $\mcH$ on $W$ is a sub linear system of $\left| - a K_W \right|$, $\Bs \mcH_W = \hat{\Gamma}$ and the pair $(W, \frac{1}{a} \mcH_W)$ is terminal.
\item Let $\hat{S}_{\lambda}$ be the proper transform of $S_{\lambda}$ on $W$.
Then, for a general $\lambda \in \mbC$, the composite of $(\varphi \circ \psi)|_{\hat{S}_{\lambda}} \colon \hat{S}_{\lambda} \to S_{\lambda}$ and $\pi_{\lambda}$ is a morphism, which we denote by $\hat{\pi}_{\lambda} \colon \hat{S}_{\lambda} \to \mbP^1$, and the intersections $\hat{E}_{\lambda} := \hat{S}_{\lambda} \cap \hat{E}$ and $\hat{F}_{\lambda} := \hat{S}_{\lambda} \cap F$ are both irreducible.
\end{enumerate}
\item For a general $\lambda \in \mbC$, the support of the divisor $(x_0 = 0)_{S_{\lambda}} \subset S_{\lambda}$ consists of curves whose intersection form is nondegenerate. 
\end{enumerate}
\end{Cond}

\begin{Thm} \label{thm:constinvbir}
Suppose that \emph{Condition \ref{condinvbir}} is satisfied and that $\msp$ is a maximal center.
Then there is a birational involution $\tau \colon X \ratmap X$ which is a Sarkisov link centered at $\msp$.
\end{Thm}

\begin{proof}
We shall first construct a birational involution $\tau$ and then show that $\tau$ is indeed a Sarkisov link.
We see that $\hat{\Gamma}$ is the unique irreducible curve on $W$ which intersects $-K_W$ negatively since $\Bs \mcH_W = \hat{\Gamma}$ and $\mcH_W \sim_{\mbQ} - a K_W$ by (4-e) of Condition \ref{condinvbir}.
By (4-e), we can pick a sufficiently small $\varepsilon > 0$ such that the pair $(W, (\frac{1}{a} + \varepsilon) \mcH_W)$ is terminal.
Since $K_W + (\frac{1}{a} + \varepsilon) \mcH_W \sim_{\mbQ} - a \varepsilon K_W$, there is a log flip $\chi \colon W \ratmap U$ along $\hat{\Gamma}$.
Let $\mcH_U$ be the proper transform of $\mcH_W$ on $U$.
An irreducible curve which intersects $K_U + (\frac{1}{a} + \varepsilon) \mcH_U \sim_{\mbQ} \varepsilon \mcH_U$ negatively must be contained in $\Bs \mcH_U$ and $\Bs \mcH_U$ is contained in the flipped curve which intersects $-K_U$ positively .
Thus $K_U + (\frac{1}{a} + \varepsilon) \mcH_U \sim_{\mbQ} - a \varepsilon K_U$ is nef.
By the log abundance (cf.\ \cite{KMM}), $-K_U$ is semiample.
Let $\eta \colon U \to \Sigma$ be the morphism defined by $\left| - l K_U \right|$ for a sufficiently divisible $l > 0$.
The curve $\hat{C}_{\lambda,\mu}$ is disjoint from $\hat{\Gamma}$ since $(-K_W \cdot \hat{C}_{\lambda,\mu}) = 0$ by (4-b), $\Bs \mcH_W = \hat{\Gamma}$ and $\mcH_W \sim_{\mbQ} - a K_W$. 
Thus $(-K_U \cdot \check{C}_{\lambda,\mu}) = 0$, where $\check{C}_{\lambda,\mu}$ is the proper transform of $\hat{C}_{\lambda,\mu}$ on $U$, and $\eta$ contracts $\check{C}_{\lambda,\mu}$.
It follows that $\eta$ is an elliptic fibration and we have a birational map $\theta \colon \Sigma \ratmap \mbP (1,1,a)$ such that the diagram
\[
\xymatrix{
W \ar[d]_{\psi} \ar@{-->}[rr]^{\chi} & & U \ar[dd]^{\eta} \\
Y \ar[d]_{\varphi} & & \\
X \ar@{-->}[r]^{\pi} & \mbP (1,1,a) & \Sigma \ar@{-->}[l]_{\hspace{8mm} \theta}}
\]
commutes.

Let $\check{E}$ and $\check{F}$ be the proper transforms of $E$ and $F$, respectively, on $U$.
Since $\check{C}_{\lambda,\mu}$ is a fiber of $\eta$ and $\hat{C}_{\lambda,\mu}$ is disjoint from the flipping curve $\hat{\Gamma}$, assumptions that $(\hat{E} \cdot \hat{C}_{\lambda,\mu}) = m$ and $(F \cdot \hat{C}_{\lambda,\mu}) = 1$ imply that $\check{E}$ is an $m$-ple section of $\eta$ and $\check{F}$ is a section of $\eta$.

Let $\tau_U$ be the birational involution of $U$ which is the reflection of the general fiber of $\eta$ with respect to $\check{F}$.
Let $\tau_W$, $\tau_Y$ and $\tau$ be the birational involutions of $W$, $Y$ and $X$, respectively, induced by $\tau_U$.
We see that $U$ has only terminal singularities since the pair $(U,(\frac{1}{a} + \varepsilon) \mcH_U)$ is terminal.
It follows that $\tau_U$ is small since $K_U$ is $\eta$-nef.
Then $\tau_W$ is small since $\chi$ is small.
Since $F$ is the unique $\psi$-exceptional divisor which intersects $\hat{C}_{\lambda,\mu}$, the proper transforms on $U$ of the other $\psi$-exceptional divisors lie on the fibers of $\eta$ and thus they are $\tau_U$-invariant. 
It follows that $\tau_Y$ is small.

Now that we have constructed the birational involution $\tau$ of $X$.
We shall show that $\tau$ is not biregular.
In the following we assume that $\tau$ is biregular and derive a contradiction.

We fix a general $\lambda \in \mbC$ such that $C_{\lambda,\mu}$ is irreducible and reduced for every $\mu \in \mbC$ and consider the surface $S_{\lambda} = (y - \lambda x = 0)_X$ which is clearly $\tau$-invariant.
Hence $\tau$ induces a biregular involution $\tau_{\lambda}$ of $S_{\lambda}$.
Let $\bar{S}_{\lambda} \to \hat{S}_{\lambda}$ be a composite of blowups so that the birational involution $\bar{\tau}_{\lambda}$ of $\bar{S}_{\lambda}$ induced by $\tau_{\lambda}$ is biregular.
We define $\bar{\pi}_{\lambda} \colon \bar{S}_{\lambda} \to \mbP^1$ (resp.\ $\sigma \colon \bar{S}_{\lambda} \to S_{\lambda}$) to be the composite of $\bar{S}_{\lambda} \to \hat{S}_{\lambda}$ and $\pi_{\lambda}$ (resp.\ $\hat{S}_{\lambda} \to S_{\lambda}$).
By (4-f) of Condition \ref{condinvbir}, $\bar{\tau}_{\lambda}$ is indeed everywhere defined and there exist exactly two $\sigma$-exceptional prime divisors which are not contracted by $\bar{\pi}_{\lambda}$.
We denote them by $\bar{E}_{\lambda}$ and $\bar{F}_{\lambda}$.
Note that $\bar{E}_{\lambda}$ and $\bar{F}_{\lambda}$ are the proper transform of $\hat{E}_{\lambda}$ and $\hat{F}_{\lambda}$, respectively.
We denote by $G_1, \dots, G_r$ the other $\sigma$-exceptional prime divisors. 

Let $C_{\lambda}$ be a general fiber of $\pi_{\lambda}$ and $\bar{C}_{\lambda}$ its proper transform on $\bar{S}_{\lambda}$, which is clearly $\bar{\tau}_{\lambda}$-invariant.
We see that $\bar{\tau}_{\lambda} |_{\bar{C}_{\lambda}}$ is the reflection with respect to the point $\bar{F}_{\lambda} \cap \bar{C}_{\lambda}$ and that $\bar{E}_{\lambda}$ is $\bar{\tau}_{\lambda}$-invariant since $\tau_{\lambda}$ is biregular.
It follows that $(\bar{E}_{\lambda} - m \bar{F}_{\lambda})|_{\bar{C}_{\lambda}} \in \Pic^0 (\bar{C}_{\lambda})$ is a $2$-torsion.
In particular, $\bar{E}_{\lambda} - m \bar{F}_{\lambda}$ is numerically equivalent to a linear combination of $\bar{\pi}_{\lambda}$-vertical divisors.

Let $\Gamma, \Delta_1, \dots, \Delta_k$ be the irreducible and reduced component of the support of $(x_0 = 0)_{S_{\lambda}}$.
We see that $C_{\lambda,\mu}$'s and $(x_0 = 0)_{S_{\lambda}}$ exhaust the fibers of $\pi_{\lambda}$.
Since $C_{\lambda,\mu}$ is irreducible and reduced for every $\mu \in \mbC$ and all the fibers of $\bar{\pi}_{\lambda}$ are numerically equivalent to each other, there are rational numbers $\gamma, \delta_1,\dots,\delta_k, c_1,\dots,c_r$ such that
\[
\bar{E}_{\lambda} - m \bar{F}_{\lambda} \sim_{\mbQ} \gamma \bar{\Gamma} + \sum_{i = 1}^k \delta_i \bar{\Delta}_i + \sum_{i = 1}^r c_i G_i,
\]
where $\bar{\Gamma}$ and $\bar{\Delta}_i$ are the proper transforms of $\Gamma$ and $\Delta_i$, respectively.
We have $(\gamma,\delta_1,\dots,\delta_k) \ne (0,0,\dots,0)$ since the curves $\bar{E}_{\lambda}, \bar{F}_{\lambda}, G_1,\dots,G_r$ are $\sigma$-exceptional and their intersection form is negative-definite.
This implies that the intersection form of $\Gamma, \Delta_1,\dots,\Delta_k$ is degenerate since $\gamma \Gamma + \delta_1 \Delta_1 + \cdots \delta_k \Delta_k \sim_{\mbQ} 0$ and $(\gamma, \delta_1, \dots, \delta_k) \ne (0,0,\dots,0)$.
This contradicts (5) of Condition \ref{condinvbir} and thus the birational involution $\tau$ of $X$ is not biregular.
Therefore $\tau$ is a Sarkisov link starting with the extremal divisorial extraction $\varphi$ by Lemma \ref{untwistbirmap}.
\end{proof}

\subsection{The family $\mcG'_{19}$}

Let $X' = X'_8 \subset \mbP (1,1,2,3,2)$ be a member of $\mcG'_{19}$ and $\msp$ a point of type $\frac{1}{2} (1,1,1)$.
If there is no WCI curve of type $(1,1,2)$ passing through $\msp$, then it is proved in Theorem \ref{thm:19EI} that there is an elliptic involution centered at $\msp$.
In this subsection we assume that there is a WCI curve on $X'$ of type $(1,1,2)$ passing through $\msp$ and we shall construct a birational involution $\tau$ of $X'$ which untwists the maximal center $\msp$.

\begin{Lem}
We can choose homogeneous coordinates $x_0,x_1,y,z$ and $w$ with the following properties.
\begin{enumerate}
\item $\msp = \msp_2$.
\item Defining polynomial of $X'$ can be written as
\[
F = w^2 y (y + f_2) + w g_6 + h_8,
\]
where 
\[
\begin{split}
g_6 &= y^3 + y^2 a_2 + y (z a_1 + a_4) + z^2 + a_6, \\  
h_8 &= y^2 (z x_0 + b_4) + y (z b_3 + b_6) + z^2 b_2 + z b_5 + b_8,
\end{split}
\]
for some $f_2, a_i, b_i \in \mbC [x_0,x_1]$.
\end{enumerate}
\end{Lem}

\begin{proof}
The defining polynomial of $X'$ is of the form $w^2 y (y+f_2) + w g_6 + h_8$, where $f_2 \in \mbC [x_0,x_1]$ and $g_6,h_8 \in \mbC [x_0,x_1,y,z]$.
By replacing $w \mapsto w - \gamma y$ for a suitable $\gamma \in \mbC$, we may assume $\msp = \msp_2$.
By quasismoothness of $X'$ at $\msp$, $y^3 \in g_6$ and thus we may assume that the coefficient of $y^3$ in $g_6$ is $1$.
If $h_8 = y^3 h_2 + (\text{other terms})$ for some $h_2 \in \mbC [x_0,x_1]$, then we may assume that $h_2 = 0$ by replacing $w$ with $w - h_2$.
Thus can write 
\[
g_6 =  y^3 + y^2 a_2 + y (z a_1 + a_4) + z^2 + z a_3 + a_6
\]
and 
\[
h_8 = y^2 (z b_1 + b_4) + y (\alpha z^2 + z b_3 + b_6) + z^2 b_2 + z b_5 + b_8
\] 
for some $\alpha \in \mbC$ and $a_i,b_i \in \mbC [x_0,x_1]$.
We see that $X'$ contains a WCI curve of type $(1,1,2)$ passing through $\msp$ if and only $\alpha = 0$.
Finally, we have $b_1 \ne 0$ because otherwise $X'$ has a singular point at $(0 \!:\! 0 \!:\! -1 \!:\! 1 \!:\! 0)$.
Thus we may assume that $b_1 = x_0$ after replacing $x_0$ and $x_1$.
This completes the proof.
\end{proof}

In the following, we fix homogeneous coordinates of $\mbP (1,1,2,3,2)$ as in the above lemma.
Let $\varphi' \colon Y' \to X'$ be the Kawamata blowup at $\msp$ with exceptional divisor $E$.
Let $\msq \in Y'$ be the singular point of type $\frac{1}{3} (1,1,2)$ which is the inverse image of the point $\msp_3$ by $\varphi'$.
Let $Z \to Y'$ be the Kawamata blowup at $\msq$ with exceptional divisor $F \cong \mbP (1,1,2)$ and $W \to Z$ the Kawamata blowup at the $\frac{1}{2} (1,1,1)$ point contained in $F$ with exceptional divisor $G$.
We denote by $\psi \colon W \to Y'$ the composite of the above Kawamata blowups.
For complex numbers $\lambda,\mu$, we define $S_{\lambda} = (x_1 - \lambda x_0 = 0)_{X'}$ and $T_{\mu} = (w - \mu x_0^2 = 0)_{X'}$.
We see that $S_{\lambda}$ is normal for a general $\lambda \in \mbC$.
We set $\bar{F} = F (x_0,\lambda x_1,y,z,\mu x_0^4) \in \mbC [x_0,y,z]$.
Since $F$ is contained in the ideal $(x_0,x_1,w)$ and $y^2 z x_0 \in F$, $\bar{F}$ is divisible by $x_0$ but not by $x_0^2$. 
We have $T_{\mu} |_{S_{\lambda}} = \Gamma + C_{\lambda,\mu}$, where $\Gamma = (x_0 = x_1 = w = 0)$ and 
\[
C_{\lambda,\mu} = (x_1 - \lambda x_0 = w - \mu x_0^2 = \bar{F}/x_0 = 0).
\]
For a curve or a divisor $\Delta$ on $X'$, we denote by $\tilde{\Delta}$ the proper transform of $\Delta$ on $Y'$, and for a curve or a divisor $\Delta$ on $X'$, $Y'$ or $Z$, we denote by $\hat{\Delta}$ the proper transform of $\Delta$ on $W$.

\begin{Lem} \label{lem:G19spintenum}
We have
\[
(-K_{Y'} \cdot \tilde{C}_{\lambda,\mu}) = \frac{2}{3}
\]
and
\[
(\hat{E} \cdot \tilde{C}_{\lambda,\mu}) = (G \cdot \hat{C}_{\lambda,\mu}) = 1, (-K_W \cdot \hat{\Gamma}) = -1 \text{ and } (-K_W \cdot \hat{C}_{\lambda,\mu}) = 0.
\]
Moreover $\hat{F}$ is disjoint from $\hat{C}_{\lambda,\mu}$.
\end{Lem}

\begin{proof}
We have $(E \cdot \tilde{\Gamma}) = 1$ and $(-K_{X'} \cdot \Gamma) = 1/6$, which implies
\[
(-K_{Y'} \cdot \tilde{\Gamma}) = (-K_{X'} \cdot \Gamma) - \frac{1}{2} (E \cdot \tilde{\Gamma}) = \frac{1}{6} - \frac{1}{2} = - \frac{1}{3}.
\]
Since $\tilde{S}_{\lambda} \sim_{\mbQ} -K_{Y'}$ and $\tilde{T}_{\mu} \sim_{\mbQ} - 2 K_{Y'}$, we have
\[
(-K_{Y'} \cdot \tilde{\Gamma} + \tilde{C}_{\lambda,\mu})
= (-K_{Y'} \cdot \tilde{S}_{\lambda} \cdot \tilde{T}_{\lambda}) 
= 2 (-K_{Y'})^3 = \frac{1}{3}.
\]
This implies $(-K_{Y'} \cdot \tilde{C}_{\lambda,\mu}) = 2/3$.

We have 
\[
K_W = \psi^* K_{Y'} + \frac{1}{3} \hat{F} + \frac{2}{3} G.
\]
It is easy to see that both $\hat{\Gamma}$ and $\hat{C}_{\lambda,\mu}$ are disjoint from $\hat{F}$, and $(G \cdot \hat{\Gamma}) = (G \cdot \hat{C}_{\lambda,\mu}) = 1$.
Thus we get
\[
(-K_W \cdot \hat{\Gamma}) = (- K_{Y'} \cdot \tilde{\Gamma}) - \frac{1}{3} (\hat{F} \cdot \hat{\Gamma}) - \frac{2}{3} (G \cdot \hat{\Gamma}) 
= - \frac{1}{3} - \frac{2}{3} = -1.
\]
The computation $(-K_W \cdot \hat{C}_{\lambda,\mu}) = 0$ can be done similarly. 
\end{proof}

\begin{Prop} \label{prop:G'19spexcl}
If there are infinitely many pairs $(\lambda,\mu)$ of complex numbers such that $C_{\lambda,\mu}$ is reducible, then $\msp$ is not a maximal center.
\end{Prop}

\begin{proof}
Let $(\lambda,\mu)$ be a pair such that $C := C_{\lambda,\mu}$ is reducible.
The curve $C$ is defined by three equations $x_1 - \lambda x_0 = w - \mu x_0^2 = 0$ and
\[
\alpha z^2 x_0 + (y^2 + \beta y x^2_0 + \gamma x^2_0)z + \mu y^3 x_0 + \delta y x^5_0 + \varepsilon x^7_0 = 0,
\]
where $\alpha,\beta,\cdots,\varepsilon$ are complex numbers which are polynomials in $\lambda,\mu$.
It follows that $C = \Delta_1 + \Delta_2$, where 
\[
\Delta_1 = (x_1 - \lambda x_0 = w - \mu x_0^2 = z + c_3 = 0)
\] 
and 
\[
\Delta_2 = (x_1 - \lambda x_0 = w - \mu x_0^2 = \alpha z x_0 + d_4 = 0)
\] 
for some $c_3, d_4 \in \mbC [x_0,y]$ such that $y^2 \in d_4$.
Since $\Delta_2$ does not pass through $\msp$, we have $(-K_{Y'} \cdot \tilde{\Delta}_2) = (-K_{X'} \cdot \Delta_2) = 2/3$.
It follows that 
\[
(-K_{Y'} \cdot \tilde{\Delta}_1) = (-K_{Y'} \cdot \tilde{C}) - (-K_{Y'} \cdot \tilde{\Delta}_2) = 2/3 - 2/3 = 0
\] 
and $(E \cdot \tilde{\Delta}_1) > 0$ since $\Delta_1$ passes through $\msp$.
This shows that there are infinitely many curves on $Y'$ which intersects $-K_{Y'}$ non-positively and $E$ positively.
By Lemma \ref{crispnegdef}, $\msp$ is not a maximal center.
\end{proof}

This implies that if we assume that $\msp$ is a maximal singularity, then (2) of Condition \ref{condinvbir} is satisfied. 

Let $\pi \colon X' \ratmap \mbP (1,1,2)$ be the projection to the coordinates $x_0,x_1,w$ which is defined outside $\Gamma$.
Let $\mcH \subset |-2K_{X'}|$ be the linear system on $X'$ generated by $x_0^2, x_0 x_1, x_1^2$ and $w$, and let $\mcH_W$, $\mcH_{Y'}$ be the proper transform of $\mcH$ on $Y'$, $W$, respectively.
Note that $\mcH_{Y'} = |-2 K_{Y'}|$, $\mcH_Z = |- 2 K_Z|$ and $\mcH_W = |-2 K_W|$.
Note also that the base loci of $\mcH$, $\mcH_{Y'}$ and $\mcH_W$ are $\Gamma$, $\tilde{\Gamma}$ and $\hat{\Gamma}$, respectively.
Let $\pi_{\lambda} \colon S_{\lambda} \ratmap \mbP (1,2) \cong \mbP^1$ be the restriction of $\pi$ to $S_{\lambda}$ and $\hat{\pi}_{\lambda} \colon \hat{S}_{\lambda} \ratmap \mbP^1$ be the composite of $(\varphi' \circ \psi)|_{\hat{S}_{\lambda}} \colon \hat{S}_{\lambda} \to S_{\lambda}$ and $\pi_{\lambda}$.
We set $\hat{E}_{\lambda} = \hat{E} \cap \hat{S}_{\lambda}$ and $\hat{G}_{\lambda} = G \cap \hat{S}_{\lambda}$.
We can write $\mcH|_{S_{\lambda}} = \mcL_{\lambda} + \Gamma$, where $\mcL_{\lambda}$ is the movable part which defines $\pi_{\lambda}$.
Note that $C_{\lambda,\mu}$'s for $\mu \in \mbC$ and $2 (x_0 = 0)|_{S_{\lambda}} - \Gamma = \Gamma + 2 \Delta$ are the members of $\mcL_{\lambda}$, where $\Delta = (x_0 = x_1 = w y^2 + y^3 + z^2 = 0) \subset S_{\lambda}$.

\begin{Lem} \label{lem:G'19projindet}
Let $\lambda$ be a general complex number.
Then the following assertions hold.
\begin{enumerate}
\item The indeterminacy locus of $\pi_{\lambda}$ consists of two points $\msp$ and $\msp_3$, and $\hat{\pi}_{\lambda}$ is a morphism.
\item Both $\hat{E}_{\lambda}$ and $\hat{G}_{\lambda}$ are irreducible.
\item The pair $(W, \frac{1}{2} \mcH_W)$ is terminal.
\end{enumerate}
\end{Lem}

\begin{proof}
We have $C_{\lambda,\mu_1} \cap C_{\lambda,\mu_2} = \{ \msp,\msp_4\}$ for $\mu_1 \ne \mu_2$.
This shows that the indeterminacy locus of $\pi_{\lambda}$ consists of $\msp$ and $\msp_3$.
Let $\hat{\mcL}_{\lambda}$ be the proper transform of $\mcL_{\lambda}$ on $\hat{S}_{\lambda}$, which defines $\hat{\pi}_{\lambda}$.
We shall show that $\hat{\mcL}_{\lambda}$ is base point free and that both $\hat{E}_{\lambda}$ and $\hat{G}_{\lambda}$ are irreducible.
We see that $\varphi' \colon Y' \to X'$ can be identified with the embedded weighted blowup of $X'$ at $\msp$ with $\wt (x_0,x_1,z,w) = \frac{1}{2} (1,1,1,2)$ and we have
\[
\hat{E} \cong E \cong (w + z x_0 = 0) \subset \mbP (1,1,1,2),
\]
where $x_0,x_1,z,w$ are the homogeneous coordinates of $\mbP (1,1,1,2)$.
It follows that $\hat{E}_{\lambda} \cong E \cap \tilde{S}_{\lambda}$ is isomorphic to
\[
(w + z x_0 = x_1 - \lambda x_0 = 0) \subset \mbP (1,1,1,2),
\]
which is isomorphic to $\mbP^1$.
Recall that $C_{\lambda,\mu}$ is defined by equations $x_1 - \lambda x_0 = w - \mu x_0^2 = 0$ and
\[
\alpha z^2 x_0 + (y^2 + \beta y x^2_0 + \gamma x^4_0)z + \mu y^3 x_0 + \delta y x^5_0 + \varepsilon x^7_0 = 0.
\]
It follows that $\hat{C}_{\lambda,\mu}$ intersects $\hat{E}$ at $(1\!:\! \lambda \!:\! -\mu \!:\! \mu)$.
This shows that $\Bs \hat{\mcL}_{\lambda}$ is disjoint from $\hat{E}$. 
We see that $x_0,x_1,y,w$ vanish along $F$ to order respectively $1/3,1/3,2/3,5/3$ and thus $Z \to Y'$ can be identified with the embedded weighted blowup of $X'$ at $\msp_3$ with $\wt (x_0,x_1,y,w) = \frac{1}{3} (1,1,2,5)$.
Then we have
\[
F \cong (w + y^2 x_0 + y b_3 + b_5 = 0) \subset \mbP (1,1,2,5).
\]
We see that the proper transform of $C_{\lambda,\mu}$ intersects $F$ at $(0\!:\!0\!:\!1\!:\!0)$.
We introduce the orbifold chart of $Z$.
Let $X'_z$ be the open subset $(z \ne 0)$, which is isomorphic to the hyperquotient
\[
(F (x_0,x_1,y,1,w) = 0) \subset \mbA^4/\mbZ_3 (1,1,2,2).
\]
Let $\tilde{\mbA}^4$ be the affine $4$-space with affine coordinates $\tilde{x}_0, \tilde{x}_1,y'$ and $\tilde{w}$, and let $\tilde{\mbA}^4 \to \mbA^4$ be the morphism determined by
\[
x_i \mapsto \tilde{x}_i y', y \mapsto {y'}^2, w \mapsto \tilde{w} {y'}^5.
\]
We refer the readers to \cite[6.38 in Chapter 6]{KSC} and \cite[Section 4.1]{Okada} for orbifold coordinates of weighted blowups.
We see that the morphism factors through $V := \tilde{\mbA}^4/\mbZ_3 (0,0,1,0)$ and $V$ is isomorphic to the affine $4$-space with affine coordinates $\tilde{x}_0, \tilde{x}_1, \tilde{y} = {y'}^3$ and $\tilde{w}$.
Let $Z_y$ be the subscheme of $V$ defined by the equation
\[
\tilde{F} := F (\tilde{x}_0 y', \tilde{x}_1 y', {y'}^3, 1, \tilde{w} {y'}^5)/ {y'}^5 = 0.
\]
Note that $\tilde{F} \in \mbC [\tilde{x}_0,\tilde{x}_1,\tilde{y},\tilde{w}]$ under the identification $\tilde{y} = {y'}^3$ and it is of the form $\tilde{F} = \tilde{w} + \tilde{x}_0 + \tilde{y} h$ for some $h \in \mbC [\tilde{x}_0, \tilde{x}_1, \tilde{y}, \tilde{w}]$ whose constant term is zero.
Note also that $Z_y$ is isomorphic to an open subset of $Z$ and its origin corresponds to the $\frac{1}{2} (1,1,1)$ point of $Z$ lying on $F$.
Then the Kawamata blowup of $Z$ at the $\frac{1}{2} (1,1,1)$ point is the blowup of $Z_y$ at the origin.
It follows that
\[
G \cong (\tilde{w} + \tilde{x}_0 = 0) \subset \mbP^3,
\]
where $\tilde{x}_0, \tilde{x}_1,\tilde{y},\tilde{w}$ are the homogeneous coordinates of $\mbP^3$.
We see that $\hat{C}_{\lambda,\mu}$ intersects $G$ at $(1 \!:\! \lambda \!:\! -\mu \!:\! 1)$.
This shows that $\Bs \hat{\mcL}_{\lambda}$ is disjoint from $G$.
Moreover $\hat{G}_{\lambda}$ is isomorphic to 
\[
(\tilde{x}_1 - \lambda \tilde{x}_0 = \tilde{w} + \tilde{x}_0 = 0) \subset \mbP^4.
\]
Therefore (1) and (2) are verified.

We see that $\hat{\Gamma}$ is a nonsingular curve, $W$ is nonsingular along $\hat{\Gamma}$ and the multiplicity of a general member of $\mcH_W$ along $\hat{\Gamma}$ is $1$.
From this we deduce that $(W, \frac{1}{2} \mcH_W)$ is terminal.
Thus (3) is verified.
\end{proof}

\begin{Lem} \label{lem:G'19nondegint}
For a general $\lambda \in \mbC$, the intersection form of $(x_0 = 0)|_{S_{\lambda}} = \Gamma + \Delta$ is nondegenerate.
\end{Lem}

\begin{proof}
Since $\Gamma + \Delta \sim_{\mbQ} A|_{S_{\lambda}}$ and $\Gamma$ intersects $\Delta$ at one nonsingular point, we have $(\Gamma + \Delta \cdot \Gamma) = (A \cdot \Gamma) = 1/6$, $(\Gamma + \Delta \cdot \Delta) = (A \cdot \Delta) = 1/2$ and $(\Gamma \cdot \Delta) = 1$.
Thus, we get
\[
\begin{pmatrix}
(\Gamma^2) & (\Gamma \cdot \Delta) \\
(\Gamma \cdot \Delta) & (\Delta^2)
\end{pmatrix}
=
\begin{pmatrix}
- \frac{5}{6} & 1 \\
1 & - \frac{1}{2}
\end{pmatrix}.
\]
Therefore the intersection form of $(x_0 = 0)|_{S_{\lambda}} = \Gamma + \Delta$ is nondegenerate.
\end{proof}

\begin{Thm} \label{thm:G'19invbir}
If $\msp$ is a maximal center, then there is a birational involution $\tau \colon X' \ratmap X'$ which is a Sarkisov link centered at $\msp$.
\end{Thm}

\begin{proof}
By Theorem \ref{thm:constinvbir}, it is enough to show that Condition \ref{condinvbir} is satisfied.
This follows from Lemmas \ref{lem:G19spintenum}, \ref{lem:G'19projindet}, \ref{lem:G'19nondegint} and Proposition \ref{prop:G'19spexcl}.
Here we note that Proposition \ref{prop:G'19spexcl} indeed implies (2) of Condition \ref{condinvbir} since we are assuming that $\msp$ is a maximal singularity.
\end{proof}

\subsection{The family $\mcG_{19}$} \label{sec:spfam19}

Let $X = X_{6,8} \subset \mbP := \mbP (1^2,2,3,4^2)$ be a quasismooth member of the family No.\ $19$.
In \cite{Okada}, we prove Theorem \ref{thm:prev} assuming the following generality condition: $X$ does not contain a WCI curve of type $(1,1,4,6)$.
This condition is equivalent to the condition that the scheme $(x_0 = x_1 = 0)_X$ is irreducible.
In the following, we assume that $(x_0 = x_1 = 0)_X$ is reducible.
In \cite{Okada}, this generality condition is only used to exclude singular points of type $\frac{1}{2} (1,1,1)$, hence it is enough to consider those points.

Let $\msp \in X$ be a singular point of type $\frac{1}{2} (1,1,1)$.
The aim of this subsection is to construct a birational involution of $X$ which untwists the maximal singularity centered at $\msp$ under the assumption that $\msp$ is a maximal center.
Let $x_0,x_1,y,z,s_0,s_1$ be homogeneous coordinates of $\mbP (1,1,2,3,4,4)$ of degree respectively $1,1,2,3,4,4$.

\begin{Lem}
We can choose homogeneous coordinates $x_0,x_1,y,z,s_0,s_1$ with the following properties.
\begin{enumerate}
\item $\msp = \msp_2$.
\item Defining polynomials of $X$ can be written as
\[ 
\begin{split}
F_1 &= y s_0 + z^2 + s_1 a_2 + a_6, \\ 
F_2 &= y^2 s_1 + y g_6 + h_8,
\end{split}
\]
where
\[
\begin{split} 
g_6 &= s_0 b_2 + s_1 b'_2 + z b_3 + b_6, \\ 
h_8 &= s_1^2 + s_0 s_1 + s_0 (z x_0 + c_4) + s_1 (z c_1 + c'_4) + z^2 c_2 + z c_5 + c_8,
\end{split}
\]
for some $a_i,b_i,b'_2,c_i,c'_4 \in \mbC [x_0,x_1]$.
\end{enumerate}
\end{Lem}

\begin{proof}
By replacing coordinates, we can assume $\msp = \msp_2$.
Then we can write defining polynomials as $F_1 = y s_0 + f_6$ and $F_2 = y^2 s_1 + y g_6 + h_8$, where $f_6,g_6,h_8 \in \mbC [x_0,x_1,z,s_0,s_1]$.
Replacing $y$ with $y - f'_2$, where $f'_2 = f'_2 (x_0,x_1)$ is a polynomial such that $f_6 = s_0 f'_2 + f'_6$ for $f'_6 = f'_6 (x_0,x_1,z,s_1)$, we may assume that $f_6$ does not involve $s_0$.
Quasismoothness of $X$ implies that $X$ does not pass through $\msp_3$, that is, $z^2 \in f_6$.
Then, replacing $z$ suitably, we may assume that $f_6 = z^2 + s_1 a_2 + a_6$ for some $a_i \in \mbC [x_0,x_1]$.
Let $\alpha \in \mbC$ be the coefficient of $z^2$ in $g_6$.
By replacing $F_2$ with $F_2 - \alpha y F_1$ and $s_1$ with $s_1 - \alpha s_0$, we may assume that $\alpha = 0$.
Thus, $g_6 = s_0 b_2 + s_1 b'_2 + z b_3 + b_6$ for some $b_i, b'_2 \in \mbC [x_0,x_1]$.
We write $h_8 = \beta s_0^2 + \gamma s_0 s_1 + \delta s_1^2 + s_0 h'_4 + s_1 h''_4 + h'_8$, where $h'_4,h''_4,h'_8 \in \mbC [x_0,x_1,z]$.
We shall show that $\beta = 0$, $\gamma \ne 0$ and $\delta \ne 0$.
We have
\[
(x_0 = x_1 = 0)_X 
= (x_0 = x_1 = y s_0 + z^2 = y^2 s_1 + \beta s_0^2 + \gamma s_0 s_1 + \delta s_1^2).
\]
By the assumption, $(x_0 = x_1 = 0)_X$ is reducible, which is equivalent to the condition $\beta = 0$.
It follows that $X$ passes through $\msp_4$.
Since $X$ is quasismooth at $\msp$, we have $\gamma \ne 0$.
If $\delta = 0$, then $X$ passes through $\msp_5$.
But then $X$ cannot be quasismooth at $\msp_5$ since $y s_1 \notin F_1$.
This shows $\delta \ne 0$.
By re-scaling coordinates, we may assume $\gamma = \delta = 1$.
It remains to show that we can write $h_4 = z x_0 + c_4$ for some $c_4 = c_4 (x_0,x_1)$ after replacing $x_0$ and $x_1$.
If neither $z x_0 \in h_4$ nor $z x_1 \in h_4$, then $X$ is not quasismooth at $(0 \!:\! 0 \!:\! 1 \!:\! 1 \!:\! -1 \!:\! 0) \in X$.
This is a contradiction and thus at least one of $z x_0$ and $z x_1$ appears in $h_4$.
Therefore, we can write $h_4 = z x_0 + c_4$ for some $c_4$ after replacing $x_0$ and $x_1$.
This completes the proof.
\end{proof}

In the following, we  fix homogeneous coordinates of $\mbP (1,1,2,3,4,4)$ as in the above lemma.
Let $\varphi \colon Y \to X$ be the Kawamata blowup of $X$ at $\msp$ with exceptional divisor $E$.
Let $\msq \in Y$ be the singular point of type $\frac{1}{4} (1,1,3)$ which is the inverse image of $\msp_4 \in X$ by $\varphi$.
Let $\psi \colon W \to Y$ be the Kawamata blowup of at $\msq$ with exceptional divisor $F \cong \mbP (1,1,3)$.
For complex numbers $\lambda$ and $\mu$, we set $S_{\lambda} := (x_1 - \lambda x_0 = 0)_X$ and $T_{\lambda} := (s_1 - \mu x_0^4 = 0)_X$.
We see that $S_{\lambda}$ is normal for a general $\lambda \in \mbC$.
For $i = 1,2$, we define $\bar{F}_i = F_1 (x_0,\lambda x_1,y,z,s_0,\mu x_0^4) \in \mbC [x_0,y,z,s_0]$.
We have $T_{\lambda}|_{S_{\lambda}} = \Gamma + C_{\lambda,\mu}$, where 
\[
\Gamma = (x_0 = s = s_1 = y s_0 + z^2 = 0)
\] 
and 
\[
C_{\lambda,\mu} = (x_1 - \lambda x_0 = s_1 - \mu x_0^4 = \bar{F}_1 = \bar{F}_2/x_0 = 0).
\]
Note that $\bar{F}_2$ is indeed divisible by $x_0$ but not by $x_0^2$ since $F_1 \in (x_0,x_1,s_1)$ and $s_0 z x_0 \in F_2$.
By eliminating $x_1$ and $w$, the curve $C_{\lambda,\mu}$ is isomorphic to the weighted complete intersection in $\mbP (1,2,3,4)$ with homogeneous coordinates $x_0,y,z,s_0$ defined by the equations of the form $y s_0 + z^2 + \alpha x_0^6 = d_3 s + e_7 = 0$ for some $d_3, e_7 \in \mbC [x_0,y,z]$ such that $z y^2 \notin e_7$, $y^3 x_0 \notin e_7$ and the coefficient of $z$ in $d_3$ is $1$. 
For a curve or a divisor $\Delta$ on $X$, we denote by $\tilde{\Delta}$ the proper transform of $\Delta$ on $Y$, and for a curve of a divisor $\Delta$ on $X$ or $Y$, we denote by $\hat{\Delta}$ the proper transform of $\Delta$ on $W$.

\begin{Lem} \label{lem:G19defeq}
Let $C$ be a curve defined in the weighted projective space $\mbP (1,2,3,4)$ with homogeneous coordinates $x,y,z,s$ by the equations
\[
y s + z^2 + \alpha x^6 = d_3 s + e_7 = 0,
\]
where $\alpha \in  \mbC$ and $d_3, e_7 \in \mbC [x,y,z]$.
Assume that the coefficient of $z$ in $d_3$ is nonzero and that neither $z y^2 \in e_7$ nor $y^3 x \in e_7$.
Then the following hold.
\begin{enumerate}
\item $C$ contains at most one component which is contracted by the projection $\mbP (1,2,3,4) \ratmap \mbP (1,2,3)$ to the coordinates $x,y,z$ and if $C$ contains such a component, then it is the curve $(y = d_3 = 0)$.
\item There is a unique component $C'$ of $C$ which passes through the point $(0 \!:\! 0 \!:\! 0 \!:\! 1)$ and if there is a component $C^{\circ}$ of $C$ other than $C'$, then $C^{\circ}$ passes through $(0\!:\!1\!:\!0\!:\!0)$.
\item If $C$ is irreducible, then it is reduced.
\end{enumerate}
\end{Lem}

\begin{proof}
The curve $C$ contains a component which is contracted by the projection $\pi \colon \mbP (1,2,3,4) \ratmap \mbP (1,2,3)$ if and only if $(y = d_3 = 0) \cap C$ is a curve.
By the assumption $z \in d_3$, the curve $(y = d_3 = 0)$ is irreducible and reduced.
This proves (1).
Since $C$ is quasismooth at $(0 \!:\!0 \!:\! 0 \!:\! 1)$, there is a unique component $C'$ of $C$ passing through the point.
Suppose that there is a component $C^{\circ}$ of $C$ other than $C'$.
Then, $C^{\circ}$ cannot be contracted by $\pi$.
The image $\overline{C} = \pi (C)$ is the curve in $\mbP (1,2,3)$ defined by $h := y e_7 - d_3 (z^2 + \alpha x^6) = 0$. 
We shall show that any component of $\overline{C}$ passes through $\msq = (0\!:\!1\!:\!0)$.
We see that $\overline{C}$ passes through $\msq$, hence, if $\overline{C}$ is irreducible, then we are done.
Assume that $\overline{C}$ is reducible.
Then, since $h$ is cubic in varible $z$, we have $h = h_1 h_2$, where $h_1 \ni z$ and $h_2 \ni z^2$ are linear and quadratic  with respect to $z$, respectively. 
The curve $\overline{C}_1 = (h_1 = 0)$ clearly passes through $\msq$.
If $y^3 \in h_2$, then $z y^3 \in h$.
This cannot happen since $z y^2 \notin e_7$.
Hence $y^3 \notin h_2$ and $\overline{C}_2 = (h_2 = 0)$ passes through $\msq$.
This shows that any component of $\overline{C}$ passes through $\msp$.
Thus $\pi (C^{\circ})$ passes through $\msq$ and this implies $(x = z = 0) \cap C^{\circ} \ne \emptyset$.
Since 
\[
\emptyset \ne (x = z = 0) \cap C^{\circ} \subset (x = z = 0) \cap C = \{ (0 \!:\!1 \!:\! 0 \!:\! 0), (0 \!:\!0 \!:\! 0 \!:\! 1) \}
\] 
and $C^{\circ}$ does not pass through $(0 \!:\!0 \!:\! 0 \!:\! 1)$, we conclude that $C^{\circ}$ passes through $(0\!:\!1\!:\! 0\!:\!0)$.
This proves (2).
Finally, since $C$ contains at least one component along which $C$ is reduced, it follows that $C$ is reduced if it is irreducible.
This proves (3).
\end{proof}

For complex numbers $\lambda, \mu$ and $\zeta$, we set
\[
C_{\lambda,\mu,\zeta} := (x_1 - \lambda x_0 = s_1 - \mu x_0^4 = y = z + \zeta x_0^3= 0) \subset \mbP (1,1,2,3,4,4).
\]
Note that $C_{\lambda,\mu,\zeta}$ does not pass through $\msp$.

\begin{Lem} \label{lem:G19Cirr}
One of the following holds.
\begin{enumerate}
\item There are infinitely many pairs $(\lambda,\mu) \in \mbC^2$ such that $C_{\lambda,\mu}$ is reducible and contains a curve $C_{\lambda,\mu,\zeta}$ as a component for some $\zeta \in \mbC$ depending on $\lambda,\mu$.
\item There are infinitely many pairs $(\lambda,\mu) \in \mbC^2$ such that $C_{\lambda,\mu}$ is reducible but does not contain a curve $C_{\lambda,\mu,\zeta}$ for any $\zeta \in \mbC$.
In this case there is a component $C^{\circ}_{\lambda,\mu}$ of $C_{\lambda,\mu}$ such that $C^{\circ}_{\lambda,\mu}$ passes through $\msp$ but does not pass through $\msp_4$.
\item For all but finite exception of pairs $(\lambda,\mu) \in \mbC^2$, the curve $C_{\lambda, \mu}$ is irreducible and reduced.
\end{enumerate}
\end{Lem}

\begin{proof}
Recall that, by eliminating $x_1$ and $s_1$, we see that $C_{\lambda,\mu}$ is isomorphic to the complete intersection in $\mbP (1,2,3,4)$ defined by equations of the form $y s_0 + z^2 + \alpha x_0^6 = d_3 s + e_7 = 0$, where $z \in d_3$ and $z y^2 \notin e_7$ and $y^3 x_0 \notin e_7$.
Note that the points $(0\!:\!1\!:\!0\!:\!0)$ and $(0\!:\!0\!:\!0\!:\!1)$ of $\mbP (1,2,3,4)$ corresponds to $\msp$ and $\msp_4$, respectively. 
Therefore, our claim follows from Lemma \ref{lem:G19defeq}.
\end{proof}

We compute several intersection numbers. 
We describe $\varphi$ and $\psi$ explicitly.
We see that $x_0,x_1,z$ vanish along $E$ to order $1/2$ and, by looking at the defining equations, $s_0$, $s_1$ vanish along $E$ to order respectively $2/2$ and $4/2$.
It follows that $\varphi \colon Y \to X$ can be realized as the embedded weighted blowup of the ambient weighted projective at $\msp$ with $\wt (x_0,x_1,z,s_0,s_1) = \frac{1}{2} (1,1,1,2,4)$.
Hence we have
\[
E \cong (s_0 + z^2 = s_1 + s_0 b_2 + z b_8 + s_0 z x_0 + z^2 c_2 = 0) \subset \mbP (1,1,1,2,4).
\]
Here $x_0,x_1,z,s_0,s_1$ are the coordinates of $\mbP (1,1,1,2,4)$.
We see that $x_0,x_1,z$ vanish along $F$ to order respectively $1/4,1/4,3/4$ and, by looking at the defining equations, $y$, $s_1$ vanish along $F$ to order respectively $6/4$, $4/4$.
It follows that $\psi \colon W \to Y$ can be realized as the embedded weighted blowup of the ambient weighted projective space at $\msp_4$ with $\wt (x_0,x_1,y,z,s_1) = \frac{1}{4} (1,1,6,3,4)$.
Hence we have
\[
F \cong (y + z^2 + s_1 a_2 + a_6 = s_1 + z x_0 + c_4 = 0) \subset \mbP (1,1,6,3,4).
\]
Here $x_0,x_1,y,z,s_1$ are the homogeneous coordinates of $\mbP (1,1,6,3,4)$.

\begin{Lem} \label{lem:G19intnum}
We have 
\[
(-K_Y \cdot \tilde{C}_{\lambda,\mu}) = \frac{1}{4}
\]
and
\[
(\hat{E} \cdot \hat{C}_{\lambda,\mu}) = 3, (F \cdot \tilde{C}_{\lambda,\mu}) = 1, (-K_W \cdot \hat{C}_{\lambda,\mu}) = 0, (-K_W \cdot \hat{\Gamma}) = -\frac{1}{3}.
\]
\end{Lem}

\begin{proof}
We see that $\tilde{\Gamma}$ intersects $E$ at one point $(0 \!:\! 0 \!:\! 1 \!:\! -1 \!:\! 0) \in \mbP (1,1,1,2,4)$ so that $(E \cdot \tilde{C}_{\lambda,\mu}) = 1$.
Hence we have
\[
(-K_Y \cdot \tilde{\Gamma}) = (-K_X \cdot \Gamma) - \frac{1}{2} (E \cdot \tilde{\Gamma}) = \frac{1}{4} - \frac{1}{2} = - \frac{1}{4}.
\]
Recall that $C_{\lambda,\mu}$ is defined by the equations $x_1 - \lambda x_0 = s_1 - \mu x_0^4 = y s_0 + z^2 + \alpha x_0^6 = d_3 s_0 + e_7 = 0$ for some $\alpha \in \mbC$ and $d_3,e_7 \in \mbC [x_0,y,z]$ such that the coefficient of $z$ in $d_3$ is $1$ and $z y^2, y^3 x_0 \notin e_7$.
Let $\beta$ be the coefficient of $y x_0$ in $d_3$ and let $\gamma,\delta$ be the coefficients of $z^2 x_0$, $z y x_0^2$ in $e_7$, respectively.
The intersection points of $\hat{C}_{\lambda,\mu}$ and $\hat{E}$ are the solutions of the equation
\[
x_1 - \lambda x_0 = s_1 - \mu x_0^4 = s_0 + z^2 = (z + \beta x_0) s_0 + \gamma z^2 x_0 + \delta z x_0^2 + \mu x_0^3 = 0
\]
in $\mbP (1,1,1,2,4)$, and they consist of $3$ points.
Hence $(\hat{E} \cdot \hat{C}_{\lambda,\mu}) = 3$.
Let $\varepsilon$ be the coefficient of $x_0^3$ in $d_3$.
Then the intersection points of $\hat{C}_{\lambda,\mu}$ and $F$ are the solutions of the equations
\[
x_1 - \lambda x_0 = s_1 - \mu x_0^4 = y + z^2 + \alpha x_0^6 = z + \varepsilon x_0^3 = 0
\]
in $\mbP (1,1,6,3,4)$, which consists of one point $(1 \!:\! \lambda \!:\! -\alpha - \gamma^2 \!:\! \mu)$.
Hence $(F \cdot \hat{C}_{\lambda,\mu}) = 1$.
Since $\tilde{S}_{\lambda} \sim_{\mbQ} - K_Y$, $\tilde{T}_{\mu} \sim_{\mbQ} - 4 K_Y$ and $(-K_Y^3) = 0$, we have
\[
0 = 4 (-K_Y^3) = (-K_Y \cdot \tilde{S}_{\lambda} \cdot \tilde{T}_{\mu}) = (-K_Y \cdot \tilde{\Gamma} + \tilde{C}_{\lambda,\mu}).
\]
This shows that 
\[
(-K_Y \cdot \tilde{C}_{\lambda,\mu}) = - (-K_Y \cdot \tilde{\Gamma}) = \frac{1}{4}.
\]
We have
\[
(-K_W \cdot \hat{C}_{\lambda,\mu}) = (-K_Y \cdot \tilde{C}_{\lambda,\mu}) - \frac{1}{4} (F \cdot \hat{C}_{\lambda,\mu}) = \frac{1}{4} - \frac{1}{4} = 0.
\]
Finally, since $\hat{S}_{\lambda} \sim_{\mbQ} - K_W$, $\hat{T}_{\mu} \sim_{\mbQ} - 4 K_W$ and $(-K_W^3) = (-K_Y^3) - (F^3)/4^3 = -1/12$, we have
\[
- \frac{1}{3} = 4 (-K_W^3) = (-K_W \cdot \hat{S}_{\lambda} \cdot \hat{T}_{\mu}) = (-K_W \cdot \hat{\Gamma} + \hat{C}_{\lambda,\mu}) = (-K_W \cdot \hat{\Gamma}).
\]
This completes the proof.
\end{proof}

Let $\mcH \subset |-4K_X|$ be the linear system generated by the sections $x_0^4,x_0^3 x_1,\dots,x_1^4$ and $s_1$, and let $\mcH_Y$, $\mcH_W$ be proper transform of $\mcH$ on $Y$ and $W$, respectively.
We see that $\mcH_Y = |-4 K_Y|$ and $\mcH_W = |-4 K_W|$ and the base loci of $\mcH$, $\mcH_Y$ and $\mcH_W$ are $\Gamma$, $\tilde{\Gamma}$ and $\hat{\Gamma}$, respectively.

\begin{Prop} \label{prop:G19excl}
If there are infinitely many pairs $(\lambda,\mu)$ of complex numbers such that $C_{\lambda,\mu}$ is reducible, then $\msp$ is not a maximal singularity.
\end{Prop}

\begin{proof}
We divide the proof into two cases according to cases (1) or (2) in Lemma \ref{lem:G19Cirr}.
We first assume that we are in case (1), that is, there are infinitely many $(\lambda,\mu)$ such that $C_{\lambda,\mu}$ contains a curve $C_{\lambda,\mu,\zeta}$ for some $\zeta$.
Let $(\lambda,\mu)$ be any pair among the above infinitely many pairs.
We write $C_{\lambda,\mu} = C'_{\lambda,\mu} + C_{\lambda,\mu,\zeta}$.
We see that $\tilde{C}_{\lambda,\mu,\zeta}$ is disjoint from $E$ since $C_{\lambda,\mu,\zeta}$ does not pass through $\msp$ so that we compute 
\[
(-K_Y \cdot \tilde{C}_{\lambda,\mu,\zeta}) = (-K_X \cdot C_{\lambda,\mu,\zeta}) = 1/4.
\]
By Lemma \ref{lem:G19intnum}, we have 
\[
(-K_Y \cdot \tilde{C}'_{\lambda,\mu}) = (-K_Y \cdot \tilde{C}_{\lambda,\mu}) - (-K_Y \cdot \tilde{C}_{\lambda,\mu,\zeta}) = 1/4 - 1/4 = 0.
\]
Note that a component of $\tilde{C}'_{\lambda,\mu}$ which is disjoint from $E$ intersects $-K_Y$ positively (if it exists) and there exists at least one component of $\tilde{C}'_{\lambda,\mu}$ which intersects $E$. 
It follows that there is a component $C^{\circ}_{\lambda,\mu}$ of $C'_{\lambda,\mu}$ such that $(-K_Y \cdot \tilde{C}^{\circ}_{\lambda,\mu}) \le 0$ and $(\tilde{C}^{\circ}_{\lambda,\mu} \cdot E) > 0$.
Therefore $\msp$ is not a maximal center by Lemma \ref{crispinfc}.

Assume that we are in case (2), that is, there are infinitely many $(\lambda,\mu)$ such that $C_{\lambda,\mu}$ is reducible but does not contain $C_{\lambda,\mu,\zeta}$ for any $\zeta$.
Let $(\lambda,\mu)$ be one of the above pair.
Let $C_{\lambda,\mu}^{\circ}$ be a component of $C_{\lambda,\mu}$ which passes through $\msp$ but does not pass through $\msp_4$.
It follows that $(E \cdot \tilde{C}^{\circ}_{\lambda,\mu}) > 0$ and $\tilde{C}^{\circ}_{\lambda,\mu}$ does not pass through $\msq$.
Note that $(-K_W \cdot \hat{\Gamma}) < 0$ by Lemma \ref{lem:G19intnum}, so that $\hat{\Gamma}$ is the only irreducible curve on $W$ which intersects $-K_W$ negatively since $\mcH_W \sim_{\mbQ} - 4 K_W$ and $\Bs \mcH_W = \hat{\Gamma}$.
It then follows from $(-K_W \cdot \hat{C}_{\lambda,\mu}) = 0$ that $(-K_W \cdot \hat{C}_{\lambda,\mu}^{\circ}) = 0$.
Hence we have $(-K_Y \cdot \tilde{C}_{\lambda,\mu}^{\circ}) = 0$ because $\tilde{C}_{\lambda,\mu}^{\circ}$ does not pass through the center $\msq$ of the Kawamata blowup $\psi \colon W \to Y$.
This shows that there are infinitely many curves $\tilde{C}^{\circ}_{\lambda,\mu}$ on $Y$ which intersects $-K_Y$ non-positively and $E$ positively.
Therefore, $\msp$ is not a maximal center by Lemma \ref{crispinfc}.
\end{proof}

From now on, we assume that we are in case (3) of Lemma \ref{lem:G19Cirr}.
Let $\pi \colon X \ratmap \mbP (1,1,4)$ be the projection to the coordinates $x_0,x_1,s_1$, which is defined outside $\Gamma$.
Let $\pi_{\lambda} \colon S_{\lambda} \ratmap \mbP (1,4) \cong \mbP^1$ be the restriction of $\pi$ to $S_{\lambda}$ and $\hat{\pi}_{\lambda} \colon \hat{S}_{\lambda} \ratmap \mbP^1$ be the composite of $(\varphi \circ \psi)|_{\hat{S}_{\lambda}} \colon \hat{S}_{\lambda} \to S_{\lambda}$ and $\pi_{\lambda}$.
We set $\hat{E}_{\lambda} = \hat{E} \cap \hat{S}_{\lambda}$ and $\hat{F}_{\lambda} = F \cap \hat{S}_{\lambda}$.
We can write $\mcH|_{S_{\lambda}} = \mcL_{\lambda} + \Gamma$, where $\mcL_{\lambda}$ is the movable part which defines $\pi_{\lambda}$.
Note that $C_{\lambda,\mu}$'s for $\mu \in \mbC$ and $4 (x_0 = 0)|_{S_{\lambda}} - \Gamma = 3 \Gamma + 4 \Delta$ are the members of $\mcL_{\lambda}$, where $\Delta = (x_0 = x_1 = s_0 + s_1 + y^2 = y s_0 + z^2 = 0)$.

\begin{Lem} \label{lem:G19S}
Let $\lambda$ be a general complex number.
The the following hold.
\begin{enumerate}
\item The indeterminacy locus of $\pi_{\lambda}$ consists of two points $\msp$ and $\msp_4$, and $\hat{\pi}_{\lambda}$ is a morphism.
\item Both $\hat{E}_{\lambda}$ and $\hat{F}_{\lambda}$ are irreducible.
\end{enumerate}
\end{Lem}

\begin{proof}
We have $C_{\lambda,\mu_1} \cap C_{\lambda,\mu_2} = \{\msp,\msp_4\}$ for $\mu_1 \ne \mu_2$.
This shows that the indeterminacy locus of $\pi_{\lambda}$ consists of $\msp$ and $\msp_4$.

Let $\hat{\mcL}_{\lambda}$ be the proper transform of $\mcL_{\lambda}$ on $\hat{S}_{\lambda}$.
We shall show that $\hat{\mcL}_{\lambda}$ is base point free and that both $\hat{E}_{\lambda}$ and $\hat{F}_{\lambda}$ are irreducible.
We use the explicit description of $\varphi$ and $\psi$ before Lemma \ref{lem:G19intnum}.
Then we see that $\hat{E}_{\lambda} \cong E \cap \tilde{S}_{\lambda}$ is isomorphic to
\[
(s_0 + z^2 = s_1 + s_0 b_2 + z b_8 + s_0 z x_0 + z^2 c_2 = x_1 - \lambda x_0 = 0) \subset \mbP (1,1,1,2,4).
\]
Hence $\hat{E}_{\lambda} \cong \mbP^1$.
Recall that the intersection points $\hat{F}_{\lambda} \cap \hat{C}_{\lambda,\mu}$ are the solutions of the equations 
\[
x_1- \lambda x_0 = s_1 - \mu x_0^4 = s_0 + z^2 = (z + \beta x_0) s_0 + \gamma z^2 x_0 + \delta z x_0^2 + \mu x_0^3 = 0
\]
in $\mbP (1,1,1,2,4)$.
This shows that $\hat{C}_{\lambda,\mu_1} \cap \hat{F}$ is disjoint from $\hat{C}_{\lambda,\mu_2} \cap \hat{F}_{\lambda}$ for $\mu_1 \ne \mu_2$, and hence $\Bs \hat{\mcL}_{\lambda}$ is disjoint from $\hat{F}_{\lambda}$.

We have
\[
\hat{F}_{\lambda} \cong (y + z^2 + s_1 a_2 + a_6 = s_1 + z x_0 + c_4 = x_1 - \lambda x_0 = 0) \subset \mbP (1,1,6,3,4).
\]
Hence $\hat{F}_{\lambda} \cong \mbP (1,3) \cong \mbP^1$.
Recall that the intersection point $\hat{C}_{\lambda,\mu} \cap \hat{F}_{\lambda}$ is the solution of the equations
\[
x_1 - \lambda x_0 = s_1 - \mu x_0^4 = y + z^2 + \alpha x_0^6 = z + \varepsilon x_0^3 = 0
\]
in $\mbP (1,1,6,3,4)$.
Thus $\hat{C}_{\lambda,\mu_1} \cap F \ne \hat{C}_{\lambda,\mu_2} \cap F$ for $\mu_1 \ne \mu_2$ and $\Bs \hat{\mcL}$ is disjoint from $F$.
Hence (1) and (2) are proved.
\end{proof} 

As we will see below, the pair $(W, \frac{1}{4} \mcH_W)$ is strictly canonical and its canonical center is the $\frac{1}{3} (1,1,2)$ point lying on $F$.
Thus we need a further blowup of $W$. 
Let $W' \to W$ be the Kawamata blowup at the $\frac{1}{3} (1,1,2)$ point lying on $F$ and $\overline{W} \to W'$ be the Kawamata blowup at the $\frac{1}{2} (1,1,1)$ point lying on the exceptional divisor of $W' \to W$.
Let $p \colon \overline{W} \to W$ be the resulting birational morphism and we set $\bar{\psi} = \psi \circ p \colon \overline{W} \to Y$.
For a curve or a divisor $\Delta$ on $X$, $Y$ or $W$, we denote by $\overline{\Delta}$ its proper transform on $\overline{W}$. 
Then $\overline{W}$ is nonsingular along $\bar{\Gamma}$, a general member of the proper transform $\mcH_{\overline{W}}$ of $\mcH_W$ has multiplicity $1$ along $\bar{\Gamma}$ and the pair $(\overline{W}, \frac{1}{4} \mcH_{\overline{W}})$ is crepant over $(W, \frac{1}{4} \mcH_W)$. 
This shows that $(\overline{W}, \frac{1}{4} \mcH_{\overline{W}})$ is terminal (and $(W, \frac{1}{4} \mcH_W)$ is canonical).

\begin{Lem} \label{lem:verifcondW}
The birational morphism $\bar{\psi} \colon \overline{W} \to Y$ satisfies \emph{(4)} of \emph{Condition \ref{condinvbir}}
\end{Lem}

\begin{proof}
We verify conditions (a-f) in (4).
(a) is clearly satisfied and (e) is already verified.
Since the curves $\hat{C}_{\lambda,\mu}$ form a base point free pencil on $\hat{S}_{\lambda}$, a general $C_{\lambda,\mu}$ does not pass through the $\frac{1}{3} (1,1,2)$ point lying on $F$.
Hence $(\overline{F} \cdot \overline{C}_{\lambda,\mu}) = (F \cdot \hat{C}_{\lambda,\mu}) = 1$ and $\overline{C}_{\lambda,\mu}$ is disjoint from the $p$-exceptional divisors, which verify (d).
Moreover, we have $(-K_{\overline{W}} \cdot \overline{C}_{\lambda,\mu}) = (-K_W \cdot \hat{C}_{\lambda,\mu}) = 0$.
Since $W$ has only terminal singularity, we have $(-K_{\overline{W}} \cdot \overline{\Gamma}) \le (-K_W \cdot \hat{\Gamma}) < 0$.
This verifies (b).
Since $p$ is an isomorphism over $\hat{E}$, we have $(\overline{E} \cdot \overline{C}_{\lambda,\mu}) = 1$ and (c) is verified.

The exceptional divisor of $W' \to W$ is isomorphic to $\mbP (1,1,2)$ and  the lifting on $W'$ of sections $x_0,x_1$ restricts to the two coordinates of degree $1$ on the exceptional divisor $\mbP (1,1,2)$.
Moreover, the proper transform of $F$ on $W'$ avoids the point $(0\!:\!0\!:\!1) \in \mbP (1,1,2)$ which is the $\frac{1}{2} (1,1,1)$ point of $W'$.
This shows that $\overline{S}_{\lambda} \cap \overline{F}$ is the proper transform of $\hat{F}_{\lambda}$ and thus it is irreducible.
Clearly $\overline{S}_{\lambda} \cap \overline{E}$ is irreducible since it is isomorphic to $\hat{E}_{\lambda}$.
Thus, (f) is verified and the proof is completed.
\end{proof}

\begin{Lem} \label{lem:G19intnondeg}
For a general $\lambda \in \mbC$, the intersection form of $(x_0 = 0)|_{S_{\lambda}} = \Gamma + \Delta$ is nondegenerate.
\end{Lem}

\begin{proof}
Since $A|_{S_{\lambda}} \sim_{\mbQ} \Gamma + \Delta$, we have $(\Gamma + \Delta \cdot \Gamma) = \deg \Gamma = 1/4$ and $(\Gamma + \Delta \cdot \Delta) = \deg \Delta = 1/4$.
We also have $(\Gamma \cdot \Delta) = 1$.
Thus we get
\[
\begin{pmatrix}
(\Gamma^2) & (\Gamma \cdot \Delta) \\
(\Gamma \cdot \Delta) & (\Delta^2) 
\end{pmatrix}
=
\begin{pmatrix}
- \frac{3}{4}  & 1 \\
1 & - \frac{3}{4}
\end{pmatrix}.
\]
Therefore the intersection form $(x_0 = 0)|_{S_{\lambda}} = \Gamma + \Delta$ is nondegenerate.
\end{proof}

\begin{Thm} \label{thm:G19invbir}
If $\msp$ is a maximal center, then there is a birational involution $\tau \colon X \ratmap X$ which is a Sarkisov link centered at $\msp$. 
\end{Thm}

\begin{proof}
By Theorem \ref{thm:constinvbir}, it is enough to show that Condition \ref{condinvbir} is satisfied.
This follows from Lemmas \ref{lem:G19intnum}, \ref{lem:G19S}, \ref{lem:verifcondW}, \ref{lem:G19intnondeg} and Proposition \ref{prop:G19excl}.
\end{proof}

\section{The tables} \label{sec:tables}

\newlength{\myheight}
\setlength{\myheight}{0.65cm}

\subsection{Families $\mcG_i$} \label{sec:table1}

\newlength{\myheightt}
\setlength{\myheightt}{0.65cm}

The following table is the list of the family $\mcG_i$ with $i \in I_{F,cAx/2} \cup I_{F,cAx/4}$.

\begin{center}
\begin{tabular}{cc|cc}
\hline
\parbox[c][\myheightt][c]{0cm}{} No. & $X_{d_1,d_2} \subset \mbP (a_0,\dots,a_5)$ & No. & $X_{d_1,d_2} \subset \mbP (a_0,\dots,a_5)$ \\
\hline
\parbox[c][\myheightt][c]{0cm}{} 17 & $X_{6,8} \subset \mbP (1,1,1,3,4,5)$ & 49 & $X_{10,14} \subset \mbP (1,1,2,5,7,9)$ \\
\parbox[c][\myheightt][c]{0cm}{} 19 & $X_{6,8} \subset \mbP (1,1,2,3,4,4)$ & 50 & $X_{10,14} \subset \mbP (1,2,3,5,7,7)$  \\ 
\parbox[c][\myheightt][c]{0cm}{} 23 & $X_{6,10} \subset \mbP (1,1,2,3,5,5)$ & 55 & $X_{12,14} \subset \mbP (1,1,4,6,7,8)$ \\
\parbox[c][\myheightt][c]{0cm}{} 29 & $X_{8,10} \subset \mbP (1,1,2,4,5,6)$ & 69 & $X_{14,16} \subset \mbP (1,1,5,7,8,9)$ \\
\parbox[c][\myheightt][c]{0cm}{} 30 & $X_{8,10} \subset \mbP (1,1,3,4,5,5)$ & 74 & $X_{14,18} \subset \mbP (1,2,3,7,9,11)$ \\
\parbox[c][\myheightt][c]{0cm}{} 41 & $X_{10,12} \subset \mbP (1,1,3,5,6,7)$ & 77 & $X_{16,18} \subset \mbP (1,1,6,8,9,10)$ \\
\parbox[c][\myheightt][c]{0cm}{} 42 & $X_{10,12} \subset \mbP (1,1,4,5,6,6)$ & 82 & $X_{18,22} \subset \mbP (1,2,5,9,11,13)$
\end{tabular}
\end{center}

\subsection{Families $\mcG'_i$} \label{sec:table2}

In this subsection, we list the families $\mcG'_i$ with $i \in I_{F,cAx/2} \cup I_{F,cAx/4}$.
In each family, a standard defining equation is described.
The first column indicates the number and type of singular points of $X'$.
The second column indicates how to exclude the singular point as maximal center.
If ``$(B^3) \le 0$" and a divisor of the form ``$b B + e E$" are indicated in the second column, then the corresponding singular point is excluded in Section \ref{sec:singptsnonpos}.
If proposition number is given in the second column, then the corresponding point is excluded in the proposition.
The mark $\not\exists (m_1,m_2,m_3)$ or $\exists (m_1,m_2,m_2)$ indicates the condition on the (non)existence of a WCI curve of type $(m_1,m_2,m_3)$ passing through the corresponding point $\msp$, and the mark $(*h)$ or $(0 h)$, where $h$ is a monomial indicates the condition on the (non)existence of the monomial $h$ in the defining polynomial.

The third column indicates the existence of Sarkisov link centered at the corresponding singular point.
The mark ``Q.I." (resp.\ ``E.I.", resp.\ ``I.I.") means that there is a Sarkisov link $X' \ratmap X'$ which is a quadratic (resp.\ elliptic, resp.\ invisible) involution (see Sections \ref{sec:birinv} and \ref{sec:spinv}) and the mark ``Link to $X_{d_1,d_2} \in \mcG_i$" means that there is a Sarkisov link to $X_{d_1,d_2} \in \mcG_i$ (see Section \ref{sec:SLWHWCI}).

\begin{center}
\begin{flushleft}
No. 17: $X'_8 \subset \mbP (1,1,1,4,2)$, $A^3 = 1$ \nopagebreak \\ 
Eq: $w^3 x_0^2 + w^2 x_0 f_3 + w g_6 + h_8$
\end{flushleft} \nopagebreak
\begin{tabular}{|p{105pt}|p{135pt}|p{100pt}|}
\hline
\parbox[c][\myheight][c]{0cm}{} $\msp_4 = cAx/2$ & & Link to $X_{6,8} \in \mcG_{17}$ \\
\hline
\end{tabular}
\end{center} 

\begin{center}
\begin{flushleft}
No.~19: $X'_8 \subset \mbP (1,1,2,3,2)$, $A^3 = 2/3$ \nopagebreak \\ 
Eq: $w^2 y (y + f_2) + w g_6 + h_8$, $z^2 \in g_6$
\end{flushleft} \nopagebreak
\begin{tabular}{|p{105pt}|p{135pt}|p{100pt}|}
\hline
\parbox[c][\myheight][c]{0cm}{} $\msp_2 \msp_4 = 2 \times \frac{1}{2} (1,1,1)$ & ($\not\exists (1,1,2)$) & E.I. \\
\parbox[c][\myheight][c]{0cm}{} & ($\exists (1,1,2)$) & I.I. \\
\hline
\parbox[c][\myheight][c]{0cm}{} $\msp_3 = \frac{1}{3} (1,1,1)$ & & Q.I.\\
\hline
\parbox[c][\myheight][c]{0cm}{} $\msp_4 = cAx/2$ & & Link to $X_{6,8} \in \mcG_{19}$\\
\hline
\end{tabular}
\end{center}

\begin{center}
\begin{flushleft}
No.~23: $X'_{10} \subset \mbP (1,1,2,3,4)$, $A^3 = 5/12$ \nopagebreak \\ 
Eq: $w^2 x_0 (x_0 + f_1) + w g_6 + h_{10}$, $z^2 \in g_6$
\end{flushleft} \nopagebreak
\begin{tabular}{|p{105pt}|p{135pt}|p{100pt}|}
\hline
\parbox[c][\myheight][c]{0cm}{} $\msp_2 \msp_4 = \frac{1}{2} (1,1,1)$ & $B^3 < 0$, $T \in |B|$, ($\not\exists (1,1,4)$) & none \\
\parbox[c][\myheight][c]{0cm}{} & Proposition \ref{exclspptNo23}, ($\exists (1,1,4)$) & none \\
\hline
\parbox[c][\myheight][c]{0cm}{} $\msp_3 = \frac{1}{3} (1,1,2)$ & & Q.I.\\
\hline
\parbox[c][\myheight][c]{0cm}{} $\msp_4 = cAx/4$ & & Link to $X_{6,10} \in \mcG_{23}$ \\
\hline
\end{tabular}
\end{center}

\begin{center}
\begin{flushleft}
No.~29: $X'_{10} \subset \mbP (1,1,2,5,2)$, $A^3 = 1/2$ \nopagebreak \\ 
Eq: $w^3 y^2 + w^2 y f_4 + w g_8 + h_{10}$
\end{flushleft} \nopagebreak
\begin{tabular}{|p{105pt}|p{135pt}|p{100pt}|}
\hline
\parbox[c][\myheight][c]{0cm}{} $\msp_2 \msp_4 = 3 \times \frac{1}{2} (1,1,1)$ & $B^3 = 0$, $T \in |B|$ & none \\
\hline
\parbox[c][\myheight][c]{0cm}{} $\msp_4 = cAx/2$ & & Link to $X_{8,10} \in \mcG_{29}$ \\
\hline
\end{tabular}
\end{center}

\begin{center}
\begin{flushleft}
No.~30: $X'_{10} \subset \mbP (1,1,3,4,2)$, $A^3 = 5/12$ \nopagebreak \\ 
Eq: $w^2 y (y + f_3) + w g_8 + h_{10}$, $z^2 \in g_8$
\end{flushleft} \nopagebreak
\begin{tabular}{|p{105pt}|p{135pt}|p{100pt}|}
\hline
\parbox[c][\myheight][c]{0cm}{} $\msp_2 = \frac{1}{3} (1,1,2)$ & ($*y^2 z$) & Q.I. \\
\parbox[c][\myheight][c]{0cm}{} & Proposition \ref{prop:G'30-3sp}, ($0 y^2 z$) & none \\
\hline
\parbox[c][\myheight][c]{0cm}{} $\msp_3 = \frac{1}{4} (1,1,3)$ & & Q.I. \\
\hline
\parbox[c][\myheight][c]{0cm}{} $\msp_4 = cAx/2$ & & Link to $X_{8,10} \in \mcG_{30}$ \\
\hline
\end{tabular}
\end{center}

\begin{center}
\begin{flushleft}
No.~41: $X'_{12} \subset \mbP (1,1,3,6,2)$, $A^3 = 1/3$ \nopagebreak \\ 
Eq: $w^3 y^2 + w^2 y f_5 + w g_{10} + h_{12}$
\end{flushleft} \nopagebreak
\begin{tabular}{|p{105pt}|p{135pt}|p{100pt}|}
\hline
\parbox[c][\myheight][c]{0cm}{} $\msp_2 \msp_3 = 2 \times \frac{1}{3} (1,1,2)$ & & Q.I. \\
\hline
\parbox[c][\myheight][c]{0cm}{} $\msp_4 = cAx/2$ & & Link to $X_{10,12} \in \mcG_{41}$\\
\hline
\end{tabular}
\end{center}

\begin{center}
\begin{flushleft}
No.~42: $X'_{12} \subset \mbP (1,1,4,5,2)$, $A^3 = 3/10$ \nopagebreak \\ Eq: $w^2 y (y + f_4) + w g_{10} + h_{12}$, $z^2 \in g_{10}$
\end{flushleft} \nopagebreak
\begin{tabular}{|p{105pt}|p{135pt}|p{100pt}|}
\hline
\parbox[c][\myheight][c]{0cm}{} $\msp_2 \msp_4 = \frac{1}{2} (1,1,1)$ & $B^3 < 0$, $T \in |B|$ & none \\
\hline
\parbox[c][\myheight][c]{0cm}{} $\msp_3 = \frac{1}{5} (1,1,4)$ & & Q.I. \\
\hline
\parbox[c][\myheight][c]{0cm}{} $\msp_4 = cAx/2$ & & Link to $X_{10,12} \in \mcG_{42}$ \\
\hline
\end{tabular}
\end{center}

\begin{center}
\begin{flushleft}
No.~49: $X'_{14} \subset \mbP (1,1,2,7,4)$, $A^3 = 1/4$ \nopagebreak \\ 
Eq: $w^3 x_0^2 + w^2 x_0 f_5 + w g_{10} + h_{14}$
\end{flushleft} \nopagebreak
\begin{tabular}{|p{105pt}|p{135pt}|p{100pt}|}
\hline
\parbox[c][\myheight][c]{0cm}{} $\msp_2 \msp_4 = \frac{1}{2} (1,1,1)$ & $B^3 < 0$, $T \in |B|$ & none \\
\hline
\parbox[c][\myheight][c]{0cm}{} $\msp_4 = cAx/4$ & & Link to $X_{10,14} \in \mcG_{49}$ \\
\hline
\end{tabular}
\end{center}

\begin{center}
\begin{flushleft}
No.~50: $X'_{14} \subset \mbP (1,2,3,5,4)$, $A^3 = 7/60$ \nopagebreak \\ 
Eq: $w^2 z (z + f_3) + w g_{10} + h_{14}$, $t^2 \in g_{10}$
\end{flushleft} \nopagebreak
\begin{tabular}{|p{105pt}|p{135pt}|p{100pt}|}
\hline
\parbox[c][\myheight][c]{0cm}{} $\msp_1 \msp_4 = \frac{1}{2} (1,1,1)$ & $B^3 < 0$, $3 B + E$, ($\not\exists (1,3,4)$) & none \\
\parbox[c][\myheight][c]{0cm}{} & Proposition \ref{exclsppt2No50}, ($\exists (1,3,4)$) & none \\
\hline
\parbox[c][\myheight][c]{0cm}{} $\msp_2 = \frac{1}{3} (1,1,2)$ & $B^3 < 0$, $T \in |B|$, ($*z^3 t$) & none \\
\parbox[c][\myheight][c]{0cm}{} & Proposition \ref{exclsppt3No50}, ($0 z^3 t$) & none \\ 
\hline
\parbox[c][\myheight][c]{0cm}{} $\msp_3 = \frac{1}{5} (1,2,3)$ & & Q.I. \\
\hline
\parbox[c][\myheight][c]{0cm}{} $\msp_4 = cAx/4$ & & Link to $X_{10,14} \in \mcG_{50}$ \\
\hline
\end{tabular}
\end{center}

\begin{center}
\begin{flushleft}
No.~55: $X'_{14} \subset \mbP (1,1,4,7,2)$, $A^3 = 1/4$ \nopagebreak \\ 
Eq: $w^3 y^2 + w^2 y f_6 + w g_{12} + h_{14}$
\end{flushleft} \nopagebreak
\begin{tabular}{|p{105pt}|p{135pt}|p{100pt}|}
\hline
\parbox[c][\myheight][c]{0cm}{} $\msp_2 = \frac{1}{4} (1,1,3)$ & $B^3 > 0$, Proposition \ref{prop:G'55-4} & none \\
\hline
\parbox[c][\myheight][c]{0cm}{} $\msp_2 \msp_4 = \frac{1}{2} (1,1,1)$ & $B^3 < 0$, $T \in |B|$ & none \\
\hline
\parbox[c][\myheight][c]{0cm}{} $\msp_4 = cAx/2$ & & Link to $X_{12,14} \in \mcG_{55}$\\
\hline
\end{tabular}
\end{center}

\begin{center}
\begin{flushleft}
No.~69: $X'_{16} \subset \mbP (1,1,5,8,2)$, $A^3 = 1/5$ \nopagebreak \\ 
Eq: $w^3 y^2 + w^2 y f_7 + w g_{14} + h_{16}$
\end{flushleft} \nopagebreak
\begin{tabular}{|p{105pt}|p{135pt}|p{100pt}|}
\hline
\parbox[c][\myheight][c]{0cm}{} $\msp_2 = \frac{1}{5} (1,2,3)$ & $B^3 > 0$, Proposition \ref{prop:G'69-5} & none \\
\hline
\parbox[c][\myheight][c]{0cm}{} $\msp_4 = cAx/2$ & & Link to $X_{14,16} \in \mcG_{69}$ \\
\hline
\end{tabular}
\end{center}

\begin{center}
\begin{flushleft}
No.~74: $X'_{18} \subset \mbP (1,2,3,9,4)$, $A^3 = 1/12$ \nopagebreak \\ 
Eq: $w^3 z^2 + w^2 z f_7 + w g_{14} + h_{18}$
\end{flushleft} \nopagebreak
\begin{tabular}{|p{105pt}|p{135pt}|p{100pt}|}
\hline
\parbox[c][\myheight][c]{0cm}{} $\msp_1 \msp_4 = \frac{1}{2} (1,1,1)$ & $B^3 < 0$, $3B+E$ & none \\
\hline
\parbox[c][\myheight][c]{0cm}{} $\msp_2 \msp_3 = 2 \times \frac{1}{3} (1,1,2)$ & $B^3 < 0$, $T \in |B|$ & none \\
\hline
\parbox[c][\myheight][c]{0cm}{} $\msp_4 = cAx/4$ & & Link to $X_{14,18} \in \mcG_{74}$ \\
\hline
\end{tabular}
\end{center}

\begin{center}
\begin{flushleft}
No.~77: $X'_{18} \subset \mbP (1,1,6,9,2)$, $A^3 = 1/6$ \nopagebreak \\
Eq: $w^3 y^2 + w^2 y f_8 + w g_{16} + h_{18}$
\end{flushleft} \nopagebreak
\begin{tabular}{|p{105pt}|p{135pt}|p{100pt}|}
\hline
\parbox[c][\myheight][c]{0cm}{} $\msp_2 \msp_3 = \frac{1}{3} (1,1,2)$ & $B^3 = 0$, $T \in |B|$ & none \\
\hline
\parbox[c][\myheight][c]{0cm}{} $\msp_2 \msp_4 = \frac{1}{2} (1,1,1)$ & $B^3 < 0$, $T \in |B|$ & none \\
\hline
\parbox[c][\myheight][c]{0cm}{} $\msp_4 = cAx/2$ & & Link to $X_{16,18} \in \mcG_{77}$ \\
\hline
\end{tabular}
\end{center}

\begin{center}
\begin{flushleft}
No.~82: $X'_{22} \subset \mbP (1,2,5,11,4)$, $A^3 = 1/20$ \nopagebreak \\ 
Eq: $w^3 z^2 + w^2 z f_9 + w g_{18} + h_{22}$ \nopagebreak \\
\end{flushleft} \nopagebreak
\begin{tabular}{|p{105pt}|p{135pt}|p{100pt}|}
\hline
\parbox[c][\myheight][c]{0cm}{} $\msp_1 \msp_4 = \frac{1}{2} (1,1,1)$ & $B^3 < 0$, $5B + 2 E$ & none \\
\hline
\parbox[c][\myheight][c]{0cm}{} $\msp_2 = \frac{1}{5} (1,1,4)$ & $B^3 = 0$, $T \in |B|$ & none \\
\hline
\parbox[c][\myheight][c]{0cm}{} $\msp_4 = cAx/4$ & & Link to $X_{18,22} \in \mcG_{82}$ \\
\hline
\end{tabular}
\end{center}


\begin{thebibliography}{99}


\bibitem{AZ13}
H.~Ahmadinezhad and F.~Zucconi,
Circle of Sarkisov links on a Fano threefold,
arXiv:1304.6009, 2013, to appear in Proc. Edinburgh Math. Soc.

\bibitem{AZ14}
H.~Ahmadinezhad and F.~Zucconi,
Birational rigidity of Fano $3$-folds and Mori dream spaces,
arXiv:1407.3724, 2014.

\bibitem{BZ10}
G.~Brown and F.~Zucconi,
Graded rings of rank $2$ Sarkisov links,
Nagoya Math. J. {\bf 197} (2010), 1--44.

\bibitem{CP}
I.~Cheltsov and J.~Park,
Birationally rigid Fano threefold hypersurfaces,
arXiv:1309.0903v4, 2013, to appear in Mem. Amer. Math. Soc.

\bibitem{Co95}
A.~Corti, 
Factoring birational maps of threefolds after Sarkisov,
J. Algebraic Geom. {\bf 4} (1995), no.~2, 223--254.

\bibitem{Co00}
A.~Corti,
Singularities of linear systems and $3$-fold birational geometry, {\it Explicit birational geometry of $3$-folds},
London Math. Soc. Lecture Note Ser., {\bf 281}, Cambridge Univ. Press, Cambridge, 2000.

\bibitem{CM}
A.~Corti and M.~Mella,
Birational geometry of terminal quartic $3$-folds,
Amer. J. Math. {\bf 126} (2004), no.~4, 739--761.

\bibitem{CPR}
A.~Corti, A.~V.~Pukhlikhov and M.~Reid,
Fano $3$-fold hypersurfaces,
{\it Explicit birational geometry of $3$-folds}, London Math. Soc. Lecture Note Ser., {\bf 281}, Cambridge Univ. Press, Cambridge, 2000.

\bibitem{Ha99}
T.~Hayakawa,
Blowing ups of $3$-dimensional terminal singularities, Publ. Res. Inst. Math. Sci. {\bf 35} (1999), 515--570.

\bibitem{IF00}
A.~R.~Iano-Fletcher,
Working with weighted complete intersections,
{\it Explicit birational geometry of $3$-folds},
London Math. Soc. Lecture Note Ser., {\bf 281},
Cambridge Univ. Press, Cambridge, 2000.

\bibitem{IM71}
V.~A.~Iskovskikh and Yu.~I.~Manin,
Three-dimensional quartics and counterexamples to the L\"{u}roth problem, Math. USSR. Sb. {\bf 151} (1971), 141--166.

\bibitem{Ka05}
M.~Kawakita, 
Three-fold divisorial contractions to singularities of higher indices,
Duke Math. J. {\bf 130} (2005), no.~1, 57--126.

\bibitem{Kaw96}
Y.~Kawamata,
``Divisorial contractions to $3$-dimensional terminal quotient singularities'' in Higher-Dimensional Complex Varieties (Trento, Italy, 1994), de Gruyter, Berlin, 1996, 241--246.

\bibitem{KMM}
S.~Keel, K.~Matsuki and J.~McKernan,
Log abundance theorem for threefolds,
Duke Math. J. {\bf 75} (1994), no.~1, 99--119.

\bibitem{KM}
J.~Koll\'{a}r and S.~Mori,
{\it Birational geometry of algebraic varieties},
Cambridge Tracts in Mathematics, {\bf 134},
Cambridge University Press, Cambridge, 1998, viii+254pp.

\bibitem{KSC}
J.~Koll\'{a}, K.~E.~Smith and A.~Corti,
{\it Rational and nearly rational varoeties},
Cambridge Studies in Advanced Mathematics, {\bf 92},
Cambridge University Press, Cambridge, 2004, vi+235pp. 

\bibitem{Okada}
T.~Okada,
Birational Mori fiber structures of $\mathbb{Q}$-Fano $3$-fold weighted complete intersection,
Proc. Lond. Math. Soc. (3) {\bf 109} (2014), no.~6, 1549--1600.
\end{thebibliography}
\end{document}